\documentclass[10pt,british]{article}
\usepackage{amsfonts}
\usepackage{latexsym}
\usepackage{amsmath}
\usepackage{amssymb}
\usepackage{amssymb}
\usepackage{amsthm}
\usepackage{mathrsfs}
\usepackage[pagebackref,colorlinks=true]{hyperref}

\newtheorem{thrm}{Theorem}[section]

\newtheorem{lemma}[thrm]{Lemma}

\newtheorem{prop}[thrm]{Proposition}

\newtheorem{cor}[thrm]{Corollary}

\newtheorem{dfn}[thrm]{Definition}

\newtheorem{rmrk}[thrm]{Remark}

\newtheorem{conv}[thrm]{Convention}
\newtheorem{exam}[thrm]{Example}

\newcommand{\newsection}{    % Numeration of eqs. is automatic
\setcounter{equation}{0}\section}
\def\appendix#1{\addtocounter{section}{1}\setcounter{equation}{0}
\renewcommand{\thesection}{\Alph{section}}
\section*{Appendix \thesection\protect\indent \parbox[t]{11.15cm}{#1}}
\addcontentsline{toc}{section}{Appendix \thesection\ \ \ #1}}

\newcommand{\be}{\begin{eqnarray}}
\newcommand{\ee}{\end{eqnarray}}
\newcommand{\bea}{\begin{eqnarray}}
\newcommand{\eea}{\end{eqnarray}}
\newcommand{\ba}{\begin{array}}
\newcommand{\ea}{\end{array}}

\newenvironment{frcseries}{\fontfamily{frc}\selectfont}{}

\def\d{\delta}

\def\T{\Theta}

\def\ts {{\theta^{su}}}
\def\tn {{\nu}}
\def\tsp{{\theta^{spin}}}
\def\sb {{\nabla}}
\def\LC{{\nabla^g}}
\def\ps{{\Psi^+}}
\def\psi{{\Psi^{a+}}}

\def\sp{{\Psi^-}}
\def\ph{{\Phi}}
\def\om{{\Omega}}

\begin{document}
%\begin{titlepage}
\begin{center}
%\today
\vspace*{-1.0cm}
%\hfill hep-th/yymmnnn \\
%\hfill UB-ECM-PF-06-43 \\

%\vspace{2.0cm} {\Large \bf Vanishing theorems  for Hermitian
%manifolds
% satisfying $\omega^\ell\wedge \partial\bar\partial \omega^k=0$
%} \\[.2cm]

%\vspace{1.0cm}
{\Large \bf Torsion and curvature of %Instanton  and parallel torsion on 
ACYT %almost  Calabi-Yau  with torsion 
 and AHKT 8-manifold  %and generalized Ricci solitons
% satisfying $\omega^\ell\wedge \partial\bar\partial \omega^k=0$
} %\\[.2cm]

\vspace{0.5  cm}
 {\large Stefan Ivanov${}^1$}% and  N. Stanchev$^2$}

\vspace{0.5cm}

${}^1$ University of Sofia, Faculty of Mathematics and
Informatics,\\ blvd. James Bourchier 5, 1164, Sofia, Bulgaria
\\and  Institute of Mathematics and Informatics,
Bulgarian Academy of Sciences\\
email: ivanovsp@fmi.uni-sofia.bg

%\vspace{0.5cm}
%${}^2$ University of Sofia, Faculty of Mathematics and
%Informatics,\\ blvd. James Bourchier 5, 1164, Sofia, Bulgaria\\

%\vspace{0.5cm}

\end{center}

\vskip 0.5 cm
\begin{abstract}
We observe that on a compact almost Hermitian  8 manifold  with totally skew-symmetric  Nijenhuis tensor parallel with respect to the unique almost hermitian  connection with torsion three form if the three Ricci tensors of this connection vanish then the torsion 3-form is closed. Consequently, on a compact complex 8-manifold if the three Ricci tensors of the Strominger-Bismut connection vanish then its torsion is closed, i.e. the space is  pluriclosed. We also deduce that a compact ACYT 8-manifold with paralle Nijenhuis tensor with respect to the torsion connection  has closed torsion if and only if the trace of the exterior derivative of the torsion vanishes.
It is shown that a compact ACYT 8-manifold  wich Nijenhuis tensor is parallel with respect to the torsion connection is a Ricci flat  $SU(4)$  instanton exactly when the torsion 3-form is parallel with respect to the torsion and to the Levi-Civita connections simultaneously. 

We consider an almost hyperhermitian 8-manifold   with an $Sp(2)$ structure admitting linear connection with torsion three form preserving the $Sp(2)$ structure call it AHKT manifold. The exact conditions an almost hyperhermitian 8-manifold to be an AHKT are presented % in terms of the codifferential of the wedge square of the K\"ahler forms  and the three Lee 1-forms associated with the almost hyperhermitian structure and  give an explicit expression of  the torsion three form.  
and it is shown that an AHKT 8-manifold is HKT if and only if  the three Lee forms coincide.

\medskip


\medskip

\end{abstract}

\vskip 0.5cm
\noindent{\bf Acknowledgements:} \vskip 0.1cm
%We would like to thank Jeffrey Streets and Ilka Agricola for the extremely useful remarks,  comments and suggestions.

The research  is partially supported by Contract KP-06-H72-1/05.12.2023 with the National Science Fund of Bulgaria,  by Contract 80-10-40 / 01.04.2026    with the Sofia University "St.Kl.Ohridski". %and the National Science Fund of Bulgaria, National Scientific Program ``VIHREN", Project KP-06-DV-7. The research of N.S.  is partially  financed by the European Union-Next Generation EU, through the National Recovery and Resilience Plan of the Republic of Bulgaria, project %1SUMMIT BG-RRP-2.004-0008-C01.

\vskip 0.5cm

%Statements and Declarations: not applicable

%\end{titlepage}

\tableofcontents

\setcounter{section}{0}
\setcounter{subsection}{0}

%%%%%%%%%%%%%%%%%%%%%%%%%%%%%%%%%%%%%%%%%%%%%%%%%%%%%%%%%%%%%%%%%%%%%%%%%%

%%%%%%%%%%%%%%%%%%%%%%%%%%%%%%%%%%%%%%%%%%%%%%%%%%%%%%%%%%%%%%%%%%%%%%%%%%

\newsection{Introduction}
Riemannian manifolds with metric connections having totally skew-symmetric torsion and special holonomy received a lot of interest in mathematics and theoretical physics mainly from supersymmetric string theories and supergravity.  The main reason comes from the Hull-Strominger system which describes the supersymmetric background in heterotic string theories \cite{Str,Hull}. The number of preserved supersymmetries depends on the number of parallel spinors with respect to a metric connection $\sb$ with totally skew-symmetric torsion $T$. %  which is related to the Levi-Civita
%connection $\LC$ by $$ \nabla = \LC + \frac{1}{2}H. $$  the three form
The torsion 3-form $T$ is identified with the 3-form field strength in these theories. 
The presence of a $\nabla-$parallel spinor leads to restriction of the
holonomy group $Hol(\nabla)$ of the torsion connection $\nabla$.
Namely, $Hol(\nabla)$ has to be contained in $SU(n), dim=2n$, 
%\cite{Str,GMW,IP1,IP2,Car,BB,BBE,GIP}, 
$Sp(n), dim=4n$ %\cite{} 
the exceptional group $G_2,
dim=7$ %\cite{FI,GKMW,FI1}, 
the Lie group $Spin(7), dim=8$.
%\cite{GKMW,I1}. 
A detailed analysis of the possible geometries is
carried out in \cite{GMW}.

Complex non-K\"ahler geometries appear
in string compactifications  and are studied intensively for a long time
\cite{Str,GMW,GKMW,GMPW,GPap,BB,BBE}. 
Hermitian manifolds have widespread applications in both physics and differential geometry in connection with solutions to the Hull-Strominger system. %see \cite{LY,yau,yau1,FIUV,XS,PPZ,PPZ3,PPZ4,FHP1,GRST,Ph} and references therein.
On a Hermitian manifold,  there exists   a  unique
connection which preserves  the Hermitian structure and has
totally skew-symmetric torsion tensor. Its existence and explicit expression first appeared in Strominger's seminal paper \cite{Str} in 1986 in connection with the heterotic supersymmetric string background, where he called it the H-connection. Three years later, Bismut formally discussed
and used this connection in his local index theorem paper \cite{bismut}, which leads to the name Bismut
connection in literature.
We call this connection the Strominger-Bismut connection. Note that the connection also appeared implicitly earlier (\cite{Yano}) and in some literature it was also called the KT connection (K\"ahler with torsion).  When the holonomy of the Strominger-Bismut connection is contained in $SU(n)$ one has the notion of Calabi-Yau manifold with torsion (CYT)  and if the holonomy of the Stominger-Bismut connection is contained in $Sp(n)$ one has the notion of HKT spaces. HKT manifolds were introduced by Howe and Papapdopoulos in \cite{hkt} (for a nice mathematical interpretation see \cite{GP}). CYT and HKT  manifolds are of great interest in string theories and in mathematics since they appear in the Hull-Strominger system. Some types of Hermitian non-complex manifolds have also been invented  in the
string theory due to the mirror symmetry and T-duality
\cite{KST,Car,Car1,KLMS,KML,HLM}. %From the point of view of physics  one has to consider compact  CYT space with an exact Lee form $\theta=df$, i.e. the interior multiplication of the torsion $T$  with the K\"ahler two form $F$  should be an exact 1-form \cite{Str}.

%and were intensively developed by Verbitski and his colaborators \cite{}.%An interesting geometric explanation of the curvature of the Strominger-Bismut connection appears in the framework of generalized Ricci flow
%developed by Garcia-Fernandez and Streets \cite{GFS}. 
Vanishing theorems for
the Dolbeault cohomology on compact Hermitian non-K\"ahler manifold were found in terms of the Strominger-Bismut connection in  \cite{AI,IP2,IP1,YZZ}.

If the torsion 3-form $T$ of $\nabla$ is closed, $dT=0$ which, in Hermitian manifold,  is equivalent to the condition $\partial\bar\partial F=0$, where $F$ is the fundamental 2-form,  the Hermitian metric g is called SKT (strong K\"ahler with torsion)  or pluriclosed.
The SKT (pluriclosed) metrics have found many applications in
both physics %, see eg \cite{hull, howe, Str, hethor,
%hethor1,sethi,IP1} 
and geometry. %, see eg \cite{IP2, salamon,  yau, yau1, FIUV,FT1, streets,Str1} and references therein. 
For example in type II string theories, %the three form$T$ is identified with the 3-form field strength and 
it is required
by construction the 3-form $T$ to be closed, $dT=0$ (see e.g. \cite{GKMW,GMW}). In a series of papers Streets and Tian
\cite{streets,ST,ST1} introduced a Hermitian Ricci flow under which the
pluriclosed or equivalently strong KT condition is preserved, show the relationship of pluriclosed flow to type II string backgrounds and find that pluriclosed flow preserves N=(2,2) supersymmetry (generalized K\"ahler geometry). More generally, the geometry of a torsion connection with closed torsion form appears in the framework of the generalized Ricci flow and the generalized (gradient) Ricci solitons developed by Garcia-Fernandez and Streets \cite{GFS} (see the references therein). Compact pluriclosed (strong) CYT spaces are connected with the pluriclosed flow and the equations of motion in compactifications of type II supergravity due to the  work of Garcia-Fernandez, Jordan and Streets \cite{GFJS}.  It is shown in \cite[Proposition~8.14]{GFS}, (see also \cite[Proposition~2.6]{GFJS} that any compact pluriclosed CYT space is automatically a steady generalized gradient Ricci soliton  \cite[Proposition~8.14]{GFS}, (see also \cite[Proposition~2.6]{GFJS}). 
%Moreover, global existence and convergence of pluriclosed flow to a Stromonger-Bismut-flat metric on complex manifolds admitting Strominger-Bismut-flat metric is also shown in \cite{GFJS}.

Generalizations of the pluriclosed condition $\partial\bar\partial F=0$ on 2n dimensional Hermitian manifolds in the form $
F^\ell\wedge
\partial\bar\partial F^k=0~, ~~~1\leq k+\ell\leq n-1~$  has been investigated in \cite{FU,Pop,FWW,IP3} etc.

Almost Hermitian manifolds with totally skew-symmetric Nijenhuis
tensor arise as target spaces of a class of (2,0)-supersymmetric
two-dimensional sigma models \cite{Pap}. For the consistency of
the theory, the Nijenhuis tensor has to be parallel with respect
to the torsion connection with holonomy contained in $SU(n)$. The known models are group manifolds as well as the nearly K\"ahler spaces. %A 6-dimensional nilmanifold, which is neither complex nor nearly Kaehler, is presented in \cite{II}.
In general, an almost Hermitian manifold does not admit a metric connection preserving the almost Hermitian structure and having  totally skew-symmetric torsion. It is shown in \cite[Theorem~10.1]{FI} that this is equivalent to the condition that the Nijenhuis tensor %of type (0,3), $N(X,Y,Z)=g(N(X,Y),Z)$
 is totally skew-symmetric, i.e. this is  the class $G_1$ in Gray-Hervella classification \cite{GrH}. In this case,  the connection is unique. This class contains the Hermitian manifolds with Strominger-Bismut connection as well as the nearly K\"ahler spaces. In the nearly K\"ahler case the torsion connection $\sb$ is the characteristic connection considered by Gray, \cite{gray} and the torsion $T$ and the Nijenhuis tensor are $\sb-$parallel, $\sb T=\sb N = 0$ (see e.g. \cite{Kir,BM}). If the holonomy of the torsion connection $\sb$ is contained in $SU(n)$, that is,  the Ricci 2-form $\rho$ of $\sb$  representing the first Chern class is identically zero,  then we have the notion of \emph{Almost Calabi-Yau manifold with torsion (briefly ACYT)} and if $Hol(\sb)\subset Sp(n)$ we have the notion of \emph{Almost HKT manifold (briefly AHKT)}. ACYT spaces in dimension six are investigated in \cite{IS1,IU} and strong  ACYT spaces are considered in \cite{KenStr}.

 %As it was pointed out in \cite{Str}, the existence of a parallel spinor with respect to a metric connection $\sb$ with torsion 3-form, in dimension eight,  leads to the restriction that its holonomy group has to be contained in $ Sp(2)\subset SU(4)\subset Spin(7)$. %This means one has to consider an $SU(3)$-structure, i.e. an almost Hermitian manifold $(M^6,g,J)$ with topologically trivial canonical bundle trivialized by a (3,0) with respect to $J$ form  $\Psi=\Psi^++\sqrt{-1}\Psi^-$\ endowed with a metric connection preserving the $SU(3)$ structure with totally skew-symmetric torsion. %,i.%.e. an ACYT space.
 One of the main purposes is to investigate the  properties of the torsion connection on an ACYT and an AHKT  8-manifold. We do this using the inclusions  $ Sp(2)\subset SU(4)\subset Spin(7)$ and  applaying  the known results for $Spin(7)$ manifolds. Our first main result is 
\begin{thrm}\label{main0}
Let $(M,g,J)$ be a compact 8-dimensional almost Hermitian manifold  with totally skew symmetric Nijenhuis tensor $N$   parallel with respect to the unique  connection $\sb$ with totally skew-symmetric torsion $T$ preserving the almost hermitian structure.    

If the three Ricci tensors of $\sb$ vanish then the torsion is closed  with constant norm and $\sb$-parallel Lee form
\[\sb N=\rho=Ric=\kappa=0 \quad imply \quad dT=d||T||^2=\sb\theta=0.\]
\end{thrm}
On a compact 8-dimensional complex space we have
\begin{cor}\label{main0comlex}
Let $(M,g,J)$ be a compact complex 8-manifold. If the three Ricci tensors of the Strominger-Bismut connection vanish then its torsion is closed, i.e. the complex space is  pluriclosed, the norm of the torsion is constant and the Lee form is parallel with respect to Strominger-Bismut connection.
\end{cor} 
The Hull-Strominger system is consisted of the Killing spinor equations and the anomaly cancellation condition.  The Killing spinor equations lead to consider compact  CYT space with an exact Lee form. The anomaly cancellation condition expresses the exterior derivative of the torsion 3-form in terms of the difference of the first Pontrjagin 4-form of an instanton connection on a vector bundle and a connection on the tangent bundle. It is shown in \cite{I1} that solutions to the Hull-Strominger system solve the heterotic equations of motion in dimensions 5,6,7 and 8 if and only if the connection on the tangent bundle in the anomally cancellation condition is an instanton (see \cite{MS} and \cite{XS} for extension of this result to all dimensions). The $SU(n)$-instanton condition (Yang-Mills connection)  means that the curvature 2-form $R$ of a connection $\sb$ belongs to the Lie algebra $\frak{su}(n)$ and it is highly non-linear PDE. %and the physically relevant connection is the Hull connection which is defined as the metric connection with torsion $-T$. 
Recall  that if there exists an SU(n)-instanton on a holomorphic vector bundle on a complex 2n-manifold then  it is unique due to  the non-K\"ahler version of the Donaldson-Uhlenbeck-Yau theorem  proved  by Li-Yau in \cite{LiY}.
 
 It is also well known that on a compact $SU(4)$ manifold    a connection $\sb$ on the tangent bundle with curvature 2-form $R$ is an absolute minimum of the Yang-Mills functional with torsion $YM_T=\int_M||R||^2vol-\int_Mtr(R\wedge R)\wedge\ph$, where $\ph=\frac12F\wedge F$ is the square of the K\"ahler form.

 Another purpose of the paper is to develop   the $SU(4)$ instanton  condition of the torsion  connection on an ACYT 8-manifold which will fix the connection on the tangent bundle in the anomaly cancellation condition as an absolute minimum of the Yang-Mills functional with torsion.
 
% We recall that on an ACYT 8-manifold the $SU(4)$ instanton condition for $\sb$  reads  %the curvature $R$ of $\sb$  satisfies 
%\[R\in\frak{su}(4)\otimes\frak{su}(4).\]
 It is known from  \cite[Lemma~3.4]{I} that the curvature $R$ of a metric connection $\sb$ with  torsion 3-form $T$ is symmetric in exchanging the first and the second pairs,  $R\in S^2\Lambda^2$, 
 if and only if the covariant derivative of the torsion with respect to the torsion connection  is a 4-form, $\sb T\in \Lambda^4$ which is equivalent  the torsion $T$ to be is a Killing 3-form. If  the holonomy group of a metric connection with torsion 3-form lies in $\frak{su}(4)$ the condition $\sb T\in \Lambda^4$ implies that the curvature is an $SU(4)$ instanton. In particular, if the torsion is parallel with respect to the torsion connection then its  curvature is automatically an instanton. 
Note that Hermitian and  almost Hermitian manifold of type $G_1$ with $\sb-$parallel torsion 3-form  are investigated in \cite{Scho,CMS,ZZ2,BPT,BFG,BFGV,Pap1,MorS}  and a large number of examples are known. 

Another point of interest in the paper is to investigate when the $SU(4)$ instanton condition implies the torsion is parallel?

We first observe in Proposition~\ref{jriccom1} that on a compact ACYT 8-manifold with $\sb$-parallel Nijenhuis tensor and $SU(4)$ instanton torsion connection  the Lee form is $\sb$-parallel and the Ricci tensor of $\sb$ is symmetric. This helps to show
%In this direction we have
%On an ACYT -8 manifold the $SU(4)$ instanton means that the curvature $R$ of $\sb$  satisfies 
%\[R\in\frak{su}(4)\otimes\frak{su}(4).\]
% Since the $SU(4)$ instanton condition imply $\kappa=0$, we obtain from Theorem~\ref{main0} 
\begin{thrm}\label{main0inst}
Let $(M,g,J)$ be a compact 8-dimensional ACYT space with $\sb$-parallel Nijenhuis tensor. 

If $\sb$ is Ricci flat $SU(4)$-instanton,
\[\sb N=0,\quad R\in\frak{su}(4)\otimes\frak{su}(4),\quad Ric=0,\]
 then the torsion is parallel with respect to the torsion connection $\sb$ and to the Levi-Civita connection  $\LC$ simultaneously, $\sb T=\LC T=0$. 

In particular, the torsion 3-form  is closed and co-slosed $dT=\delta T=0$,  the exterior derivative $dF$ of the K\"ahler 2-form is $\sb$-parallel, $\sb dF=0$,  the space satisfies the Riemannian first Bianchi identity \ref{RB}.
\end{thrm} 
We deduce  that if the torsion on an ACYT 8-manifold is closed then the Nijenhuis tensor is $\sb$-parallel.

For the converse, we observe that on a compact ACYT 8-manifold with $\sb$-parallel Nijenhuis tensor the torsion is closed if and only if the interior multiplication of $dT$ with the K\"ahler 2-form $F$ vanishes, $F\lrcorner dT=0$ (see Theorem~\ref{closT}).

We  show in Theorem~\ref{clossT}  that on a compact ACYT 8-manifold with closed torsion the torsion connection $\sb$ is an $SU(4)$ instanton if and only if the torsion is $\sb$-parallel, $\sb T=0$. Consequently, %we get%In the complex case, we obtain 
\begin{cor}\label{cclosTc}
Let $(M,g,J,\Psi)$ be a  compact  pluriclosed CYT 8-manifold. 

Then the torsion is parallel with respect to the Strominger-Bismut connection if and only if the Strominger-Bismut connection is an $SU(4)$ instanton.
\end{cor}

%We notice that any compact ACYT 8-manifold with closed torsion 3-form is a steady generalized gradient Ricci soliton and this condition is equivalent to a certain vector field to be parallel with respect to the torsion connection. We also find out that this vector field preserves the $SU(4)$ structure, thus extending the above-mentioned result in \cite{GFJS} for compact pluriclosed CYT space to compact ACYT space with closed torsion in dimension eight.

We consider an almost hyperhermitian 8-manifold   with an $Sp(2)$ structure $(M,g,J_a), a=1,2,3$ admitting linear connection preserving the $Sp(2)$ structure and with totally skew symmetric torsion $T$ and call it AHKT manifold since if the $Sp(2)$ structure is integrable it is known as HKT manifold. 

We  investigate necessary and sufficient  conditions an almost hyperhermitian 8-manifold $(M,g,J_a, a=1,2,3)$ to be AHKT. We present in Theorem~\ref{spthm1}  exact conditions an almost hyperhermitian 8-manifold to be an AHKT in terms of the codifferential of the wedge square of the K\"ahler forms $\d(F^a\wedge F^a), a=1,2,3$ and the three Lee 1-forms associated with the almost hyperhermitian structure and  give an explicit expression of  the torsion three form. 
In particular, we show that an AHKT 8-manifold is HKT exactly when the three Lee forms coincide. We also prove
\begin{thrm}\label{ahktmain}
Let $(M,g,J_a,F^a), a=1,2,3)$ be a compact AHKT 8-manifold. 

The next conditions are equivalent:
\begin{itemize}
\item[a).] The torsion is closed, $dT=0$;
\vspace{-2.4mm}
\item[b).] The interior multiplication of the 4-form $dT$ with the K\"ahler 2 forms vanish,
$$F^a\lrcorner dT=0,\quad a=1,2,3.$$
\end{itemize}
\end{thrm}
%\begin{cor}\label{corhkt}
%A compact HKT 8-maifold has closed torsion, $dT=0$ if and only if $F^a\lrcorner dT=0$.
%\end{cor}
In the HKT case, $F^a\lrcorner dT=F^b\lrcorner dT=F^c\lrcorner dT$, a fact observed by Jeffry Streets \cite{JSTR}, and we obtain 
\begin{cor}\label{closThkt}
On a compact HKT 8-manifold the next conditions are equivalent:
\begin{itemize}
\item[a).] The torsion is closed, $dT=0$.
\vspace{-2.4mm}
\item[b).] The two form  $F^a\lrcorner dT=0,\quad   a\in\{1,2,3\}$.
\vspace{-2.4mm}
\item[c).] The Ricci tensor is given by  $Ric=-\sb\theta$.
\end{itemize}
\end{cor}
\begin{rmrk}
We remark that HKT spaces satisfying the condition $F^a\lrcorner dT=0$ are called almost strong in \cite{IP2}. Corollary~\ref{closThkt}  shows that a compact almost strong HKT space in dimension 8 is strong and  answers positively in dimension 8 to a question posed to the author by Jeffry Streets \cite{JSTR} while  Theorem~\ref{ahktmain} generalizes this result to  compact almost strong AHKT 8-manifolds.
\end{rmrk}
We construct in Example~\ref{cstr} a compact AHKT 8-manifold with closed and $\sb$-parallel torsion with two balanced almost complex structures and one not balanced but complex structure which is not HKT.  This example also fits with the very recent paper \cite{GeM} where the semi-integrable almost hyperhermitian structures are investigated. We show in Theorem~\ref{balcom} that a compact AHKT 8-manifold with one balanced complex structure is necessarilly a compact balanced HKT manifold with holonomy of the Obata connection  contained in $SL(2,H)$ and therefore having a holomorphically trivial canonical bundle due to \cite{Ver1,Ver2} (c.f. also \cite[Theorem~2.2]{IvP}). We also %investigate orthogonal semi-integrable AHKT 8-manifold and 
show that a compact orthogonal semi-integrable AHKT 8-manifold with one co-closed Nijenhuis tensor has $\sb$-parallel Lee form and $\sb$-parallel Nijenhuis tensors (Theoem~\ref{semin}) which yields that a compact orthogonal semi-integrable AHKT 8-manifold has closed torsion 3-form exactly when it is Ricci flat (Corollary~\ref{seminc}).

%}
%Note that Hermitian, almost Hermitian, almost hyperhermitian manifolds of type $G_1$ and   HKT spaces with $\sb-$parallel torsion 3-form  are investigated in  \cite{AFS,Scho}, recently in \cite{CMS,ZZ2,BPT,BFG,BFGV,Pap1,MorS}  and a large number of examples are given there. 
%More generally, Riemannian manifolds equipped with metric connection $\sb$ with $\sb-$parallel  torsion 3-form are investigated in \cite{AFF,AFer,Ch1,Ch2}.
%It is conjectured very recently that if a complete Hermitian manifold has Strominger-Bismut parallel curvature and torsion and zero all Ricci tensors then it has to be Strominger-Bismut flat, \cite[Conjecture~1.8]{NZ1}.  A special case with the additional conditions that the space is K\"ahler-like and CYT is proved in \cite[Proposition~6.8]{NZ1}.
 It is worth mentioning  \cite[Theorem~4.1]{AFF} which states that an irreducible complete and simply connected Riemannian manifold of dimension bigger or equal to $5$ with $\sb-$parallel and closed torsion 3-form, $\sb T=dT=0$ (which is equivalent to $\sb T=\sigma^T=0$ due to \eqref{dh}  below), is a simple compact Lie group or its dual non-compact symmetric space with biinvariant metric, and, in particular, the torsion connection is the flat Cartan connection.

\begin{rmrk}
 A combination of Theorem~\ref{main0inst} with \cite[Theorem~4.1]{AFF} shows that an irreducible complete and simply connected ACYT 8-manifold with $\sb$-parallel Nijenhuis tensor and Ricci flat $SU(4)$ instanton torsion connection should be a simple compact Lie group with the flat  Cartan  connection, for example the compact Lie group $SU(3)$.%Moreover, due to Theorem~\ref{cmainsu3} and \cite[Theorem~4.1]{AFF},   irreducible simply connected compact 6-dimensional  ACYT space satisfying conditions \eqref{cs2l2}  does not exist since there are no simple compact Lie groups of dimension six.
\end{rmrk}

\begin{conv}
Everywhere in the paper, we will make no difference between tensors and the corresponding forms via the metric as well as we will use Einstein summation conventions, i.e. repeated Latin indices are summed over.
\end{conv}

\section{Preliminaries}
In this section, we recall some known curvature properties of a metric connection with totally skew-symmetric torsion on Riemannian manifold as well as
the notions and existence of a metric linear connection preserving a given Almost Hermitian structure and having skew-symmetric torsion from \cite{I,FI,IS}.
\subsection{Metric connection with totally skew-symmetric  torsion and its curvature}%satisfying the  Riemannian first Bianchi identity}
Let $(M,g,\sb)$ be a Riemannian manifold with a metric connection $\sb$ and $T(X,Y)=\sb_XY-\sb_YX-[X,Y]$ be its torsion. The torsion of type (0,3) is denoted by the same letter and it is defined by $T(X,Y,Z)=g(T(X,Y),Z)$. The torsion is totally skew-symmetric if $T(X,Y,Z)=-T(X,Z,Y)$.

On a Riemannian manifold $(M,g)$ of dimension $n$ any metric connection $\sb$ with totally skew-symmetric torsion $T$ is connected with the Levi-Civita connection $\sb^g$ of the metric $g$ by
\begin{equation}\label{tsym}
\sb^g=\sb- \frac12T \quad leading\quad to \quad \LC T=\sb T+\frac12\sigma^T,
\end{equation}
where the 4-form $\sigma^T$, introduced in \cite{FI}, is defined by
 \begin{equation}\label{sigma}
 \sigma ^T(X,Y,Z,V)=\frac12\sum_{j=1}^n(e_j\lrcorner T)\wedge(e_j\lrcorner T)(X,Y,Z,V), 
 \end{equation} 
$(e_a\lrcorner T)(X,Y)=T(e_a,X,Y)$ is the interior multiplication and $\{e_1,\dots,e_n\}$ is an orthonormal  basis.

The properties of the 4-form $\sigma^T$ are studied in detail in \cite{AFF} where it is shown that $\sigma^T$ measures the `degeneracy' of the 3-form $T$.

For a $p$-form $\alpha$ we denote by $d^{\sb}\alpha$ the (p+1)-form which is the totally skew-symmetric part of $\sb\alpha$. 

The exterior derivative $dT$ has the following  expression (see e.g. \cite{I,IP2,FI})
\begin{equation}\label{dh}
\begin{split}
dT(X,Y,Z,V)=d^{\sb}T(X,Y,Z,V) +2\sigma^T(X,Y,Z,V), \quad where\\
d^{\sb}T(X,Y,Z,V)=(\nabla_XT)(Y,Z,V)+(\nabla_YT)(Z,X,V)+(\nabla_ZT)(X,Y,V)-(\nabla_VT)(X,Y,Z).
 \end{split}
 \end{equation}
 We  recall the next useful identity  observed in \cite[Proposition~3.1]{IS} 
 \begin{equation}\label{sigmatt}
 T(e_i,e_j,e_k)\sigma^T(e_i,e_j,e_k,X)=0,
 \end{equation}
  For the curvature of $\sb$ we use the convention $ R(X,Y)Z=[\nabla_X,\nabla_Y]Z -
 \nabla_{[X,Y]}Z$ and $ R(X,Y,Z,V)=g(R(X,Y)Z,V)$. It has the well-known properties
 \begin{equation}\label{r1}
 R(X,Y,Z,V)=-R(Y,X,Z,V)=-R(X,Y,V,Z).
 \end{equation}
   The first Bianchi identity for $\nabla$ can be written in the  form (see e.g. \cite{I,IP2,FI})
 \begin{equation}\label{1bi}
 \begin{split}
 R(X,Y,Z,V)+ R(Y,Z,X,V)+ R(Z,X,Y,V)\\%=(\nabla_XH)(Y,Z,V)+(\nabla_YH)(Z,X,V)+(\nabla_ZH)(X,Y,V)+\sigma^T(X,Y,Z,V)\\
 =dT(X,Y,Z,V)-\sigma^T(X,Y,Z,V)+(\nabla_VT)(X,Y,Z).
 \end{split}
 \end{equation}
It is proved in \cite[p.307]{FI} that the curvature of  a metric connection $\sb$ with totally skew-symmetric torsion $T$  satisfies also the  identity
 \begin{equation*}%\label{gen}
 \begin{split}
 R(X,Y,Z,V)+ R(Y,Z,X,V)+ R(Z,X,Y,V)-R(V,X,Y,Z)-R(V,Y,Z,X)-R(V,Z,X,Y)\\
 =\frac32dT(X,Y,Z,V)-\sigma^T(X,Y,Z,V),
 \end{split}
 \end{equation*}
 which combined with \eqref{1bi} yields \cite{IS}
 \begin{equation}\label{1bi1}
 \begin{split}
R(V,X,Y,Z)+R(V,Y,Z,X)+R(V,Z,X,Y)= -\frac12dT(X,Y,Z,V)+(\nabla_VT)(X,Y,Z).%+R(X,Y,Z,V)+ R(Y,Z,X,V)+ R(Z,X,Y,V)\\%=(\nabla_XH)(Y,Z,V)+(\nabla_YH)(Z,X,V)+(\nabla_ZH)(X,Y,V)+\sigma^T(X,Y,Z,V)\\
%\\ =dT(X,Y,Z,V)-\sigma^T(X,Y,Z,V)+(\nabla_VT)(X,Y,Z)-\frac32dT(X,Y,Z,V)+\sigma^T(X,Y,Z,V)
 \end{split}
 \end{equation}
\begin{dfn} We say that the curvature $R$ satisfies the Riemannian first Bianchi identity if
\begin{equation}\label{RB}
R(X,Y,Z,V)+R(Y,Z,X,V)+R(Z,X,Y,V)=0.
\end{equation}
\end{dfn}
 It is well known algebraic fact that \eqref{r1} and \eqref{RB} imply $R\in S^2\Lambda^2$, i.e.
 \begin{equation}\label{r4}
 R(X,Y,Z,V)=R(Z,V,X,Y).
 \end{equation}
 \begin{rmrk}
 Note that, in general, \eqref{r1} and \eqref{r4} do not imply \eqref{RB}.
 \end{rmrk}
We have the next recent result from \cite{IS}
\begin{thrm} \cite[Theorem~1.2]{IS}\label{tFBI}
A metric connection $\sb$ with torsion 3-form $T$ satisfies the Riemannian first Bianchi identity exactly when %the next identities hold
$%\begin{equation*}%\label{FBT}
3 dT=-6\nabla T=2\sigma^T.
$%\end{equation*}

In this case, the torsion 3-form $T$ is parallel with respect to a metric connection with torsion 3-form $\frac13T$ \cite{AF} and therefore has a constant norm, $||T||^2=const.$
\end{thrm}
It is known due to \cite[Lemma~3.4]{I} that a metric connection $\sb$ with totally skew-symmetric torsion $T$ has curvature $R\in S^2\Lambda^2$, i.e. it satisfies \eqref{r4}  if and only if the covariant derivative of the torsion with respect to the torsion connection  is a 4-form  i.e. the torsion $T$ is a Killing 3-form. Killing forms on Riemannian manifold, introduced by Yano in \cite{Yan},  are natural generalization of a Killing vector fields and are enough investigated in mathematics and physics (see e.g. \cite{USem} and references therein).
\begin{lemma}\cite[Lemma~3.4]{I} The next equivalences hold for a metric connection with torsion 3-form
\begin{equation}\label{4form}
(\sb_XT)(Y,Z,V)=-(\sb_YT)(X,Z,V) \Longleftrightarrow R(X,Y,Z,V)=R(Z,V,X,Y) ) \Longleftrightarrow dT=4\LC T.
\end{equation}
\end{lemma}
%\begin{equation}\label{fourf}
%(\sb_XT)(Y,Z,V)=-(\sb_YT)(X,Z,V).
%\end{equation}

 The   Ricci tensors and scalar curvatures of the connections $\LC$ and $\sb$ are related by \cite[Section~2]{FI}, (see also \cite [Prop. 3.18]{GFS})
\begin{equation}\label{rics}
\begin{split}
Ric^g(X,Y)=Ric(X,Y)+\frac12 (\delta T)(X,Y)+T^2(X,Y), \quad T^2(X,Y)=\frac14\sum_{i=1}^n(g(T(X,e_i),T(Y,e_i);\\
Scal^g=Scal+\frac14||T||^2,\qquad Ric(X,Y)-Ric(Y,X)=-(\delta T)(X,Y),
\end{split}
\end{equation}
where $\delta=(-1)^{np+n+1}*d*$ is the co-differential acting on $p$-forms and $*$ is the Hodge star operator satisfying $*^2=(-1)^{p(n-p)}$.

%Following \cite{GFS}, we denote 
%$T^2_{ij}=T_{iab}T_{jab}:=\sum_{a,b=1}^nT_{iab}T_{jab}.$
%Then the first equality in \eqref{rics} reads
%$$Ric^g=Ric+\frac12\delta T+\frac14T^2.$$

One  has the general identities for $\alpha\in\Lambda^1$ and $\beta\in \Lambda^k$ 
\begin{equation}\label{1star}
\begin{split}
*(\alpha\lrcorner\beta)=(-1)^{k+1}(\alpha\wedge*\beta);\qquad (\alpha\lrcorner\beta)=(-1)^{n(k+1)}*(\alpha\wedge*\beta);\\
*(\alpha\lrcorner*\beta)=(-1)^{n(k+1)+1}(\alpha\wedge\beta);\qquad (\alpha\lrcorner*\beta)=(-1)^{k}*(\alpha\wedge\beta).
\end{split}
\end{equation}
Denote by $\delta^{\sb}T$ the negative  trace of $\sb T$, $\delta^{\sb}T(X,Y)=-(\sb_{e_i}T)(e_i,X,Y)$.
 
 It follows  from \eqref{dh} and \eqref{tsym} that % following equivalences  valid for any metric connection with torsion 3-form
 \begin{equation}\label{inst6}
 d^{\sb}T=0 \Longleftrightarrow dT=2\sigma^T;\qquad \delta^{\sb}T=\delta T.
 \end{equation}

%One  has the general identities for $\alpha\in\Lambda^1$ and $\beta\in \Lambda^k$
%\begin{equation}\label{1star}
%\begin{split}
%*(\alpha\lrcorner\beta)=(-1)^{k+1}(\alpha\wedge*\beta);\qquad (\alpha\lrcorner\beta)=(-1)^{n(k+1)}*(\alpha\wedge*\beta);\\
%*(\alpha\lrcorner*\beta)=(-1)^{n(k+1)+1}(\alpha\wedge\beta);\qquad (\alpha\lrcorner*\beta)=(-1)^{k}*(\alpha\wedge\beta);
%\end{split}
%\end{equation}

\subsection{Torsion connection preserving an almost Hermitian structure} Not any almost Hermitian manifold admits a metric linear connection preserving a given almost Hermitian structure and having totally skew-symmetric torsion. The  conditions for the existence of such a connection were given by Strominger in \cite{Str} (see also \cite{gauduchon,FI}).

We recall the notions and existence of a metric linear connection preserving a given almost Hermitian structure and having totally skew-symmetric torsion from \cite{FI}.

Let $(M,g,J)$ be an almost Hermitian 2n-dimensional manifold with Riemannian
metric $g$ and an almost complex structure $J$ satisfying
$$g(JX,JY)=g(X,Y).$$
The Nijenhuis tensor $N$, the K\"ahler
form $F$ and the Lee form $\theta^{su}$ are defined by
\begin{equation}\label{cy1}
N=[J.,J.]-[.,.]-J[J.,.]-[.,J.], \quad F=g(.,J.), \quad \theta^{su}(.)=\delta F(J.)=F\lrcorner dF,
\end{equation}
respectively, where $\delta=-*d*$ is the co-differential on an even dimensional manifold.% and * is the Hodge star operator.

\noindent
Consequently, the  1-form $J\theta^{su}$ defined by $J\theta^{su}(X)=-\theta^{su}(JX)$ is co-closed due to $\delta(J\theta^{su})=-\delta^2 F=0$.

Further, for a p-form $\gamma$ we adopt the notation $J\gamma(X_1,\dots,X_p)=(-1)^p\gamma(JX_1,\dots,JX_p)$

The Nijenhuis tensor of type (0,3)  defined by $N(X,Y,Z)=g(N(X,Y),Z)$ is of type (3,0)+(0,3) with respect to the almost complex structure $J$.

%On an almost Hermitian manifold, there not always exists a metric connection with totally skew-symmetric torsion preserving the almost Hermitian structure.
 It is shown in \cite[Theorem~10.1]{FI} that there
exists a unique linear connection $\sb$ preserving an almost hermitian
structure $(g,J)$, $\sb g=\sb J=0$ and having totally skew-symmetric torsion $T$  if and only if
the Nijenhuis tensor $N$ is a 3-form.  This is precisely the class $G_1$ in the Gray-Hervella classification of almost Hermitian manifolds \cite{GrH} and includes Hermitian manifolds, $N=0$ as well as nearly K\"ahler spaces, $(\LC_XJ)X=0$. We call this connection \emph{the torsion connection}.

The torsion $T$ of the torsion connection is
determined by \cite[Theorem~10.1]{FI}
\begin{equation}\label{cy2}
 T=JdF+N=-dF(J.,J.,J.)+N=-dF^+(J.,J.,J.)+\frac{1}{4}N,
\end{equation}
where $dF^+$ denotes the (1,2)+(2,1)-part of $dF$ with respect to $J$ and satisfies the equality
\begin{equation}\label{12}
\begin{split}
dF^+(X,Y,Z)=dF^+(JXJ,Y,Z)+dF^+(JX,Y,JZ)+dF^+(X,JY,JZ).
\end{split}
\end{equation}
It follows from \eqref{12} that
  \begin{equation}\label{normH}
dF^+(Je_k,Je_i,e_j)dF^+(e_k,e_i,e_j)=\frac13||dF^+||^2.
\end{equation}
The (3,0)+(0,3)-part $dF^-$ obeys the identity
\begin{equation}\label{123}
dF^-(JX,Y,Z)=dF^-(X,JY,Z)=dF^-(X,Y,JZ)
\end{equation}
and  is determined
completely by the Nijenhuis tensor \cite{gauduchon}.  If $N$ is a three form  then (see e.g.\cite{gauduchon,FI})
\begin{equation}\label{acy}
dF^-(X,Y,Z)=\frac{3}{4}N(JX,Y,Z)=\frac34JN(X,Y,Z).
\end{equation}
The Lee form $\theta^{su}:=\delta F\circ J$ of the $G_1$- manifold is given in terms of the  torsion $T$ as
 \begin{equation}\label{liff}
 \begin{split}
\theta^{su}(X)=-\frac12T(JX,e_i,Je_i)=\frac12g(T(JX),F)=\frac12(F\lrcorner
T(JX)),\\
\theta^{su}_i={1\over 2} T_{kjl} F_{jl}F_{ki}=\frac16T_{kjl}\ph_{kjli}, \qquad J\theta^{su}_i=\frac12T_{ijk}F_{jk}=-\theta^{su}_sF_{si}.
\end{split}
\end{equation}
where $\{e_1,\dots,e_n,e_{n+1}=Je_1,\dots,e_{2n}=Je_n\}$ is an orthonormal basis and $\ph=\frac12F\wedge F$.

When the almost complex structure is integrable, $N=0$, the torsion connection was defined and studied by Strominger \cite{Str} in connection with the heterotic string background and was used by Bismut to prove a local index theorem for
the Dolbeault operator on Hermitian non-K\"ahler manifold \cite{bismut}. This formula was applied also in string theory (see e.g. \cite{BBE}). In the Hermitian case, the torsion connection is also
known as the Strominger-Bismut connection and has attained a lot of consideration both in mathematics and physics (see the Introduction).

In the nearly K\"ahler case,  the torsion connection $\sb$ is the characteristic connection
considered by Gray, \cite{gray} and the torsion $T$ and the Nijenhuis tensor are $\sb-$parallel, $\sb T=\sb N = 0$ (see e.g. \cite{Kir,BM}).

In addition to the Ricci tensor of the torsion connection on a $G_1$ almost Hermitian manifold the curvature $R$ of the torsion connection $\sb$ has two more Ricci type two forms defined by
$$\rho(X,Y)=\frac12R(X,Y,e_i,Je_i), \qquad \kappa(X,Y)=\frac12R(e_i,Je_i,X,Y).
$$
The first Ricci two form $\rho$  represents the first Chern class and the second Ricci two form $\kappa$ is an (1,1) form with respect to the almost complex structure $J$.

 The holonomy group of $\sb$ is contained in $U(n)$. The reduction of the holonomy group $\sb$ to a subgroup of $SU(n)$ can be expressed in terms of the first Ricci 2-form $\rho$, namely $\rho=0$.
 
We remark that the formula  \cite[(3.16)]{IP2} holds in the general case of a $G_1$-manifold (see also \cite{FI}) , namelly, the next formula holds true
\begin{equation}\label{ssc}
\rho(X,Y)=Ric(X,JY)+(\sb_X\ts)(JY)+\frac14dT(X,Y,e_i,Je_i).
\end{equation}
Thus, on a $G_1$ manifold the restricted holonomy group of the torsion connection is contained in $SU(n)$ if and only if the next  condition holds \cite[Theorem~10.5]{FI}
\begin{equation}\label{su}
Ric(X,Y)+(\sb_X\theta^{su})(Y)-\frac14dT(X,JY,e_i,Je_e)=0.
\end{equation}
In this case the two form $\kappa\in\frak{su}(n)$ since $F\lrcorner\kappa=F\lrcorner\rho=0$.

We calculate from \eqref{dh} using \eqref{cy2} and \eqref{normH}
\begin{multline}\label{su2}
dT(e_j,Je_j,e_i,Je_i)=-8(\sb_{e_i}\theta^{su})(e_i)+8|\theta^{su}|^2-4T(e_i,e_j,e_k)T(Je_i,Je_j,e_k)\\
=8\delta \theta^{su}+8||\theta^{su}||^2-\frac43||dF^+||^2+\frac14||N||^2=8\delta \theta^{su}+8|\theta^{su}|^2-\frac43||T||^2+\frac13||N||^2.
\end{multline}
\begin{rmrk}
It follows from \eqref{su2} that if the structure is balanced, $\theta=0$, and the torsion is closed, $dT=0$ then $16|dF|^2=3|N|^2$. In the complex case, $N=0$,  this confirms the old result first established in \cite{AI}  that  a balanced Hermitian manifold with closed torsion must be K\"ahler. In the almost complex case, $N\not=0$, this is not true.  The first example of compact non-K\"ahler almost complex manifold with closed torsion and vanishing Lee form is given in \cite{KenStr} in dimension 6. Two more examples of compact balanced ACYT 8-manifold with closed torsion are the ACYT 8-structures $(J_2,g)$ and $(J_3,g)$ described in Example~\ref{cstr} which are parts of an orthogonal semi-integrable AHKT 8-structure.
\end{rmrk}
The trace in \eqref{su} together with \eqref{su2} and \eqref{rics} yield the following formulas for the scalar curvature of the torsion connection and the Riemannian scalar curvature on an ACYT space 
\begin{equation}\label{scal2}
\begin{split}
Scal=3\delta\ts+2||\ts||^2-\frac13||dF^+||^2+\frac1{16}||N||^2;\\
Scal^g=3\delta\ts+2||\ts||^2-\frac1{12}||dF^+||^2+\frac5{64}||N||^2.
\end{split}
\end{equation}

\section{Sp(2), SU(4) and Spin(7) structures} 
We set up the  notions.

\centerline{\Large\bf{Sp(2)-structures}}

\medskip

 Let $(M,g,J_1,J_2,J_3)$ be an almost hyper-Hermitian 8-manifold with a Riemannian
metric $g$ and an almost hypercomplex structure $J_1,J_2,J_3$, i.e. the almost complex ctructures satisfy the algebraic identities of imaginary quaternions, $J_1J_2=-J_2J_1=J_3, \quad J_1^2=J_2^2=J_3^2=-1,\quad g(J_a,J_a)=g(,), a=1,2,3.$

Hence, an $Sp(2)$-structure  is determined by the three 2-forms $F^a(..,)=g(,J_a)$. To be more explicit, we may choose a local
orthonormal frame  $e_0,\dots,e_7$,  identifying it  with the dual
basis via the metric. Write $e_{i_1 i_2\dots i_p}$ for the
monomial $e_{i_1} \wedge e_{i_2} \wedge \dots \wedge e_{i_p}$. An
$Sp(2)$-structure can be locally described 
\begin{equation}\label{AAsp}
\begin{split}
F^1 =-e_{01} - e_{23}-e_{45}-e_{67}, \\
F^2 =-e_{02} + e_{13}-e_{46}+e_{57},\\
F^3 =-e_{03} - e_{12}-e_{47}-e_{56}.
\end{split}
\end{equation}
The subgroup of $SO(8)$ fixing the forms $F^1,F^2,F^3$  simultaneously is $Sp(2)$.
The three 2- forms $F^1,F^2,F^3$  determine the metric completely. The Lie algebra of $Sp(2)$
is denoted $\frak{sp}(2)$.

Note that the 2-forms $F^2$ and $F^3$ are of type (2,0)+(0,2) with respect to the almost complex structure $J_1$. Hence, the complex 4-form $\Psi=\frac12(F^2+\sqrt{-1}F^3)\wedge(F^2+\sqrt{-1}F^3)$ is a non-degenerate (4,0)-form with respect to $J_1$.

\centerline{\Large\bf{SU(4)-structures}}

\medskip

Let $(M,g,J)$ be an almost Hermitian 8-manifold with a Riemannian
metric $g$ and an almost complex structure $J$, i.e. $(g,J)$
define an $U(4)$-structure.

An $SU(4)$-structure is determined by an additional  non-degenerate
(4,0)-form $\Psi=\Psi^++\sqrt{-1}\Psi^-$ of constant norm, or equivalently by a
non-trivial spinor.  An
$SU(4)$-structure can be  described locally by
\begin{equation}\label{AA}
\begin{split}
F =-e_{01} - e_{23}-e_{45}-e_{67}, \quad F^4=24vol.\\
\Psi=(e_0+\sqrt{-1}e_1)\wedge (e_2+\sqrt{-1}e_3)\wedge  (e_4+\sqrt{-1}e_5)\wedge  (e_6+\sqrt{-1}e_7),\\
\Psi^+ = e_{0246} - e_{0257} -e_{0347} -e_{0356} +e_{1357} - e_{1346} - e_{1256} -e_{1247}=*\ps, \\
  \Psi^- =e_{0247} + e_{0256} +e_{0346} -e_{0357} -e_{1356} - e_{1347} - e_{1257} +e_{1246}=*\sp.
\end{split}
\end{equation}
These forms are subject to the compatibility relations
\begin{equation*}%\label{comp}
F\wedge\Psi^{\pm}=0, \qquad \ps\wedge\sp=0, \qquad 3\Psi^+\wedge\Psi^+=3\sp\wedge\sp=2F^4;
\end{equation*}
The subgroup of $SO(8)$ fixing the forms $F$ and  $\Psi$ simultaneously is $SU(4)$.
The two forms $F$ and $\Psi$ determine the metric completely. The Lie algebra of $SU(4)$
is denoted $\frak{su}(4)$.

We list below those of the Gray-Hervella \cite{GrH} classes which we will use later.
\begin{description}
%\begin{enumerate}
\item[$W_1:$] The class of Nearly K\"ahler (weak holonomy) manifold defined by
$dF$ to be (3,0)+(0,3)-form.
\vspace{-2.4mm}
%\item[$\tau \in W_2$] The class of almost K\"ahler manifolds, $dF=0$.
\item[$W_3:$] The class of balanced Hermitian manifold determined by the
conditions $N=\theta=0$.
\vspace{-2.4mm}
%\item[$\tau \in W_4$] The class of locally conformally K\"ahler spaces characterized by
%$dF=\theta\wedge F$.
\item[$ W_1\oplus W_3\oplus W_4:$] The class called by Gray-Hervella $G_1$-manifolds
determined by the condition that the Nijenhuis tensor is totally skew-symmetric.
This is the precise class which we are interested in.
\vspace{-2.4mm}
\item[$ W_1\oplus W_3:$] The class of balanced (semi-K\"ahler)  $G_1$-manifolds
defined by the condition that the Nijenhuis tensor is totally skew-symmetric and the Lee form vanishes, $N(X,Y,Z)=-N(X,Z,Y),\quad \theta=0$.
\vspace{-2.4mm}
\item[$ W_3\oplus W_4:$] The class of Hermitian manifolds, $N=0$.
%\end{enumerate}
\end{description}
%If all five components are zero then we have a Ricci-flat K\"ahler (Calabi-Yau) 4-fold.

Let $(M,g,J,\Psi)$ be an $SU(4)$ manifold. Letting the four form $\Phi=\frac12F^2=\frac12F\wedge F$ we have the relations
\begin{equation}\label{star}
\begin{split}
*\Phi=\ph,\qquad
*\Psi^+=\Psi^+,\quad *\Psi^-=\Psi^-\\
*(\alpha\wedge F)=\alpha\lrcorner*F=-J\alpha\wedge\ph, \quad \alpha\in \Lambda^1,\quad
*(\alpha\wedge\Phi)=\alpha\lrcorner \ph, \quad \alpha\in\Lambda^1,\\
*(\alpha\wedge\Psi^+)=\alpha\lrcorner\Psi^+ \quad \alpha\in \Lambda^1,\qquad
*(\alpha\wedge\Psi^-)=\alpha\lrcorner\Psi^- \quad \alpha\in \Lambda^1,\\
\beta\wedge\Phi=*(\beta\lrcorner\Phi), \quad \beta\wedge\ps=*(\beta\lrcorner\ps), \quad 
\beta\wedge\sp=*(\beta\lrcorner\sp), \quad \beta\in \Lambda^2.
\end{split}
\end{equation}
We extend the action of the almost complex structure on the exterior algebra, $J\beta=(-1)^p\beta(J.,\dots,J.)$ for a p-form $\beta\in\Lambda^p$. We recall that the Lie algebra 
$$\frak{su}(4)=\{\beta\in\Lambda^2 |J\beta=\beta, \beta\lrcorner F=0\}=\{\beta\in\Lambda^2 |\beta\lrcorner F=\beta\lrcorner \ps=\beta\lrcorner \sp=0  \}=\{\beta\in\Lambda^2 |\beta\lrcorner\Phi=-\beta \}
$$
Writing
$$
F=\frac12F_{ij}e_{ij},\quad \Psi^+=\frac1{24}\Psi^+_{ijkl}e_{ijkl}, \quad  \Psi^-=\frac1{24}\Psi^-_{ijkl}e_{ijkl},\quad  \Phi=\frac1{24}\Phi_{ijkl}e_{ijkl}
$$
we have the obvious identites % (c.f. \cite{BV})

\begin{equation}\label{idensu}
\begin{split}
F_{ip}F_{pj}=-\delta_{ij},\quad \Phi_{jslm}=F_{js}F_{lm}+F_{sl}F_{jm}+F_{lj}F_{sm},\quad 
\Psi^{\pm}_{ijpq}F_{pq}=\Phi_{ijkl}\Psi^{\pm}_{mjkl}=0,\\
 \Psi^{\pm}_{ijks}F_{sl}=\pm\Psi^{\mp}_{ijkl}, \quad
\Psi^+_{iabc}\Psi^-_{jabc}=24F_{ij}, \quad \Psi^{\pm}_{iabc}\Psi^{\pm}_{jabc}=24\delta_{ij};\\
\ps_{jslm}\ps_{lmab}=4\Phi_{jsab}-4F_{js}F_{ab}+4\delta_{ja}\delta_{sb}-4\delta_{jb}\delta_{sa}\\
=\sp_{jslm}\sp_{lmab}=4F_{sa}F_{jb}-4F_{ja}F_{sb}+4\delta_{ja}\delta_{sb}-4\delta_{jb}\delta_{sa}\\
\ps_{jslm}\sp_{lmab}%=4\Phi_{jsac}F_{cb}+4F_{js}\delta_{ab}+4\delta_{ja}F_{sb}-4F_{jb}\delta_{sa}
=4\delta_{ja}F_{sb}-4F_{jb}\delta_{sa}+4F_{ja}\delta_{sb}-4\delta_{jb}F_{sa}\\
\Phi_{ijab}\ps_{abkl}=2\ps_{ijkl}, \quad \Phi_{ijab}\sp_{abkl}=2\sp_{ijkl},\quad 
\Phi_{ijkl}\Phi_{klqr}=4F_{ij}F_{qr}-2\delta_{jq}\delta_{ir}+2\delta_{iq}\delta_{jr}.%\\
%\\\\\\\Psi^+_{kls}\Psi^-_{ijs}=\delta_{kj}F_{li}+\delta_{li}F_{kj}-\delta_{ki}F_{lj}-\delta_{lj}F_{ki},\\\Psi^+_{kls}\Psi^+_{ijs}=F_{kj}F_{li}-\delta_{li}\delta_{kj}-F_{ki}F_{lj}+\delta_{lj}\delta_{ki}\\ \Phi_{ijkl}F_{kl}=4F_{ij},\quad \Phi_{ijkl}\Psi^+_{klp}=2\Psi^+_{ijp},\quad  \Phi_{ijkl}\Psi^-_{klp}=2\Psi^-_{ijp},\\ \Phi_{ijkl}\Psi^+_{lqp}=-F_{ij}\Psi^-_{pqk}-F_{jk}\sp_{pqi}-F_{ki}\sp_{pqj},\\\Phi_{ijkl}\Psi^-_{lqp}=F_{ij}\Psi^+_{pqk}+F_{jk}\ps_{pqi}+F_{ki}\ps_{pqj},\\\Phi_{ijkl}\Phi_{rjkl}=12\delta_{ir}, \quad  \Phi_{ijkl}\Phi_{klqr}=2F_{ij}F_{qr}-2\delta_{jq}\delta_{ir}+2\delta_{iq}\delta_{jr},\\
\\\Phi_{ijkl}\Phi_{lpqr}=-F_{ij}F_{pq}\delta_{kr}-F_{ij}F_{qr}\delta_{kp}-F_{ij}F_{rp}\delta_{kq}\\
-F_{jk}F_{pq}\delta_{ir}-F_{jk}F_{qr}\delta_{ip}-F_{jk}F_{rp}\delta_{iq}
-F_{ki}F_{pq}\delta_{jr}-F_{ki}F_{qr}\delta_{jp}-F_{ki}F_{rp}\delta_{jq}.
 \\\ps_{ijks}\ps_{abcs} = \delta_{ia} \delta_{jb} \delta_{kc} +\delta_{ib} \delta_{jc} \delta_{ka} +\delta_{ic} \delta_{ja} \delta_{kb}
- \delta_{ia} \delta_{jc} \delta_{kb} -\delta_{ib} \delta_{ja} \delta_{kc} -\delta_{ic} \delta_{jb} \delta_{ka} \\
-F_{ba}\sp_{ijkc}-F_{cb}\sp_{ijka}-F_{ac}\sp_{ijkb}-F_{ji}\sp_{abck}-F_{kj}\sp_{abci}-F_{ik}\sp_{abcj}\\
+\delta_{ia}[\ps_{jkbc}+F_{kb}F_{jc}+F_{bj}F_{kc}]+\delta_{ja}[\ps_{kibc}+F_{ib}F_{kc}+F_{bk}F_{ic}]\\+\delta_{ka}[\ps_{ijbc}+F_{jb}F_{ic}+F_{bi}F_{jc}]
+ \delta_{ib}[\ps_{jkca}+F_{kc}F_{ja}+F_{cj}F_{ka}]\\+\delta_{jb}[\ps_{kica}+F_{ic}F_{ka}+F_{ck}F_{ia}]+\delta_{kb}[\ps_{ijca}+F_{jc}F_{ia}+F_{ci}F_{ja}]\\
+ \delta_{ic}[\ps_{jkab}+F_{ka}F_{jb}+F_{aj}F_{kb}]+\delta_{jc}[\ps_{kiab}+F_{ia}F_{kb}+F_{ak}F_{ib}]\\+\delta_{kc}[\ps_{ijab}+F_{ja}F_{ib}+F_{ai}F_{jb}].
\end{split}
\end{equation}

\centerline{\Large\bf{Spin(7)-structures}}

\medskip

We briefly recall the notion of a $Spin(7)$--structure. Consider
${\mathbb R}^8$ endowed with an orientation and its standard inner
product. Consider the 4-form $\Omega$ on ${\mathbb R}^8$ given by
\begin{eqnarray}\label{s1}
\om &=&e_{0123} +e_{0145} +e_{0167}+e_{2345} +e_{2367} + e_{4567} +e_{0246}
 \\ \nonumber&\phantom{=}&
 - e_{0257} -e_{0347} -e_{0356} +e_{1357} - e_{1346} - e_{1256} -e_{1247}.\nonumber
\end{eqnarray}
%where $e_{ijkl}$ denotes the 4-form $e_i\wedge e_j\wedge e_k\wedge e_l$.
The 4-form  $\om$ is self-dual, $*\om=\om$, and the 8-form $\om\wedge\om$ coincides with
 14 times the volume form of ${\mathbb R}^8$. The subgroup of $GL(8,\mathbb R)$ which
fixes $\om$ is isomorphic to the double covering $Spin(7)$ of
$SO(7)$ \cite{Br}. Moreover, $Spin(7)$ is a compact
simply-connected Lie group of dimension 21 \cite{Br}. The Lie algebra of $Spin(7)$ is
denoted by $spin(7)$ and it is isomorphic to the 2-forms satisfying  linear equations, namely
 $spin(7)\cong  \{\alpha \in
\Lambda^2(M)|*(\alpha\wedge\om)=-\alpha\}$. %We note here the sign difference with \cite{Br}.

The 4-form
$\om$ corresponds to a real spinor $\om$ and therefore,
$Spin(7)$ can be identified as the isotropy group of a non-trivial
real spinor.

A \emph{$Spin(7)$--structure} on an 8-manifold $M$ is by definition
a reduction of the structure group of the tangent bundle to
$Spin(7)$; we shall also say that $M$ is a \emph{$Spin(7)$--manifold}. This can be described geometrically by saying that
there exists a nowhere vanishing global differential 4-form $\om$
on $M$ which can be locally written as \eqref{s1}. The 4-form
$\om$ is called the \emph{fundamental form} of the $Spin(7)$--manifold $M$ \cite{Bo}.  Alternatively, a $Spin(7)$--structure can be described by the existence of three-fold
vector cross product  on the tangent spaces of $M$ (see e.g. \cite{Gr}).

The fundamental form of a $Spin(7)$--manifold determines a
Riemannian metric  $g$ which % \emph{implicitly} through
  %$g_{ij}=\frac{1}{42}\om_{iklm}\Phi_{jklm}$.   This 
  is referred  as the   metric induced by $\om$. We write $\LC$ for the associated Levi-Civita
connection and  $||.||^2$ for the tensor norm with respect to $g$. Note the difference by a factor $k!$ with the norm of a k-form.% (see Convention~\ref{conv}).

%In addition, we will freely identify vectors and co-vectors via the induced metric $g$.

%In general, not every  compact 8-dimensional Riemannian spin manifold $M^8$
%admits a $Spin(7)$--structure. We explain the precise condition
%\cite{LM}. Denote by $p_1(M), p_2(M), {\mathbb X}(M), {\mathbb
%X}(S_{\pm})$ the first and the second Pontrjagin classes, the
%Euler characteristic of $M$ and the Euler characteristic of the
%positive and the negative spinor bundles, respectively. It is well
%known \cite{LM} that a compact spin 8-manifold admits a $Spin(7)$--structure if and only if ${\mathbb X}(S_+)=0$ or ${\mathbb X}(S_-)=0$.
%The latter conditions are equivalent to $p_1^2(M)-4p_2(M)+ 8{\mathbb X}(M)=0$, for an appropriate choice of the orientation \cite{LM}.

Let us recall that a $Spin(7)$--manifold $(M,g,\om)$ is said to be
parallel (torsion-free) if the holonomy $Hol(g)$ of the metric $g$ is a subgroup of $Spin(7)$. This is equivalent to saying
that the fundamental form $\om$ is parallel with respect to the
Levi-Civita connection of the metric $g$, $\nabla^g\om=0$.  

M. Fernandez shows in \cite{F} that
$Hol(g)\subset Spin(7)$ if and only if $d\om=0$ which is equivalent to $\delta\om=0$ since $\om$ is self-dual 4-form  (see also \cite{Br,Sal}).  It was observed  by Bonan that any parallel $Spin(7)$--manifold is Ricci flat
\cite{Bo}. The first known explicit example of complete parallel $Spin(7)$--manifold with $Hol(g)=Spin(7)$ was constructed by Bryant and
Salamon \cite{BS,Gibb}.
The first compact examples of parallel $Spin(7)$--manifolds with
$Hol(g)=Spin(7)$ were constructed by Joyce \cite{J1}.

There are 4 classes of $Spin(7)$--manifolds according to the
Fernandez classification \cite{F} obtained as irreducible $Spin(7)$
representations of  the space $\nabla^g\om$.

We let the expression
$$
 \om=\frac1{24}\om_{ijkl}e_{ijkl}
$$
and thus have the  identity (c.f.  \cite{GMW,Kar2})
\begin{equation}%{eqnarray}%\label{p1}%\nonumber
%\om_{ijpq}\om_{ijpq} & = & 336;
\label{idensp}
%\om_{ijpq}\om_{ajpq} = 42\delta_{ia};\quad
%\\\nonumber
\om_{ijpq}\om_{klpq} = 6\delta_{ik}\delta_{jl} -
 6\delta_{il}\delta_{jk} + 4\om_{ijkl}.%\nonumber
 \end{equation}
% \\\om_{ijks}\om_{abcs} &=& \delta_{ia} \delta_{jb} \delta_{kc} +\delta_{ib} \delta_{jc} \delta_{ka} +\delta_{ic} \delta_{ja} \delta_{kb}\\\nonumber
%&-& \delta_{ia} \delta_{jc} \delta_{kb} -\delta_{ib} \delta_{ja} \delta_{kc} -\delta_{ic} \delta_{jb} \delta_{ka} \\\nonumber
%&+& \delta_{ia}\om_{jkbc}+\delta_{ja}\om_{kibc}+\delta_{ka}\om_{ijbc}\\\nonumber
%&+& \delta_{ib}\om_{jkca}+\delta_{jb}\om_{kica}+\delta_{kb}\om_{ijca}\\\nonumber
%&+& \delta_{ic}\om_{jkab}+\delta_{jc}\om_{kiab}+\delta_{kc}\om_{ijab}.
%\end{\equation}%{eqnarray}

The Lee form $\theta^{spin}$ is defined by \cite{C1}
\begin{equation}\label{g2li}
\theta^{spin} = -\frac{1}{7}*(*d\om\wedge\om)=\frac17*(\delta\om\wedge
\om)=\frac1{7}(\delta\om)\lrcorner\om\quad \theta^{spin}_l=\frac1{42}(\delta\om)_{ijk}\om_{ijkl},
\end{equation}
where  $\delta=-*d*$ is the co-differential acting on $k$-forms in dimension eight.

The 4 classes of Fernandez classification \cite{F} can be described in
terms of the Lee form as follows \cite{C1}: $W_0 : d\om=0; \quad
W_1 : \theta^{spin} =0; \quad W_2 : d\om = \theta^{spin}\wedge\om; \quad W :
W=W_1\oplus W_2.$

A $Spin(7)$--structure of the class $W_1$ (i.e.
$Spin(7)$--structure with zero Lee form) is called
 \emph{a balanced $Spin(7)$--structure}.
If the Lee form is closed, $d\theta^{spin}=0,$ then the $Spin(7)$--structure is
locally conformally equivalent to a balanced one \cite{I} (see also \cite{Kar1,Kar2}).
It is known due to  \cite{C1} that the Lee form of a $Spin(7)$--structure in the class $W_2$ is closed and therefore such a
manifold is locally conformally equivalent to a parallel $Spin(7)$--manifold. 

%If $M$ is compact then it is shown in \cite[Theorem~4.3]{I} that  in every conformal class of  $Spin(7)$--structures $[\ph]$ there exists a unique $Spin(7)$--structure with co-closed Lee form, $\delta^{spin}\theta=0$. The compact $Spin(7)$--spaces with closed but not exact Lee form
%(i.e. the structure is not globally conformally parallel) have very different topology than the parallel ones
%\cite{I,IPP}. 
%Coeffective cohomology and coeffective numbers of %Riemannian manifolds with 
%a $Spin(7)$ manifold are studied in \cite{Ug}.

%\subsection{Decomposition of the space of forms} We take the following description of the decomposition of the space of forms from \cite{Kar2}.

Let $(M, \om)$ be a $Spin(7)$--manifold. The action of $Spin(7)$  on the tangent space induces an
action of $Spin(7)$ on $\Lambda^k(M)$ splitting the exterior algebra into orthogonal irreducible $Spin(7)$  subspaces, where
$\Lambda^k_l$ corresponds to an $l$-dimensional $Spin(7)$-irreducible subspace of $\Lambda^k$:
\begin{equation*}\label{decsp}
\begin{split}
\Lambda^2(M)=\Lambda^2_7\oplus\Lambda^2_{21}, \qquad \Lambda^3(M)=\Lambda^3_8\oplus\Lambda^3_{48},\qquad
\Lambda^4(M)=\Lambda^4_1\oplus\Lambda^4_7\oplus\Lambda^4_{27}\oplus\Lambda^4_{35}, \quad 
where
\end{split}
\end{equation*}
\begin{equation*}%\label{dec2}
\begin{split}
\Lambda^2_7=\{\phi\in \Lambda^2(M) | *(\phi\wedge\om)=3\phi\};\\
\Lambda^2_{21}=\{\phi\in \Lambda^2(M) | *(\phi\wedge\om)=-\phi\}\cong spin(7);\\
\Lambda^3_8=\{*(\alpha\wedge\om) |  \alpha\in\Lambda^1\}=\{\alpha\lrcorner\om\};\quad
\Lambda^3_{48}=\{\gamma\in \Lambda^3(M) | \gamma\wedge\om=0\}.
\end{split}
\end{equation*}
Hence, a 2-form $\phi$ decomposes into two $Spin(7)$--invariant parts, $\Lambda^2=\Lambda^2_7\oplus\Lambda^2_{21}$, and
\begin{equation*}
\begin{split}
\phi\in \Lambda^2_7 \Leftrightarrow \phi_{ij}\om_{ijkl}=6\phi_{kl},\quad 
\phi\in \Lambda^2_{21} \Leftrightarrow \phi_{ij}\om_{ijkl}=-2\phi_{kl}.
\end{split}
\end{equation*}
Moreover, following \cite{Kar2} one  considers the operator $D:\Lambda^2\longrightarrow\Lambda^4$ defined for a 2--form $\alpha$ by
\[(D\alpha)_{ijkl}=\alpha_{is}\om_{sjkl}+\alpha_{js}\om_{iskl}+\alpha_{ks}\om_{ijsl}+\alpha_{sl}\om_{ijks}.
\]
\begin{prop}\cite[Proposition~2.3]{Kar2}\label{spin7}
The kernel of $D$ is isomorphic to $\Lambda^2_{21}\cong spin(7)$.
\end{prop}
%For $k>4$ we have $\Lambda^k_l=*\Lambda^{8-k}_l$.
For $k=4$, following \cite{Kar2}, one considers the operator $\Xi:\Lambda^4 \longrightarrow\Lambda^4$ defined as follows
\begin{equation}\label{op}
\begin{split}
(\Xi(\sigma))_{ijkl}=\sigma_{ijpq}\om_{pqkl}+\sigma_{ikpq}\om_{pqlj}+\sigma_{ilpq}\om_{pqjk}+\sigma_{jkpq}\om_{pqil}+\sigma_{jlpq}\om_{pqki}+\sigma_{klpq}\om_{pqij}.
\end{split}
\end{equation}
\begin{prop}\cite[Proposition~2.8]{Kar2}\label{427}
The spaces $\Lambda^4_1,\Lambda^4_7,\Lambda^4_{27},\Lambda^4_{35}$ are all eigenspaces of the operator $\Xi$ with distinct eigenvalues. Specifically,
\begin{equation*}%\label{dec4}
\begin{split}
\Lambda^4_1=\{\sigma\in\Lambda^4:\Xi(\sigma)=-24\sigma\};\qquad \Lambda^4_7=\{\sigma\in\Lambda^4:\Xi(\sigma)=-12\sigma\};\\\Lambda^4_{27}=\{\sigma\in\Lambda^4:\Xi(\sigma)=4\sigma\}=\{\sigma\in\Lambda^4:\sigma_{ijkl}\om_{mjkl}=0\};\qquad \Lambda^4_{35}=\{\sigma\in\Lambda^4:\Xi(\sigma)=0\};\\
\Lambda^4_+=\{\sigma\in\Lambda^4:*\sigma=\sigma\}=\Lambda^4_1\oplus\Lambda^4_7\oplus\Lambda^4_{27};\qquad \Lambda^4_-=\{\sigma\in\Lambda^4:*\sigma=-\sigma\}=\Lambda^4_{35}.
\end{split}
\end{equation*}
\end{prop}
\subsection{The inclusions $Sp(2)\subset SU(4)\subset Spin(7)$}
It is well known that $Sp(2)\subset SU(4)\subset Spin(7)$ as a groups. We describe an explicit formulas giving  inclusions. Let $\{a,b,c\}$ be a cyclic permutation of $\{1,2,3\}$ and
\[\Phi^a=\frac12F^a\wedge F^a, \qquad a=1,2,3.
\]
It follows from \eqref{AAsp} that  the 2-forms $F^b$ and $F^c$ are of type (2,0)+(0,2) with respect to the almost complex structure $J_a$. Hence, the complex 4-form 
\begin{equation}\label{spsu}
\begin{split}\Psi^a=\ps^a+\sqrt{-1}\sp^a=\frac12(F^b+\sqrt{-1}F^c)\wedge(F^b+\sqrt{-1}F^c),\\ \ps^a=\frac12F^b\wedge F^b-\frac12F^c\wedge F^c=\Phi^b-\Phi^c, \quad \sp^a=F^b\wedge F^c
\end{split}
\end{equation}
 is a non-degenerate (4,0)-form with respect to $J_a$.
 
\noindent The corresponding $SU(4)$ structures are determined by $(F^a,\ps^a,\sp^a)$ giving  inclusions  $Sp(2)\subset SU(4)$.
 
 One gets from \eqref{AA} that the 4-form $\om$ given by
 \begin{equation}\label{suspin}
 \om=\frac12F\wedge F+\ps=\Phi+\ps
\end{equation}
defines a $Spin(7)$ structure giving  inclusion $SU(4)\subset Spin(7)$.

It follows from \eqref{spsu} applying  \eqref{suspin} 
\begin{equation}\label{spspin}
\om^a=\frac12F^a\wedge F^a+\frac12F^b\wedge F^b-\frac12F^c\wedge F^c=\Phi^a+\Phi^b-\Phi^c.
\end{equation}
%gives an inclusion $Sp(2)\subset Spin(7)$.
%Further, we describe the $Sp(2)$ strictures with torsion, known as AHKT, and $SU(4)$ structures with torsion, known as ACYT from the $Spin(7)$ point of view.

\section{The $Spin(7)$--connection with skew-symmetric torsion}
%The existence of parallel spinors with respect to a metric connection with torsion 3-form  in dimension 8 is important in supersymmetric string theories since the number of parallel spinors determines the number of preserved supersymmetries and this is the first Killing spinor equation in the heterotic Hull-Strominger system in dimension eight \cite{GKMW,GMW,GMPW, MS}. %The presence of a parallel spinor with respect to a metric connection with torsion 3-form leads to the reduction of the holonomy group of the torsion connection to a subgroup of $Spin(7)$. 
It is shown in \cite{I1} that any $Spin(7)$--manifold $(M,\ph)$ admits a unique $Spin(7)$--connection with totally skew-symmetric torsion. 
\begin{thrm}\label{cspin}\cite[Theorem~1]{I1}
Let $(M,\om)$  be a $Spin(7)$--manifold with fundamental 4-form $\om$. There always exists a unique linear connection $\sb$ preserving the $Spin(7)$--structure, $\sb\om=\sb g=0,$ with totally skew-symmetric torsion $T$ given by
\begin{equation}\label{torcy}
T=*d\om-\frac76*(\theta^{spin}\wedge\om)=-\delta\om-\frac76\theta^{spin}\lrcorner\om,
\end{equation}
where the Lee form $\theta^{spin}$ is given by \eqref{g2li}.
\end{thrm}
%Note that we use here $\ph:=-\ph$ in \cite{I}.  

See also \cite{Fr,Mer} for subsequent proofs of this theorem.

\subsection{The torsion tensor}
In this section we recall some formulas for the torsion 3-form of the $Spin(7)$ structure taken from \cite{Iap}. Note the sign difference between the $Spin(7)$ form from \cite{Iap}.

Express the codifferential of a 4-form in terms of the Levi-Civita connection and then in terms of the torsion connection using \eqref{tsym} and $\sb\om=0$ to get \cite{Iap}
\begin{equation}\label{dphi}
\begin{split}
-\delta\om_{klm}%\sb^g_j\om_{jklm}=\sb_j\om_{jklm}-\frac12T_{jsk}\om_{jslm}-\frac12T_{jsl}\om_{jsmk}-\frac12T_{jsm}\om_{jskl}\\
=-\frac12T_{jsk}\om_{jslm}-\frac12T_{jsl}\om_{jsmk}-\frac12T_{jsm}\om_{jskl}.
\end{split}
\end{equation}
%since $\sb\Phi=0$.% and $\sb^g=\sb-\frac12T$.
Substitute \eqref{dphi} into \eqref{torcy} to obtain the following expression for the 3-form torsion $T$,
\begin{equation}\label{torcy2}
T_{klm}=-\frac12T_{jsk}\om_{jslm}-\frac12T_{jsl}\om_{jsmk}-\frac12T_{jsm}\om_{jskl}-\frac76\theta^{spin}_s\om_{sklm}.
\end{equation}
Applying \eqref{idensp},  it is straightforward to check from \eqref{g2li} and \eqref{torcy2} that the Lee form $\theta^{spin}$ can be expressed in terms of the torsion $T$ and  the 4-form $\om$ as follows
\begin{equation}\label{tit}
\begin{split}
\theta^{spin}_i=\frac17T_{jkl}\om_{jkli}.
\end{split}
\end{equation}
%Denote the skew-symmetric part of $\sb\theta^{spin}$ by $d^{\sb}\theta^{spin}$, $d^{\sb}\theta^{spin}_{ij}=\sb_i\theta^{spin}_j-\sb_j\theta^{spin}_i$, 
We express the co-differential of the torsion  with  the next formula from \cite[Proposition~4.2]{Iap}
%\begin{prop}\label{propDT}
%On a $Spin(7)$ manifold the following formulas hold true
\begin{equation}\label{sdeltaT}
%\theta\lrcorner\delta\ph=\theta\lrcorner T,\qquad 
-\delta T=\frac76(d\theta^{spin}\lrcorner\om+\theta^{spin}\lrcorner T)=\frac76\Big(d^{\sb}\theta^{spin}\lrcorner \om+(\theta^{spin}\lrcorner T)\lrcorner\om+\theta^{spin}\lrcorner T\Big).
\end{equation}

\section{Almost Calabi-Yau with torsion  in dimension 8}
% Necessary conditions to have a solution to the system of dilatino and gravitino
%equations in dimension 6 were derived by Strominger in \cite{Str} and then studied by
%many authors \cite{GKMW,GMPW,GMW,Car,Car1,BB,BBE,BJ,GPap}

%Necessary conditions to solve the gravitino equation are given in\cite{FI}.
The presence of a two parallel spinors with respect to a metric connection with torsion 3-form  in dimension 8 leads
 to the reduction to $SU(4)$, i.e. 
 the existence of an almost
Hermitian structure and a linear
connection preserving the almost Hermitian structure with torsion
3-form %and thirdly to the reduction of the
with holonomy group inside %of the torsion connection to be a subgroup of
$SU(4)$.
The reduction of the holonomy group of the torsion connection to $SU(4)$, was investigated intensively when the induced almost complex structure is integrable ($N=0$) and it is known as the Strominger condition \cite{Str}).  We present here the general $G_1$ case  from $Spin(7)$ point of view.

Let $(M^8,g,J,\Psi)$ be an 8-dimensional smooth manifold with an $SU(4)$-structure $(g,J,\Psi)$ 
or equivalently, the almost hermitian manifold $(M^8,g,J)$ has topologically trivial canonical
bundle trivialized by a (4,0)-form $\Psi$. We know that there exists a unique connection $\sb$ preserving the almost Hermitian structure $(g,J)$ with torsion 3-form given by \eqref{cy2} exactly when the Nijenhuis tensor $N$ is e 3-form. We investigate  under which conditions the connection $\sb$ preserves  the 4-forms $\ps$ and $\sp$. 

In addition to the Lee 1-form $\ts$ we define another 1-form $\tn$ using $N$ and $\ps,\sp$ as follows 
\begin{equation}\label{thetN}
\begin{split}
\tn_i=\frac1{24}N_{jkl}\ps_{jkli}=\frac16T_{jkl}\ps_{jkli}=-\frac1{24}N_{jkl}\sp_{jkls}F_{si}, \\(J\tn)_p=-\frac1{24}N_{jkl}\sp_{jklp}=-\frac16T_{jkl}\sp_{jklp}=-\tn_sF_{sp}.
\end{split}
\end{equation}
%where we apply \eqref{idensu}.

The 1-form $\tn$ determines completely the Nijenhuis tensor due to the next algebraic 
\begin{prop}\label{algsu}
Let $(M^8,g,J,\Psi)$ be an 8-dimensional manifold with an $SU(4)$-structure.

If the Nijenhuis tensor is a 3-form then the following identities hold true
\begin{equation}\label{nnu}
\begin{split}
N=-\tn\lrcorner \ps=J\tn\lrcorner\sp,\quad ||N||^2=24||\tn||^2;\\
 8(J\tn\wedge F)_{pqr}=N_{pjk}\ps_{jkqr}+N_{qjk}\ps_{jkrp}+N_{rjk}\ps_{jkpq}.
\end{split}
\end{equation}
\end{prop}
\begin{proof}
We obtain from \eqref{thetN} applying the last identity of \eqref{idensu} 
that
\begin{multline*}
\tn_s\ps_{pqrs}=\frac1{24}N_{ijk}\ps_{ijks}\ps_{pqrs}=N_{pqr}-(J\tn\wedge F)_{pqr}+\frac18\Big[N_{pjk}\ps_{jkqr}+N_{qjk}\ps_{jkrp}+N_{rjk}\ps_{jkpq}  \Big].
\end{multline*}
The first two terms are of type (3,0)+(0,3) with respect to $J$, the third one is of type (1,2)+(2,1). Taking into account that $N$ is of type (3,0)+(0,3) and $\ps$ is of type (4,0)+(0,4) we find that the fourth term is of type (1,2)+(2,1). Thus we get \eqref{nnu} by comparing the parts.%Then \eqref{nnu} and  \eqref{idensu} yield the (1,2)+(2,1) part vanishes.
\end{proof}
The next result gives  necessary and sufficient conditions for the existence of ACYT 8-manifold.
\begin{thrm}\label{cythm1}
Let $(M,g,J,\Psi)$ be an 8-dimensional smooth manifold with an $SU(4)$-structure $(g,J,\Psi)$.
%or equivalently, the almost hermitian manifold $(M^6,g,J)$ has topologically trivial canonical
%bundle trivialized by a (3,0)-form $\Psi$.
The next two conditions are equivalent:
\begin{enumerate}
\item[a)] $(M,g,J,\Psi)$ is an ACYT space, i.e. there exists a unique $SU(4)$-connection $\sb$ with torsion 3-form, %a linear connection with torsion 3-form which preserves the almost hermitian structure
%whose  holonomy is contained in SU(4),
\begin{equation*}%\label{consu3}
\sb g=\sb J=\sb \Psi^+=\sb\Psi^-=0.
\end{equation*}
\item[b)] The Nijenhuis tensor is totally skew-symmetric given by $N=-\tn\lrcorner\ps$ and the next  condition holds
\begin{equation}\label{cyconsu}
\delta\ps=J\tn\wedge F-\ts\lrcorner\ps\quad \Longleftrightarrow \quad  d\ps=\ts\wedge\ps+*(J\tn\wedge F)=
\ts\wedge\ps+\tn\wedge\ph
\end{equation}
\end{enumerate}
The torsion $T$  on an ACYT 8-manifold is given by
\begin{equation}\label{cyt8}
\begin{split}
T=-\delta\ph+J\ts\wedge F+N=-\delta\ph-\ts\lrcorner\ph-\tn\lrcorner\ps=-\delta\ph-\ts\lrcorner\ph+J\tn\lrcorner\sp.
\end{split}
\end{equation}
The co-differential of $\sp$ on ACYT 8-manifold satisfies
\begin{equation}\label{cyt9}
\delta\sp=-\ts\lrcorner\sp+\tn\wedge F \quad \Longleftrightarrow \quad d\sp=\ts\wedge\sp+*(\tn\wedge F)=\ts\wedge\sp-J\tn\wedge\ph.
\end{equation}
\end{thrm}
\begin{proof}
We start with the following
\begin{lemma}\label{toru4}
Let $(M,g,J,\Psi)$ be an $SU(4)$ 8-manifold with totally skew symmetric Nijenhuis tensor and $\sb$ be the unique $U(4)$ connection with torsion 3-form $T$.

Then the torsion is given by  \eqref{cyt8}.
\end{lemma}
\begin{proof}
Apply \eqref{dphi} to the $\sb$-parallel 4-form $\ph=\frac12F\wedge F$, we have using \eqref{liff}
\begin{multline}\label{deltaF}
-\delta\ph_{klm}=-\frac12\Big( T_{sbk}\ph_{sblm} +T_{sbl}\ph_{sbmk} +T_{sbm}\ph_{sbkl} \Big)\\=
%-\frac12T_{sbk}\Big(F_{sb}F_{lm}+F_{bl}F_{sm}+F_{ls}F_{bm}\Big)\\-\frac12T_{sbl}\Big(F_{sb}F_{mk}+F_{bm}F_{sk}+F_{ms}F_{bk}\Big)-\frac12T_{sbm}\Big(F_{sb}F_{kl}+F_{bk}F_{sm}+F_{ks}F_{bm}\Big)
-J\ts_kF_{lm}+T_{sbk}F_{sl}F_{bm}-J\ts_lF_{mk}+T_{sbl}F_{sm}F_{bk}-J\ts_mF_{kl}+T_{sbm}F_{sk}F_{bl}\\=-(J\ts\wedge F)_{klm}+T^+_{klm}-\frac34N_{klm}=-(J\ts\wedge F)_{klm}+T_{klm}-N_{klm},
\end{multline}
where we apply  \eqref{12} to $T^+$ and \eqref{123} to $T^-=\frac14N$. Taking into account \eqref{nnu} we obtain  \eqref{cyt8}. %The lemma is proved.
\end{proof}

Suppose a) holds. 
Let $\sb$ be an $SU(4)$-connection %preserving the (4,0)+(0,4) forms $\ps$ and $\sp$, $\sb\ps=\sb\sp=0$. 
Then the connection $\sb$ preserves the Spin(7)-4-form $\om$ defined by \eqref{suspin} and therefore $\sb$ coincides with the unique $Spin(7)$ connection with torsion 3-form given by \eqref{torcy} according to Theorem~\ref{cspin}.

We calculate the relations between  the Lee forms using \eqref{tit}, \eqref{suspin}, \eqref{cy2},\eqref{liff} and \eqref{thetN}.
\begin{equation}\label{Leefs}
\tsp_i=\frac17T_{jkl}\om_{jkli}=\frac17\Big(T^+_{jkl}+\frac14N_{jkl}\Big)\Big(F_{jk}F_{li}+F_{kl}F_{ji}+F_{lj}F_{ki}+\ps_{jkli}\Big)=\frac67(\ts_i+\tn_i).
\end{equation}

On the other hand, since the 4-form $\ph$ is self-dual, we obtain using \eqref{acy} 
\begin{equation}\label{delF}
-\delta\ph=*d\ph=*(dF\wedge F)=*(dF^+\wedge F)+\frac34*(JN\wedge F).
\end{equation}
Compare the types of \eqref{deltaF} and \eqref{delF} to obtain
\begin{equation}\label{suT}
T^+=*(dF^+\wedge F)+J\ts\wedge F; \qquad N=-*(JN\wedge F).
\end{equation}

Further, we evaluate the second term in \eqref{torcy} with the help of \eqref{Leefs} and \eqref{suspin} as follows
\begin{multline}\label{secon}
-\frac76\tsp_s\om_{sklm}=-\frac76\tsp_s\Big(F_{sk}F_{lm}+F_{kl}F_{sm}+F_{ls}F_{km}+\ps_{sklm}  \Big)\\=\frac76(J\tsp\wedge F)_{klm}-\frac76(\tsp\lrcorner\ps)_{klm}
%-(\ts_s+\tn_s)\Big(F_{sk}F_{lm}+F_{kl}F_{sm}+F_{ls}F_{km}+\ps_{sklm}  \Big)\\
=[(J\ts+J\tn)\wedge F]_{klm}-\ts_s\ps_{sklm}-\tn_s\ps_{sklm}.
\end{multline}
Substitute \eqref{deltaF} and \eqref{secon} into \eqref{torcy} and use \eqref{cy2} to get
\begin{equation*}
\begin{split}
T^++\frac14N=-\delta\om-\frac76\tsp\lrcorner\om=-\delta\ph-\delta\ps-\frac76\tsp\lrcorner(\ph+\ps)\\=-J\ts\wedge F+T^+-\frac34N-\delta\ps+J\ts\wedge F+J\tn\wedge F-\ts\lrcorner\ps-\tn\lrcorner\ps.
\end{split}
\end{equation*}
The last equality yields $
N=-\delta\ps+J\tn\wedge F-\ts\lrcorner\ps-\tn\lrcorner\ps 
$
and \eqref{nnu} implies  \eqref{cyconsu}.

To show  \eqref{cyt9}, we apply \eqref{dphi} to the $\sb$-parallel 4-forms  $\Psi^{\pm}$ yielding
\[
\begin{split}
2\delta\Psi^{\pm}_{klm}=\Big( T_{stk}\Psi^{\pm}_{stlm} +T_{stl}\Psi^{\pm}_{stmk} +T_{stm}\Psi^{\pm}_{stkl} \Big).%\\
%2\delta\sp_{klm}=\Big( T_{sbk}\sp_{sblm} +T_{sbl}\sp_{sbmk} +T_{sbm}\sp_{sbkl} \Big).
\end{split}
\]
Taking into account $\sp$ is a (4,0)+(0,4) form, this leads to
\[
2\delta\ps_{klm}F_{mp}=2\delta\sp_{klp}+\big[T_{ist}F_{sj}F_{tp}-T_{ijp}  \Big]\sp_{ijkl}=2\delta\sp_{klp}+\big[T_{sjt}F_{si}F_{tp}-T_{stp}F_{si}F_{tj}  \Big]\sp_{ijkl}.
\]
In view of \eqref{cy2}, \eqref{12},  \eqref{nnu} and the already established \eqref{cyconsu}, we have
\[
\begin{split}
4\delta\sp_{klp}=4\delta\ps_{klm}F_{mp}+N_{ijp}\sp_{ijkl}=4(J\tn\wedge F)_{klm}F_{mp}-4\theta_s\sp_{sklp}-\tn_s\ps_{spij}\sp_{ijkl}\\=-4(\theta\lrcorner\sp)_{klp}+4(\tn\wedge F)_{pkl},
\end{split}
\]
where we use \eqref{idensu} to evaluate the last term in the first line.

For the converse, suppose b) holds and let $\sb$ be the unique $U(4)$ connection with torsion 3-form $T$ given by  \eqref{cyt8} due to Lemma~\ref{toru4}.  We consider the Spin(7) structure $\om$ defined by \eqref{suspin}.
Let $\sb^s$ be the unique linear connection preserving $\om$ with torsion 3-form $T^s$ given by \eqref{torcy} according to Theorem~\ref{cspin}.  We will  show $T=T^s$ which imply $\sb^s=\sb$. Then $\sb\om=\sb\ph+\sb\ps=0$ and $\sb\ps=0$ since $\sb\ph=0$.

We obtain from \eqref{torcy} and \eqref{suspin} applying  \eqref{cyt8},\eqref{cyconsu} and \eqref{secon}   %the next sequence of equalities
\begin{multline}\label{eqsuspin}
T^s=-\delta\om-\frac76\tsp\lrcorner\om=-\delta\ph-\delta\ps-\frac76\tsp\lrcorner(\ph+\ps)\\=T-J\ts\wedge F+\tn\lrcorner\ps-J\tn\wedge F+\ts\lrcorner\ps-\frac76\tsp\lrcorner(\ph+\ps)\\=T
-(J\ts+J\tn)\wedge F+(\ts+\tn)\lrcorner\ps+\frac76(J\tsp\wedge F-\tsp\lrcorner\ps)\\=
T -(J\ts+J\tn-\frac76J\tsp)\wedge F+(\ts+\tn-\frac76\tsp)\lrcorner\ps.%=T
\end{multline}
%where we  use \eqref{Leefs} to conclude the last equality.
With the help of \eqref{Leefs}, the equation \eqref{eqsuspin} yields 
\begin{multline*}
\tsp_i=\frac17T^s_{jkl}\om_{jkli}=\frac67(\ts_i+\tn_i)\\ -\Big[(J\ts+J\tn-\frac76J\tsp)\wedge F-(\ts+\tn-\frac76\tsp)\lrcorner\ps\Big]_{jkl}\Big[\ph_{jkli}+\ps_{jkli}\Big].
\end{multline*}
Set $A=\ts+\tn-\frac76\tsp$, the above equation  can be written in the form
\[-\frac76A_i=-(JA\wedge F)_{jkl}\ph_{jkli}+A_s\ps_{sjkl}\ps_{jkli}=-18A_i-24A_i.=-42A_i,
\]
where we used \eqref{idensu} to get the last equality.
Hence, $A=\ts+\tn-\frac76\tsp=0$ which substituted into \eqref{eqsuspin} yields $T^s=T$.
%The proof is completed.
%\begin{multline*}
%\tsp_i=\frac17T^s_{jkl}\om_{jkli}=\frac67(\ts_i+\tn_i)\\ -\Big[(J\ts+J\tn-\frac76J\tsp)\wedge F-(\ts+\tn-\frac76\tsp)\lrcorner\ps\Big]_{jkl}\Big[\ph_{jkli}+\ps_{jkli}\Big].
%\end{multline*}
%Set $A=\ts+\tn-\frac76\tsp$, the above equation  can be written in the form
%\[-\frac76A_i=-(JA\wedge F)_{jkl}\ph_{jkli}+A_s\ps_{sjkl}\ps_{jkli}=-18A_i-24A_i.=-42A_i,
%\]
%where we used \eqref{idensu} to get the last equality.
%Hence, $A=\ts+\tn-\frac76\tsp=0$ which substituted into \eqref{eqsuspin} yields $T^s=T$.
The proof is completed.
\end{proof}
In view of \eqref{cyt8} and \eqref{deltaF}, the torsion of an ACYT 8-manifold has the form
\begin{equation}\label{ntor1}
T_{klm}=-\frac12T_{jsk}\Phi_{jslm}-\frac12T_{jsl}\Phi_{jsmk}-\frac12T_{jsm}\Phi_{jskl}-\ts_s\Phi_{sklm}-\tn_s\ps_{sklm}.
\end{equation}
We obtain from \eqref{nnu} and \eqref{idensu}
\begin{equation}\label{npar1}
\begin{split}
\sb_i N_{jkl}=-\sb_i\tn_s\ps_{sjkl},\quad -\delta^{\sb}N_{kl}=\sb_iN_{ikl}=\frac12d^{\sb}\tn_{is}\ps_{iskl}=(d^{\sb}\tn\lrcorner\ps)_{kl};\\
2\delta^{\sb}N_{ij}\ps_{ijpq}=-d^{\sb}\tn_{st}\ps_{stij}\ps_{ijpq}=-8(d^{\sb}\tn_{pq}-d^{\sb}\tn_{st}F_{sp}F_{tq}).
\end{split}
\end{equation}
Consequently, we obtain
\begin{prop}\label{deltan0}
On an ACYT 8-manifold the Nijenhuis tensor is $\sb$-coclosed,
$\delta^{\sb}N=0$, if and only if the two form $d^{\sb}\tn$ is $J$-invariant, $d^{\sb}\tn=Jd^{\sb}\tn$.
\end{prop}
In what follows we denote the Lee form $\ts$ of an ACYT space simply by $\theta$.

The formula  \eqref{tsym} yields  the  general identity 
 \begin{equation}\label{nnid}
 d\omega=d^{\sb}\omega+\omega\lrcorner T, \quad \omega\in\Lambda^1.
 \end{equation}
We calculate the co-differential of the Nijenhuis tensor.
\begin{prop}\label{coN} The co-differential of the Nijenhuis tensor on an ACYT 8-manifold satisfies
\begin{equation}\label{codN}
\delta N=-(d\tn-\tn\wedge\theta)\lrcorner\ps=-d\tn\lrcorner\ps+\theta\lrcorner N=\delta^{\sb}N-(\tn\lrcorner T-\tn\wedge\theta)\lrcorner\ps.
\end{equation}
In particular $\delta N$ is (2,0)+(0,2) -form.
\end{prop}
\begin{proof}
We obtain from \eqref{nnu} applying \eqref{cyconsu} and \eqref{star} that
\[\delta N=-*d*N=-*(d\tn-\tn\wedge\theta)\wedge\ps=-(d\tn-\tn\wedge\theta)\lrcorner\ps.
\]
The last equality of \eqref{codN} follows form \eqref{nnid} and \eqref{npar1}.
\end{proof}
For a two form $\alpha\in\Lambda^2$, the first identity of \eqref{idensu} yields
\begin{equation}\label{new2}
\begin{split}
\alpha\lrcorner\ph=(\alpha\lrcorner F)F-J\alpha,\quad d\beta\lrcorner\ph=[-\delta(J\beta)+g(\beta,J\theta]F-Jd\beta, \beta\in \Lambda^1,
%\\d\theta_{ij}\ph_{ijpq}=d\theta_{ij}F_{ij}F_{pq}-2d\theta_{ij}F_{ip}F_{iq}=-2d\theta_{ij}F_{ip}F_{iq},
\end{split}
\end{equation}
where we used \eqref{nnid}, $\sb J=0$ and \eqref{liff} to conclude the last identity. Indeed, we have
 \begin{equation}\label{new1} 
 d\beta\lrcorner F=d^{\sb}\beta\lrcorner F+(\beta\lrcorner T)\lrcorner F=(\sb_{e_i}J\beta)e_i+g(\beta,J\theta)=-\delta(J\beta)+g(\beta,J\theta).
%-2\delta T_{pq}F_{pq} =6d\theta_{pq}F_{pq}=6d^{\sb}\theta_{pq}F_{pq} +6\theta_sT_{spq}F_{pq}={\color{red}-12\delta(J\theta)}+12g(\theta,J\theta)=0,
 \end{equation}
\begin{prop}
On an ACYT 8-manifold the next formulas hold true
\begin{equation}\label{4suspine}
\begin{split}
d\theta\lrcorner\ps+d\tn\lrcorner\ph+\tn\lrcorner T=d\theta\lrcorner\sp-dJ\tn\lrcorner\ph-J\tn\lrcorner T=0;\\
\delta T=-d\theta\lrcorner\ph-d\tn\lrcorner\ps-\theta\lrcorner T=-d\theta\lrcorner\ph+dJ\tn\lrcorner\sp-\theta\lrcorner T,\quad \delta T\lrcorner F=0.
%\\d\theta_{ij}\ps_{ijpq}+d\tn_{ij}\Phi_{ijpq}+2\tn_sT_{spq}=0;
%\\d\theta_{ij}\sp_{ijpq}-d(J\tn)_{ij}\Phi_{ijpq}-2(J\tn)_sT_{spq}=0;
%\\-2\delta T_{pq}=d\theta_{ij}\Phi_{ijpq}+d\tn_{ij}\ps_{ijpq}+2\theta_sT_{spq};\\
%-2\delta T_{pq}=d\theta_{ij}\Phi_{ijpq}-d(J\tn)_{ij}\sp_{ijpq}+2\theta_sT_{spq};%\\d\tn_{ij}\ps_{ijpq}=-d(J\tn)_{ij}\sp_{ijpq}.
\end{split}
\end{equation}
\end{prop}
\begin{proof}
%We apply \eqref{sdeltaT} to the four $Spin(7)$ structures defined by \eqref{suspin}, \eqref{3spin7} and obtain \eqref{4suspine} by suitable summations and subtractions of the obtained 4 equalities using \eqref{3spin} and  \eqref{Leefs}.
The first equality of \eqref{4suspine} (resp. the second one of \eqref{4suspine}) follows  by applying the exterior derivative to  \eqref{cyconsu} (resp. to \eqref{cyt9}) with the help of \eqref{star}, \eqref{1star},  the self-duality of $\ph$ and \eqref{cyt8} . The third equality of \eqref{4suspine} (resp. the fourth  equality of \eqref{4suspine})  can be obtained also by taking the exterior derivative of the Hodge $*$ operator of the  third equality of \eqref{cyt8} (resp. of the  fourth equality of \eqref{cyt8} ) and again using \eqref{star},  \eqref{1star} and  the self-duality of $\ph$.
%Multiply the first equality of \eqref{4suspine} by $F$ and use \eqref{idensu} to get the third one.
%The fifth one follows from the third and the fourth equalities and also can be check directly from the definitions.

 Using  the identity $\delta(J\theta)=0$ we get from \eqref{new1} $d\theta\lrcorner F=0$ and \eqref{new2} leads to 
 $d\theta\lrcorner\ph=-Jd\theta$. The fifth  identity in \eqref{4suspine} follows from the third identity of \eqref{4suspine} multiplied by $F$ taking into account   \eqref{new1}.
\end{proof}
\begin{cor}\label{cornew}
On an ACYT 8-manifold we have the formulas
\begin{equation}\label{4suspine2}
\begin{split}
2d\theta\lrcorner\ps=2(Jd\tn-d\tn)-(Jd^{\sb}\tn-d^{\sb}\tn);\quad
2d\tn\lrcorner\ps=2(Jd\theta-d\theta)+d^{\sb}\tn\lrcorner\ps;\\
6d\tn\lrcorner F+2g(\tn,J\theta)=-3\delta(J\tn)+4g(\tn,J\theta)=0;\quad
6d(J\tn)\lrcorner F+2g(\tn,\theta)=3\delta(\tn)+4g(\tn,\theta)=0.
%\\d\theta_{ij}\ps_{ijpq}=d\tn_{ij}F_{ip}F_{jq}-d\tn_{pq}+\tn_s(T_{sij}F_{ip}F_{jq}-T_{spq})\\
%2d\theta_{js}\Phi_{jskl}+4d\theta_{kl}+d\tn_{ij}\ps_{ijkl}+\tn_sT_{spq}\ps_{pqkl}=0;
%\\ 3d\tn_{ij}F_{ij}+2\tn_s(J\theta)_s=-3\delta(J\tn)+4g(\tn,J\theta)=0; \\3d(J\tn)_{ij}F_{ij}+2(J\tn)_s(J\theta)_s= 3\delta\tn+4g(\tn,\theta)=0.
 \end{split}
 \end{equation}
\end{cor}
\begin{proof} 
The second identity of \eqref{4suspine2} is a consequence of the first equality of \eqref{4suspine} after contraction  with $\ps$ and an application of \eqref{idensu} and \eqref{nnid} and \eqref{new2}.

Multiply   the first identity of \eqref{4suspine} with $F$ and use \eqref{idensu},  \eqref{nnid},  \eqref{new2}, \eqref{liff} and \eqref{new1} to get  the third identity of \eqref{4suspine2}. The fourth one follows identically starting with the second equality in \eqref{4suspine}.

Multiplying the first identity of \eqref{4suspine} with $\ph_{pqkl}$ and using \eqref{idensu}, we calculate
\begin{equation}\label{4suspine1}
\begin{split}
0=\Phi_{pqkl}\Big[d\theta_{ij}\ps_{ijpq}+d\tn_{ij}\Phi_{ijpq}+2\tn_sT_{spq}\Big]=2d\theta_{ij}\ps_{ijkl}+d\tn_{ij}\Big[4F_{ij}F_{kl}+2\delta_{ik}\delta_{jl}-2\delta_{il}\delta_{jk}  \Big]\\+2\tn_sT_{spq}\ph_{pqkl}%(F_{pq}F_{kl}+F_{qk}F_{pl}+F_{kp}F_{ql})\\
=2d\theta_{ij}\ps_{ijkl}+4d\tn_{ij}F_{ij}F_{kl}+4d\tn_{kl}+4\tn_sJ\theta_sF_{kl}-4\tn_sT_{sij}F_{ik}F_{jl}\\%=2d\theta_{ij}\ps_{ijkl}-2d\tn_{ij}F_{ij}F_{kl}+4d\tn_{kl}-4\tn_sT_{sij}F_{ik}F_{jl}\\
=2d\theta_{ij}\ps_{ijkl}-2d\tn_{ij}F_{ij}F_{kl}+4d\tn_{kl}-4d\tn_{ij}F_{ik}F_{jl}+4d^{\sb}\tn_{ij}F_{ik}F_{jl},%\end{equation}
%\\\\ {\color{red} d\theta_{ij}\ps_{ijab}=d\tn_{JaJb}-d\tn_{ab}+\tn_s(T_{sJaJb}-T_{sab})} \\{\color{blue} -d\tn_{JaJb}-d\tn_{ab}+d\tn_{ij}F_{ij}F_{ab}+\tn_s(T_{sab}+T_{sJaJb})=0}, \\ {\color{red}d^{\sb}\tn_{JaJb}+d^{\sb}\tn_{ab}=d\tn_{ij}F_{ij}F_{ab}=-\delta(J\tn)F_{ab}.}
 \end{split}
 \end{equation}
 where we apply the already proved third equality of \eqref{4suspine2} and \eqref{nnid} 
 to achieve the last  equality.
 
 Using \eqref{idensu}, we can write the first equality of \eqref{4suspine} in the form
 \begin{equation}\label{nsp4}d\theta_{ij}\ps_{ijpq}+d\tn_{ij}F_{ij}F_{pq}-2d\tn_{ij}F_{ip}F_{jq} +2\tn_sT_{spq}=0,
 \end{equation}
 which compared with \eqref{4suspine1} proves the first identity of \eqref{4suspine2}.
 % \begin{equation}\label{new1}
% -\delta T\lrcorner F=3d\theta\lrcorner F=0.
%  -2\delta T_{pq}F_{pq} =6d\theta_{pq}F_{pq}=6d^{\sb}\theta_{pq}F_{pq} +6\theta_sT_{spq}F_{pq}={\color{red}-12\delta(J\theta)}+12g(\theta,J\theta)=0,
% \end{equation}
% where we used the general identity $\delta(J\theta)=0$.
 \end{proof}
%{\color{blue} Hence, the 2-form $\tn\lrcorner T$ is of type (1,1) and if $\delta(J\tn)=0$ then $\tn\lrcorner T\in\frak{su}(4)$ since  $(\tn\lrcorner T)\lrcorner F=3\delta(J\tn)$. }
 %  \begin{equation}\label{4suspin}
%\begin{split}
%-2\delta T_{ab}=d\theta_{ij}\Phi_{ijab}+d\tn_{ij}\ps_{ijab}+2\theta_sT_{sab};\\
%-2\delta T_{ab}=d\theta_{ij}\Phi_{ijab}-d(J\tn)_{ij}\sp_{ijab}+2\theta_sT_{sab};\\d\tn_{ij}\ps_{ijab}=-d(J\tn)_{ij}\sp_{ijab}\\
%-2\delta T_{ab}F_{ab} =6d\theta_{ab}F_{ab}=6d^{\sb}\theta_{ab}F_{ab} +6\theta_sT_{sab}F_{ab}=3\delta(J\theta)+g(\theta,J\theta)=0, \\
%3d\tn_{ij}F_{ij}+2\tn_s(J\theta)_s=0,\quad 3\delta(J\tn)-2g(\tn,J\theta)=0.
%\end{split}
%\end{equation}
The co-differential of the torsion is a subject of the next
\begin{thrm}\label{thdeltaT}
On an 8-dimensional ACYT manifold we have 
\begin{equation}\label{new3}
\begin{split}
d\theta-Jd\theta=-d\tn\lrcorner\ps-\frac12\delta^{\sb}N, \quad
d\tn-Jd\tn=-d\theta\lrcorner\ps+\frac12[d^{\sb}\tn-Jd^{\sb}\tn];\\
%\delta T_{pq}=d^{\sb}%\theta_{pq}+d\theta_{pq}-d\theta_{ij}F_{ip}F_{jq}+\frac12\tn_sT_{skl}\ps_{klpq};\\
%-2\delta T_{pq}=%d\theta_{ij}\Phi_{ijab}+d^{\sb}\tn_{ij}\ps_{ijab}+\tn_sT_{sij}\ps_{ijab}+2\theta_sT_{sab}\\-2d\theta_{ij}F_{ip}F_{jq}+d\tn_{ij}\ps_{ijpq}+2\theta_sT_{spq};\\
%\delta T_{pq}=d^{\sb}\theta_{pq}-\frac14d^{\sb}\tn_{ij}\ps_{ijpq}=d^{\sb}\theta_{pq}-\frac12\sb_iN_{ipq}=d^{\sb}\theta_{pq}+\frac12\delta^{\sb}N_{pq} ;\\
%4(d\theta_{pq}-d\theta_{ij}F_{ip}F_{jq})+d\tn_{ij}\ps_{ijpq}+\tn_sT_{sij}\ps_{ijpq}=0;\\
%d\theta_{pq}-d\theta_{ij}F_{ip}F_{jq}=-\frac12d\tn_{ij}\ps_{ijpq}+\frac14d^{\sb}\tn_{ij}\ps_{ijpq}=-\frac12d\tn_{ij}\ps_{ijpq}-\frac12\delta^{\sb}N_{pq};\\
%d\tn_{pq}-d\tn_{ij}F_{ip}F_{jq}=-\frac12d\theta_{ij}\ps_{ijpq}+\frac12 d^{\sb}\tn_{pq}-\frac12d^{\sb}\tn_{ij}F_{ip}F_{jq};\\%\frac14d^{\sb}\tn_{ij}\ps_{ijpq}=-\frac12d\tn_{ij}\ps_{ijpq}-\frac12\delta^{\sb}N_{pq}.
\delta T=d^{\sb}\theta-\frac12d^{\sb}\tn\lrcorner\ps=d^{\sb}\theta+\frac12\delta^{\sb}N, \quad \delta N=d\theta-Jd\theta+\frac12\delta^{\sb}N-\theta\lrcorner N.
\end{split}
\end{equation}
%where $d\theta^{(2,0)+(0,2)}$ denotes the (2,0)+(0,2)-part of $d\theta$.
\end{thrm}
\begin{proof}
Substitute \eqref{new2} into the second equality of \eqref{4suspine2} and use the obtained equality together with the second equality of \eqref{4suspine} to get %the first identity of \eqref{new3} 
by applying
\eqref{nnid} that % the identity $d\theta=d^{\sb}\theta+\theta\lrcorner T$.
\[
\delta T_{pq}=d^{\sb}\theta_{pq}+d\theta_{pq}-d\theta_{ij}F_{ip}F_{jq}+\frac12\tn_sT_{skl}\ps_{klpq}.
\]
%The second one is 
As a consequence of \eqref{new2} and the second equality of \eqref{4suspine} we get
\[
-2\delta T_{pq}=%d\theta_{ij}\Phi_{ijab}+d^{\sb}\tn_{ij}\ps_{ijab}+\tn_sT_{sij}\ps_{ijab}+2\theta_sT_{sab}\\
-2d\theta_{ij}F_{ip}F_{jq}+d\tn_{ij}\ps_{ijpq}+2\theta_sT_{spq}.
\]
We  obtain from the last two identities with the help of \eqref{nnid} and \eqref{npar1} the first and the third  equality of  \eqref{new3}. Multiply the first one with $\ps$ and use \eqref{idensu} to get the second one. The fourth follows from the first one  and \eqref{codN}.
\end{proof}

In the integrable case $N=\tn=0$, i.e. for CYT 8-manifold Theorem~\ref{thdeltaT} and \eqref{new3} yield
\begin{cor}
On a CYT 8-manifolds we have
%If the almost complex structure $J$ is integrable, $N=0$ then 
$d\theta\in\frak{su}(4), \quad \delta T=d^{\sb}\theta.$
\end{cor}

\subsection{The Ricci tensors of ACYT 8-manifold}
On an ACYT 8-manifold we  express the Ricci tensor in a more appropriate form as follows. Since $\sb$ preserves the $SU(4)$-structure its curvature  lies in the Lie algebra $\frak{su}(4)$, i.e. it satisfies
\begin{equation}\label{rr}
\begin{split}
R_{ijpq}\ps_{pqrd}=R_{ijpq}\sp_{pqrd}=R_{ijpq}F_{pq}=0 \Longleftrightarrow R_{ijpq}\Phi_{pqkl}=-2R_{ijkl}.
%R(X,Y,e_i,e_j)\ps(e_i,e_j,Z)=R(X,Y,e_i,e_j)\sp(e_i,e_j,Z)=R(X,Y,e_i,e_j)F(e_i,e_j)=0 \Longleftrightarrow R(X,Y,e_i,e_j)\Phi(e_i,e_j,Z,V)=-2R(X,Y,Z,V).
%.\\
%R_{ijab}p_{abk}=0 \Longleftrightarrow R_{ijab}ph_{abkl}=-2R_{ijkl}.
\end{split}
\end{equation}
\begin{prop}\label{ricsu4}
On an ACYT 8-manifold the next formulas hold true
\begin{equation}\label{ricdtsu}
\begin{split}
Ric_{ij}=\frac1{12}dT_{ipqr}\Phi_{jpqr}-\sb_i\theta_j; \\ Scal=3\delta\theta+2||\theta||^2+2||\tn||^2-\frac13||T||^2;\\
dT_{ipqr}\ps_{pqrj}=-12\sb_i\tn_j, \quad dT_{ipqr}\sp_{pqrj}=12\sb_iJ\tn_j.
\end{split}
\end{equation}
\end{prop}
\begin{proof}
We have from \eqref{rr} using \eqref{1bi1}, \eqref{liff} and \eqref{dh}  that the Ricci tensor $Ric$ of  $\sb$ is given by
\begin{multline*}
Ric_{ij}=\frac12R_{ipqr}\Phi_{jpqr}=\frac16\Big[R_{ipqr}+R_{ibca}+R_{irpq} \Big]\Phi_{jpqr}\\=\frac1{12}dT_{ipqr}\Phi_{jpqr}+\frac16\sb_iT_{pqr}\Phi_{jpqr}
=\frac1{12}dT_{ipqr}\Phi_{jpqr}-\sb_i\theta_j.
\end{multline*}
%This completes the proof of the first identity in \eqref{ricg2}.
Similarly, we have applying \eqref{thetN}
\begin{equation*}%\label{ricnew}
\begin{split}0=R_{ipqr}\ps_{pqrj}=\frac16dT_{ipqr}\ps_{pqrj}+\frac13\sb_iT_{pqr}\ps_{pqrj}=\frac16dT_{ipqr}\ps_{pqrj}+2\sb_i\tn_j.%\\
%0=R_{ipqr}\sp_{pqrj}=\frac16dT_{ipqr}\sp_{pqrj}+\frac13\sb_iT_{pqr}\sp_{pqrj}=\frac16dT_{ipqr}\sp_{pqrj}-2\sb_iJ\tn_j
\end{split}
\end{equation*}
To show the formula for the scalar curvature, we first calculate using \eqref{idensu} that
\begin{equation}\label{dtsu}
dT(X,e_p,e_q,e_r)\ph(Y,e_p,e_q,e_r)=3dT(X,JY,e_i,Je_i)=-6(F\lrcorner dT)(X,JY).
\end{equation}
Take the trace of the first formula of \eqref{ricdtsu}, apply \eqref{su2}, \eqref{nnu} and \eqref{dtsu} to complete the proof.
\end{proof}
Proposition~\ref{ricsu4}, Theorem~\ref{thdeltaT} and \eqref{rics} yield
\begin{prop}\label{ricsymj}
On an ACYT 8-manifold we have the formula
\begin{equation}\label{dTN}
6\delta^{\sb}N_{ij}=dT_{jpqr}\ph_{ipqr}-dT_{ipqr}\ph_{jpqr}.
\end{equation}
\begin{itemize}
\item[a)]
The 2-form $dT(X,JY,e_i,Je_i)$ is (1,1)-form with respect to $J$ if and only if $\delta^{\sb}N=0$.
\vspace{-2mm}
\item[b)] The Ricci tensor is symmetric if and only if $d^{\sb}\theta=-\frac12\delta^{\sb}N$.
\vspace{-2mm}
\item[c)] The Ricci tensor is $J$ invariant if and only if $\delta^{\sb}N(X,Y)=-2(\sb_X\theta)Y+2(\sb_{JX}\theta)JY$.
\end{itemize}
If the Ricci tensor is symmetric and $J$-invariant then 
$$(\sb_X\theta)Y-(\sb_{JY}\theta)JX=0\Longleftrightarrow (\sb_XJ\theta)Y+(\sb_YJ\theta)X=0.$$  In particular, $J\theta$ is a Killing vector field.
\end{prop}
\begin{proof}
We have from \eqref{rics}, \eqref{ricdtsu} and the third equality of \eqref{new3} that
\[-d^{\sb}\theta_{ij}-\frac12\delta^{\sb}N_{ij}=-\delta T_{ij}=Ric_{ij}-Ric_{ji}=-\frac1{12}\Big[dT_{jpqr}\ph_{ipqr}-dT_{ipqr}\ph_{jpqr}\Big]-d^{\sb}\theta_{ij}.\]
Hence, \eqref{dTN}  and the claim b) follow.

To show a), we apply \eqref{dtsu} to \eqref{dTN}  and get  
\begin{equation}\label{ndT}
2\delta^{\sb}N(X,Y)=-dT(X,JY,e_i,Je_i)-dT(JX,Y,e_i,Je_i).
\end{equation}
For c), we have in view of \eqref{dtsu}, \eqref{dTN}, \eqref{ndT} and \eqref{su} that
\begin{equation*}\label{ricdtsu1}
\begin{split}
Ric(X,Y)-Ric(JX,JY)=\frac14\Big[dT(X,JY,e_i,Je_i)+dT(JX,Y,e_i,Je_i)\Big]-(\sb_X\theta)Y+(\sb_{JX}\theta)JY\\=
-\frac12\delta^{\sb}N(X,Y)-(\sb_X\theta)Y+(\sb_{JX}\theta)JY.
\end{split}
\end{equation*}
The last claim is an easy consequence of b), c) and the  identity $(\sb_XJ\theta)Y=-(\sb_X\theta)JY$. %together with c) shows  $(\sb_XJ\theta)Y+(\sb_YJ\theta)X=0$.
\end{proof}
Applying \eqref{4suspine} we obtain from Theorem~\ref{thdeltaT} and Proposition~\ref{ricsymj}
\begin{cor} \label{corahkt}
Let $(M,g,J,\ps)$ be an ACYT 8-manifold.
\begin{itemize}
\item If it is  balanced, $\theta=0$, then it holds 
$$2\delta T=2\delta N=\delta^{\sb}N=-2d\tn\lrcorner\ps=-d^{\sb}\tn\lrcorner\ps, \qquad \delta\tn=\delta(J\tn)=0.$$
\item
If the Nijenhuis tensor is  $\sb$-coclosed, $\delta^{\sb}N=0$, we have 
\[\delta T=d^{\sb}\theta, \quad d\theta-Jd\theta=\delta N+\theta\lrcorner N=-d\tn\lrcorner\ps.\]
In particular, $d\theta=Jd\theta$ exactly when $d\tn=Jd\tn$.
\item If it is balanced with $\sb$-coclosed Nijenhuis tensor then 
$$\delta T=\delta N=\delta^{\sb}N=d\tn\lrcorner\ps=d^{\sb}\tn\lrcorner\ps=(\tn\lrcorner T)\lrcorner\ps=\delta\tn=\delta(J\tn)=0$$
and the Ricci tensor is symmetric and $J$-invariant.

In particular, the 2-forms $d\tn,d^{\sb}\tn, \tn\lrcorner T$ are $J$-invariant, belong to the Lie algebra $\frak{su}(4)$.
\end{itemize}
\end{cor}
For the second Ricci two form $\kappa$ we have
\begin{prop}\label{kapa}
On an ACYT 8-manifold  two forms $\kappa$ and  $F\lrcorner(d^{\sb}T)$% $d^{\sb}T(X,Y,e_i,Je_i)$ 
are given by
\begin{equation}\label{kapa1}
\begin{split}
2\kappa(X,Y)=d^{\sb}J\theta(X,Y)+(\sb_{e_i}T)(Je_i,X,Y);\\
-2(F\lrcorner d^{\sb}T)(X,Y)=d^{\sb}T(X,Y,e_i,Je_i)=2d^{\sb}J\theta(X,Y)-2(\sb_{e_i}T)(Je_i,X,Y).
\end{split}
\end{equation}
If the manifold is balanced, $\theta=0$, then
\[d^{\sb}T(X,Y,e_i,Je_i)=-2(\sb_{e_i}T)(Je_i,X,Y)=-4\kappa(X,Y) \in \frak{su}(4).\]
\end{prop}
\begin{proof}
We use the next identity from  \cite[(3.38)]{I}
\begin{multline}\label{inst1}
2R(X,Y,Z,V)-2R(Z,V,X,Y)\\=(\sb_XT)(Y,Z,V)-(\sb_YT)(X,Z,V)-(\sb_ZT)(X,Y,V)+(\sb_VT)(X,Y,Z)%\\=
%d^{\sb}T(X,Y,Z,V)-2(\sb_ZT)(X,Y,V)+2(\sb_VT)(X,Y,Z)\\
%=-d^{\sb}T(X,Y,Z,V)+2(\sb_XT)(Y,Z,V)-2(\sb_YT)(X,Z,V).
\end{multline}
The trace of  \eqref{inst1} gives with the help of  %\eqref{rr}, 
\eqref{liff} %and the first identity in \eqref{idensu} 
that 
\begin{equation}\label{inst2}
\begin{split}
2\kappa(X,Y)-2\rho(X,Y)=R(e_i,Je_i,X,Y)-R(X,Y,e_i,Je_i,)%\\=\frac12\Big[-(\sb_XT)(Y,e_i,Je_i)+(\sb_{Y}T)(X,e_i,Je_i)+(\sb_{e_i}T)(Je_i,X,Y)-(\sb_{Je_i}T)(e_i,X,Y)\Big]\\
%=(\sb_{Y}\theta)JX-(\sb_X\theta)JY+(\sb_{e_i}T)(Je_i,X,Y)=-(\sb_{Y}J\theta)X+(\sb_XJ\theta)Y+(\sb_{e_i}T)(Je_i,X,Y)\\
=d^{\sb}J\theta(X,Y)+(\sb_{e_i}T)(Je_i,X,Y),%=-2(\sb_{e_i}T)(Je_i,X,Y).
\end{split}
\end{equation}
where we used  $(\sb_XJ\theta)Y=-(\sb_X\theta)JY$ to achieve the last equality. Set $\rho=0$ to complete the proof of the first line. 
The second line follows  from \eqref{dh} similarly as in the proof of \eqref{inst2}.
\end{proof}
We remark that Proposition~\ref{kapa}  holds for ACYT spaces of dimension $2n$.
\begin{cor}\label{ncorickap}
On an  ACYT 8-manifold with $\sb$-coclosed  Nijenhuis tensor and $J$--invariant Ricci tensor it holds
\begin{equation}\label{ncoclosed}
\begin{split}
 (\sb_X\theta)Y=(\sb_{JX}\theta)JY, \quad \delta T=d^{\sb}\theta\in\frak{su}(4); \\%d^{\sb}J\theta(X,Y)=d^{\sb}J\theta(JX,JY);\\
 (\sb_{e_i}T)(Je_i,X,Y)=(\sb_{e_i}T)(Je_i,JX,JY).
 \end{split}
 \end{equation}
 If, in addition, the second Ricci 2-form $\kappa=0$ then
 \[(\sb_{e_i}T)(Je_i,X,Y)=-d^{\sb}J\theta(X,Y).\]
 \end{cor}
 \begin{proof}
 Proposition~\ref{ricsymj}, Proposition~\ref{thdeltaT} together with $\delta^{\sb}N=0$ imply the first two identities in  \eqref{ncoclosed}. The third one follows from the second one and \eqref{kapa1} since $\kappa$ is an (1,1)-form. 
 
 Set $\kappa=0$ into \eqref{kapa1} to complete the proof.
 \end{proof}
\begin{prop}\label{jriccom}
On a compact 8-dimensional ACYT space with $\sb$-coclosed  Nijenhuis tensor and $J$--invariant Ricci tensor if either $\kappa=0$ or $F\lrcorner d^{\sb}T=0$, i.e.
\begin{equation}\label{kapjric}
\begin{split}
\delta^{\sb}N=0,\quad Ric(X,Y)=Ric(JX,JY),  \quad \kappa=0, \quad or\\
\delta^{\sb}N=0, \quad Ric(X,Y)=Ric(JX,JY), \quad F\lrcorner d^{\sb}T=0
\end{split}
\end{equation}
 then the Lee form is $\sb$-parallel, $\sb\theta=0$.
%In particular, if the torsion connection ia an $SU(4)$--instanton  and the  Nijenhuis tensor is $\sb$-coclosed on a combact ACYT 8-manifold then the Lee form is $\sb$-parellel, $\sb\theta=0$.

\noindent In particular, the torsion is coclosed and the Ricci tensor is symmetric, $\delta T=Ric(Y,X)-Ric(X,Y)=0$.
\end{prop}
\begin{proof}
We have from \eqref{kapjric} and Proposition~\ref{ricsymj} that the Lee form is $J$-invariant,$(\sb_X\theta)Y=(\sb_{JX}\theta)JY$.

Now, the Ricci identiy for the torsion connection $\sb$ %the instanton condition, \eqref{ncoclosed} 
and \eqref{liff} imply
\begin{multline}\label{patrecom}
-\frac12\Delta||\theta||^2=\frac12\sb_{e_i}\sb_{e_i}||\theta||^2=\theta(e_j)(\sb_{e_i}\sb_{e_i}\theta)(e_j)+||\sb\theta||^2=-\theta(e_j)(\sb_{e_i}\sb_{Je_i}J\theta)(e_j)+||\sb\theta||^2\\
%=-\frac12\theta(e_j)(\sb_{e_i}\sb_{Je_i}-\sb_{Je_i}\sb_{e_i}J\theta)(e_j)+||\sb\theta||^2\\
=\frac12\theta(e_j)\Big[R(e_i,Je_i,e_j,e_s)J\theta(e_s)+T(e_i,Je_i,e_s)(\sb_{e_s}J\theta)(e_j) \Big]+||\sb\theta||^2\\
=\kappa(\theta,J\theta)-\theta(e_j)J\theta(e_s)(\sb_{e_s}J\theta)(e_j)+||\sb\theta||^2=\kappa(\theta,J\theta)+\frac12\sb_{\theta}||\theta||^2+||\sb\theta||^2,
%\\=\kappa(\theta,J\theta)-(\sb_{J\theta}J\theta)\theta+||\sb\theta||^2=\kappa(\theta,J\theta)+\frac12\sb_{\theta}||\theta||^2+||\sb\theta||^2,
\end{multline}
where $\Delta =-\LC_i\LC_i=-\sb_i\sb_i$ is the Laplacian acting on smooth functions.

If $\kappa=0$ then $$-\frac12\Delta||\theta||^2-\frac12\sb_{\theta}||\theta||^2=||\sb\theta||^2\ge 0.$$
 If $F\lrcorner d^{\sb}T=0$ we obtain from \eqref{kapa1} $\kappa(\theta,J\theta)=d^{\sb}J\theta(\theta,J\theta)=\sb_{\theta}||\theta||^2$  and  $$-\frac12\Delta||\theta||^2-\frac32\sb_{\theta}||\theta||^2=||\sb\theta||^2\ge 0.$$
In both cases, since $M$ is  compact, we  apply the strong maximum principle (see e.g. \cite{YB,GFS}) to achieve $\sb\theta=0.$ Hence, $\delta T=d^{\sb}\theta=0$ by Proposition~\ref{ricsymj} and \eqref{rics} yields  the Ricci tensor is symmetric.
%The second claim follows from the first one applying Proposition~\ref{jric} and \eqref{instt2}.
\end{proof}

\subsection{Parallel Nijenhuis tensor}
We obtain from \eqref{npar1}
\begin{prop}\label{npar}
On an ACYT 8-manifold the Nijenhuis tensor is $\sb$-parallel if and only if the Nijenhuis 1-form $\tn$ is $\sb$-parallel.

In particular, on an ACYT 8-manifold with parallel Nijenhuis tensor $\tn$ is  Killing.
\end{prop}
\begin{proof}
We get from \eqref{npar1}  and \eqref{thetN} that $||\sb N||^2=-\sb_iN_{jkl}\ps_{sjkl}\sb_i\tn_s=24||\sb\tn||^2
$.
\end{proof}
Proposition~\ref{npar} and  \eqref{thetN} imply
\begin{cor}
On an ACYT 8-manifold with $\sb$-parallel torsion the Nijenhuis tensor and the (1,2) +(2,1) part $dF^+$ of $dF$ are  $\sb$-parallel.
\end{cor}
Theorem~\ref{thdeltaT}, Corollary~\ref{cornew}, Proposition~\ref{npar} and Proposition~\ref{ricsu4}  
 lead to
\begin{prop}\label{delparN}
On an ACYT 8-manifold with $\sb$-parallel Nijenhuis tensor we have
\[
\begin{split}
 \delta T=d^{\sb}\theta,\quad d\tn=\tn\lrcorner T,  \quad \delta\tn=\delta(J\tn)=g(\tn,\theta) =g(J\tn,\theta)=d\tn\lrcorner F=d(J\tn)\lrcorner F=0;\\ 
 d\theta\lrcorner\ps=Jd\tn-d\tn, \quad d\tn\lrcorner\ps=Jd\theta-d\theta; \quad \delta N=d\theta-Jd\theta-\theta\lrcorner N.%\\
% \delta\tn=\delta(J\tn)=g(\tn,\theta) =g(J\tn,\theta)=d\tn_{ij}F_{ij}=d(J\tn)_{ij}F_{ij}=0.
 \end{split}
 \]
 In particular, $d\theta\in\frak{su}(4)$ if and only if $d\tn\in\frak{su}(4)$.
 \end{prop}
  \begin{cor}
 On a balanced  ACYT 8-manifold with $\sb$-parallel Nijenhus tensor $\delta T=\delta N=0$, $d\tn\in\frak{su}(4)$  and the Ricci tensor of the torsion connection is symmetric and $J$-invariant.
 \end{cor}
% \subsection{Closed torsion}
% We obtain from Proposition~\ref{npar}, Proposition~\ref{delparN} and \eqref{ricdtsu} 
%\begin{cor}\label{closTsu}
%On an ACYT 8-manifold with closed torsion, $dT=0$, the Nijenhuis tensor is $\sb$-parallel, the Ricci tensor satisfies $Ric=-\sb\theta$, $\delta T=d^{\sb}\theta$.
%\end{cor}
%\begin{proof}
%Clearly, $dT=0$ and \eqref{ricdtsu} imply $Ric=-\sb\theta, \quad\sb\tn=0$. %Since $\delta T=d^{\sb}\theta$, we have that $d\theta$ is an (1,1)-form from the second equality of Proposition~\ref{delparN}.  It follows from $d\theta=d^{\sb}\theta+\theta\lrcorner T$ that 
%\[d\theta_{ij}F_{ij}=2\sb_i\theta_j F_{ij}+2\theta(J\theta)=0\]
%because $2\sb_i\theta_j F_{ij}=\delta J\theta=\delta^2F=0$. 
%\end{proof}

 \section{ACYT 8-manifold with closed torsion}
 We recall that an ACYT space with closed torsion  is usually called strong ACYT space. 
 
 Proposition~\eqref{npar} and \eqref{ricdtsu} imply
 \begin{prop}\label{nparclos}
 On a strong ACYT 8-manifold the Nijenhuis tensor is $\sb$-parallel, $\sb N=0$.
 \end{prop}
 Further, we have
\begin{thrm}\label{closT}
Let $(M,g,J,\Psi)$ be a compact   ACYT 8-manifold. %, i.e. a 6-dimensional  manifold with an $SU(3)$-structure  satisfying \eqref{cycon}.
The following conditions are equivalent:
\begin{itemize}
\item[a).] The torsion is closed, $dT=0$.
%\item[b).] The exterior derivative of the Lee form belongs to $\frak{su}(4)$ (i.e. it is $J$-invariant trace-free 2-form) and the Ricci tensor of the torsion connection  is equal to $-\sb\theta$,
%\begin{equation}\label{clos1}d\theta(X,Y)=d\theta(JX,JY), \quad d\theta\lrcorner F=d\theta\wedge\Phi=0,  \qquad Ric=-\sb\theta
%\end{equation}
\vspace{-2mm}
\item[b).] The Nijenhuis tensor is $\sb$-parallel and the two form $F\lrcorner dT$  vanishes, $$\sb N=0, \qquad  dT(X,Y,e_r,Je_r)=(F\lrcorner dT)(X,Y)=0.$$
\vspace{-8mm}
\item[c).] The  Nijenhuis tensor is $\sb$-parallel  and the Ricci tensor of $\sb$  is equal to $-\sb\theta$,
\begin{equation}\label{clos2}\sb N=0, \qquad  \qquad Ric=-\sb\theta
\end{equation}
%\vspace{-3mm}
%\item[c.)] The  Nijenhuis tensor is parallel with respect to the torsion connection and the scalar curvature of the torsion connection  is equal to $\delta\theta$,
%\begin{equation}\label{clos3}\sb N=0, \qquad  \qquad Scal=\delta\theta.
%\end{equation}
\end{itemize}
\end{thrm}
\begin{proof}
Combine \eqref{ricdtsu} with \eqref{dtsu} to get the equivalences of b) and c).

The condition $dT=0$ and  \eqref{ricdtsu} imply $Ric=-\sb\theta$ and $\sb\tn=0$. Hence, b)  and c) follow.% from \eqref{nn1}.

For the converse, assume b) holds. Then \eqref{dtsu} yields  $dT_{ipqr}\Phi_{jpqr}=0.$
The condition $\sb N=0$ implies $\sb\tn=0$ due to Proposition~\ref{npar}. Apply \eqref{ricdtsu}  to get
$dT_{ipqr}\ps_{jpqr}=0$.  It is well known that the last two equalities show that the 4-form $dT$ is self-dual, $*dT=dT$. Indeed, consider the $Spin(7)$ 4-form $\Omega=\ph+\ps$ we have $dT_{ijkl}\Omega_{mjkl}=dT_{ijkl}(\ph_{mjkl}+\ps_{mjkl})=0$ and Proposition~\ref{427} shows $dT$ is self-dual.  Hence, $\delta dT=-*d*dT=-*d^2T=0$. Multiply it with $T$ and integrate over the compact $M$ to get $$0=\int_M<\delta dT,T>vol.=\int_M||dT||^2.vol.$$  Hence, the torsion is closed, $dT=0$. Thus, a) follows from b).
%Indeed, consider the $Spin(7)$ 4-form $\Omega=\ph+\ps$ we have $dT_{ijkl}\Omega_{mjkl}=dT_{ijkl}(\ph_{mjkl}+\ps_{mjkl})=0$ and Proposition~\ref{427} shows $dT$ is self-dual.
%We have
%\[\delta dT=-*d**dT=-*d^2T=0.\]
%Multiply with $T$ and integrate over the compact $M$ to get
%\[0=\int_M(\delta dT,T).vol=\int_M||dT||^2.vol\]Hence $dT=0$ and a) follows from b).
\end{proof}
%As a consequence of Theorem~\ref{closT},  we obtain
\begin{cor}\label{cclosT}
Let $(M,g,J,\Psi)$ be a  compact  CYT 8-manifold. %dimensional  complex manifold with an $SU(3)$-structure  satisfying \eqref{cycon} with $N=0$. T
The following conditions are equivalent:
\begin{itemize}
\item[a).]  The CYT space is pluriclosed, $dT=\partial\bar\partial F=0$.
\vspace{-2.4mm}
\item[b).] The trace of $dT=\partial\bar\partial F$ vanishes, $F\lrcorner dT=F\lrcorner \partial\bar\partial F=0$.
\vspace{-2.4mm}
\item[c).] The Ricci tensor of the Strominger-Bismut  connection satisfies  $Ric=-\sb\theta$.
%\vspace{-2.4mm}
% {\color{blue}\item[d).]The scalar curvature of the Strominger-Bismut connection satisfies $Scal=\delta\theta$.}
 %\begin{equation}\label{clos1c} Ric=-\sb\theta
%\end{equation}
\end{itemize}
\end{cor}
\begin{rmrk}
We remark that CYT spaces satisfying the condition $F\lrcorner dT=0$ are called almost strong in \cite{IP2}. Corollary~\ref{cclosT}  shows that a compact almost strong CYT space in dimension 8 is strong % and  answers positively in dimension 8 to a question posed to the author by Jeffry Streets \cite{JSTR} 
and  Theorem~\ref{closT} extends it to  compact almost strong ACYT 8-manifolds with $\sb$-parallel Nijenuis tensor.
\end{rmrk}

It is well known that a compact CYT space with closed torsion and an exact Lee form is K\"ahler, \cite{GMW,GMPW}. We extend this result to compact strong ACYT 8-manifold for which the Lee form plus or minus the Nijenhuis 1-form is exact .%with the  sum of the Lee form and the Nijenhuis 1-form is an exact form.
\begin{thrm}\label{stexact}
Let $(M,g,J,\Psi)$ be a compact  ACYT 8 manifold  with closed torsion, $dT=0$.

If the Lee form plus or minus the Nijenhuis 1-form is exact, $dT=0,\quad \theta\pm \tn=df, \quad f-smooth,$  then $(M,g,J,\Psi)$ is K\"ahler Calabi-Yau 8-manifold.
\end{thrm}
\begin{proof}
We obtain from \eqref{su2} taking into account \eqref{dtsu} and  \eqref{nnu} that  the next formula holds
\begin{equation}\label{dtsunu}
dT_{ijkl}\ph_{ijkl}=24\d\theta+24||\theta||^2+24||\tn||^2-4||T||^2=24\d(\theta\pm\tn)+24||\theta\pm\tn||^2-4||T||^2
\end{equation}
since  by Proposition~\ref{nparclos} and Proposition~\ref{delparN} we have  $\d\tn=g(\theta,\tn)=0$.

Set $dT=0$ and $\theta\pm\tn=df$ into \eqref{dtsunu} to get
$%\begin{equation}\label{dtsutn1}
||T||^2=6(\Delta f+||df||^2)=-6e^{f}\Delta u,
$ %\end{equation}
where $u=e^{-f}$.
Multiply the latter with $e^{-f}$ and integrate on the compact M to achieve $T=0$.
\end{proof}
Since any closed 1-form on a simply connected space is exact, it follows from Theorem~\ref{stexact} that  a compact simply connected strong ACYT 8 manifold  with  closed 1-form $\theta\pm \tn$ %$%dT=0=d\theta\pm d\tn$  
is K\"ahler Calabi-Yau.

Similarly to the proof of \cite[Theorem~1.1]{IS1}, we derive 
\begin{thrm}\label{closTt}
Let $(M,g,J,\Psi)$ %=\Psi^++\sqrt{-1}\Psi^-)$
be an 8-dimensional  compact  ACYT space with closed torsion, $dT=0$.

The following conditions are equivalent:
\begin{enumerate}
\item[a).] The torsion connection  is Ricci flat, $Ric=0 \Leftrightarrow \sb\theta=0$;
\vspace{-2.4mm}
\item[b).] The norm of the torsion is constant, $d||T||^2=0$;
\vspace{-2.4mm}
%\item[c).] The Lee form is $\nabla-$parallel, $\sb\theta=0$;
%\vspace{-2.4mm}
\item[d).] The torsion connection has vanishing scalar curvature, $Scal=0\Leftrightarrow \delta\theta=0$;
\vspace{-2.4mm}
\item[e).] The Riemannian scalar curvature is a non-negative constant, $Scal^g=\frac14||T||^2=const. \ge 0$;
%\vspace{-2.4mm}
%\item[f).] The Lee form is co-closed, $\delta\theta=0$.
%\vspace{-2.4mm}
%\item[g).] The potential function is a constant, $f=const$
\end{enumerate}
In each of the six cases above the torsion is co-closed and therefore it is  a harmonic 3-form.

%If $dT=0$ and the Riemannian scalar curvature vanishes, $Scal^g=0$ then we have Calabi-Yau space.
\end{thrm}
\begin{proof}
The proof goes like  the proof of \cite[Theorem~1.1]{IS1}. 

\noindent The second Bianchi identity for a metric connection with skew-symmetric torsion is \cite[Proposition~3.5]{IS}
\begin{equation}\label{e1}
d(Scal)_j-2\sb_iRic_{ji}+\frac16d||T||^2_j+\delta T_{pq}T_{pqj}+\frac16T_{pqr}dT_{jpqr}=0.
\end{equation}

We have from \eqref{clos2} and the second formula of \eqref{ricdtsu} that
\begin{equation}\label{newww}
\Delta\delta\theta+\Delta||\theta||^2-\frac16\Delta||T||^2=\Delta\delta\theta-2\theta_s\sb_j\sb_j\theta_s-2||\sb\theta||^2-\frac16\Delta||T||^2=0.
\end{equation}
On the other hand, using \eqref{clos2} and the identity 
\begin{equation}\label{iii}
\sb_i\delta T_{ij}=\frac12\delta T_{ip}T_{ipj}
\end{equation}
 shown in \cite[Proposition~3.2]{IS} for any metric connection with a totally skew-symmetric torsion,
we obtain $$-\sb_iRic_{ji}=\sb_i\sb_j\theta_i=\sb_i\sb_i\theta_j-\sb_id^{\sb}\theta_{ij}=\sb_i\sb_i\theta_j-\sb_i\delta T_{ij}=\sb_i\sb_i\theta_j-\frac12\delta T_{kl}T_{klj}$$
 and  \eqref{e1} takes the form
\[0=\sb_j\delta\theta+2\sb_i\sb_j\theta_i+\delta T_{pq}T_{pqj}+\frac16\sb_j||T||^2=\sb_j\delta\theta+2\sb_i\sb_i\theta_j+\frac16\sb_j||T||^2.
\]
Substitute the last equality into \eqref{newww} to get 
\[
\Delta\Big(\delta\theta-\frac16\Delta ||T||^2\Big)+\theta_j\sb_j\Big(\delta\theta+\frac16||T||^2\Big)=||\sb\theta||^2=||Ric||^2\le 0.
\]
The proof follows from the strong maximum principle as in \cite[Theorem~1.1]{IS1}
\end{proof}
\subsection{Proof of Theorem~\ref{main0}}
First we show 
\begin{thrm}\label{main01}
Let $(M,g,J)$ be a compact 8-dimensional almost Hermitian manifold with totally skew symmetric Nijenhuis tensor $N$ which is parallel with respect to the unique linear connection $\sb$ with totally skew-symmetric torsion $T$ preserving the almost hermitian structure.

If the three Ricci tensors of $\sb$ vanish then the torsion is closed, % {\color{red} with constant norm and $\sb$-parallel Lee form},
\[\sb N=\rho=Ric=\kappa=0 \quad imply \quad dT=0.\]%d||T||^2=\sb\theta=0.\]
\end{thrm}
\begin{proof}
The condition $\rho=0$ yields $(M,g,J)$ is an ACYT 8-manifold. Proposition~\ref{jriccom} implies $\sb\theta=0$ which combined with $Ric=0$ and \eqref{ricdtsu} shows $dT_{ipqr}\ph_{jpqr}=0$. Moreover, because of  Proposition~\ref{npar}  $\sb\tn=0$ and \eqref{ricdtsu} gives $dT_{ipqr}\ps_{jpqr}=0$. The last part of the proof of  Theorem~\ref{closT} implies $\delta dT=0$ and  $dT=0$ since $M$ is compact.
%It is well known that the two  equalities $dT_{ipqr}\ph{jpqr}=dT_{ipqr}\ps_{jpqr}=0$ show that the 4-form $dT$ is self-dual, $*dT=dT$. Indeed, consider the $Spin(7)$ 4-form $\Omega=\ph+\ps$ we have $dT_{ijkl}\Omega_{mjkl}=dT_{ijkl}(\ph_{mjkl}+\ps_{mjkl})=0$ and Proposition~\ref{427} shows $dT$ is self-dual, $*dT=dT$. Therefore,  $$\delta dT=-*d*dT=-*d^2T=0.$$ 
%Multiply the last equality with $T$ and integrate the obtained identity over the compact $M$ implies $$0=\int_M<\delta dT,T>vol.=\int_M||dT||^2.vol.$$ Hence, the torsion is closed, $dT=0$.
\end{proof}
The proof of Theorem~\ref{main0} follow from Theorem~\ref{main01} and Theorem~\ref{closTt}.
 \begin{rmrk}\label{68}
 We note an important difference between the ACYT 6-manifold with closed torsion and the ACYT 8-manifold with closed torsion.
 
 Indeed, according to \cite[Theorem~6.1]{IS1}, on an ACYT 6-manifold with closed torsion the 2-form $d\theta\in\frak{su}(3)$ while $d\theta$ on an ACYT 8-manifold with closed torsion may not be inside   the Lie algebra $\frak{su}(4)$ as the  Example~\ref{exdtit} below shows.
 \end{rmrk}

 \section{$SU(4)$--Instanton}\label{su4instanton}
% Since the torsion connection $\sb$ preserves the $SU(4)$-structure its curvature  lies in the Lie algebra $\frak{su}(4)$, i.e. it satisfies \eqref{rr}.
%\begin{equation}\label{rr}
%\begin{split}
%R_{ijpq}\ps_{pqrd}=R_{ijpq}\sp_{pqrd}=R_{ijpq}F_{pq}=0 \Longleftrightarrow R_{ijpq}\Phi_{pqkl}=-2R_{ijkl}.
%R(X,Y,e_i,e_j)\ps(e_i,e_j,Z)=R(X,Y,e_i,e_j)\sp(e_i,e_j,Z)=R(X,Y,e_i,e_j)F(e_i,e_j)=0 \Longleftrightarrow R(X,Y,e_i,e_j)\Phi(e_i,e_j,Z,V)=-2R(X,Y,Z,V).
%.\\
%R_{ijab}p_{abk}=0 \Longleftrightarrow R_{ijab}ph_{abkl}=-2R_{ijkl}.
%\end{split}
%\end{equation}

The  $SU(4)$--instanton condition means that the curvature 2-form of $\sb$ lies in the Lie algebra $\frak{su}(4)$, i.e. %it is of type (1,1) with zero trace, 
\begin{equation}\label{instt2}
R(JX,JY,Z,V)=R(X,Y,Z,V),\qquad 2\kappa(Z,V)=R(e_i,Je_i,Z,V)=0.
\end{equation}
In other words, the $SU(4)$--instanton condition  can be written in the form
\begin{equation}\label{rr1}
\begin{split}
R_{pqij}\ps_{pqrd}=R_{pqij}\sp_{pqrd}=R_{pqij}F_{pq}=0 \Longleftrightarrow R_{pqij}\Phi_{pqkl}=-2R_{klij}.
\end{split}
\end{equation}
%The next result  is a consequence of the instanton condition
\begin{cor}\label{jric}
On an 8-dimensional ACYT space with $SU(4)$--instanton torsion connection the Ricci tensor is $J$-invariant, $Ric(X,Y)=Ric(JX,JY)$.

If, in addition,  $\delta^{\sb}N=0$ then the conclusions \eqref{ncoclosed}   of Corollary~\ref{ncorickap} hold true.
\end{cor}
\begin{proof}
We obtain using \eqref{instt2} $Ric(X,Y)=R(Je_i,JX,Y,e_i)=-R(e_i,JX,Y,Je_i)=Ric(JX,JY)$.

 Corollary~\ref{ncorickap} completes the proof.
\end{proof}
%We obtain  from the proof of Proposition~\ref{jric} 
%\begin{cor}\label{kapric}
%On an ACYT 8-manifold the conditions $$\delta^{\sb}N=\kappa= Ric(X,Y)-Ric(JX,JY)=0\quad  imply\quad   \eqref{ncoclosed} \quad hold \quad true.$$ 
%\end{cor}
\begin{prop}\label{istdelN}
On  an ACYT 8-manifold manifold with  $SU(4)$-instanton torsion connection, we have
\begin{equation}\label{nnnew}
\delta^{\sb}N_{ij}=\frac1{12}\Big[d^{\sb}T_{ipqr}\Phi_{jpqr}-d^{\sb}T_{jpqr}\Phi_{ipqr}\Big]=-\frac19 \Big[\sigma^T_{ipqr}\Phi_{jpqr}-\sigma^T_{jpqr}\Phi_{ipqr}  \Big].
\end{equation}
In particular, each of the  three two forms $(F\lrcorner dT)$,$( F\lrcorner d^{\sb}T),( F\lrcorner \sigma^T)$ is $J$ invariant if and only if $\delta^{\sb}N=0$,
\end{prop}
\begin{proof}
We express the Ricci tensor in two ways using \eqref{rr} ,\eqref{1bi1}, \eqref{1bi}, \eqref{tit} and \eqref{dh}  as follows:
\begin{multline}\label{ricdt}
\begin{split}
Ric_{ij}=\frac12R_{ipqr}\Phi_{jpqr}=\frac16\Big[R_{ipqr}+R_{iqrp}+R_{irpq} \Big]\Phi_{jpqr}=%\frac1{12}dT_{ipqr}\Phi_{jpqr}+\frac16\sb_iT_{pqr}\Phi_{jpqr}\\
%=\frac1{12}dT_{ipqr}\Phi_{jpqr}-\sb_i\theta_j=
\Big[\frac1{12}d^{\sb}T_{ipqr}+\frac16\sigma^T_{ipqr}\Big]\ph_{jpqr}-\sb_i\theta_j;\\
-Ric_{ji}=\frac12R_{pqri}\Phi_{jpqr}=\frac16\Big[ R_{pqri}+R_{qrpi}+R_{rpqi} \Big]\Phi_{jpqr}%=\frac16\Big[dT_{pqri}-\sigma^T_{pqri}+\sb_iT_{pqr}\Big]\Phi_{jpqr}\\
=\Big[-\frac16d^{\sb}T_{ipqr}-\frac16\sigma^T_{ipqr}\Big]\Phi_{jpqr}-\sb_i\theta_j.
\end{split}
\end{multline}
Take the sum of the two equations in \eqref{ricdt}   using \eqref{rics},   Proposition~\ref{thdeltaT} to get
\begin{equation}\label{ins4}
\begin{split}
-d^{\sb}\theta_{ij}-\frac12\delta^{\sb}N_{ij}=-\delta T_{ij}=Ric_{ij}-Ric_{ji}%=\frac16\Big[-\frac12dT_{ipqr}+\sigma^T_{ipqr}\Big]\Phi_{jpqr}-2\sb_i\theta_j\\
=-\frac1{12}d^{\sb}T_{ipqr}\Phi_{jpqr}-2\sb_i\theta_j.%=-\frac14\sb_pT_{qri}\Phi_{pqrj}-\frac32\sb_i\theta_j.
\end{split}
\end{equation}
The skew symmetric part of \eqref{ins4} proves the first equality in \eqref{nnnew}. The second one follows from \eqref{dTN}, \eqref{dh} and the already established first equality.

The last claim follow from \eqref{nnnew} with applications of \eqref{dtsu}. 
\end{proof}
\begin{prop}
If on an ACYT 8-manifold  the  curvature of the torsion  connection $\sb$    is an $SU(4)$-instanton then the next equalities hold true
\begin{equation}\label{bi24}
\sb_iR_{iplm}=\theta_rR_{rplm}, \quad \tn_rR_{rplm}=0.
\end{equation}
\end{prop}
\begin{proof}
The second Bianchi identity for $\sb$ reads (see e.g. \cite{KN})
\begin{equation}\label{bi22}
\sb_iR_{jklm}+\sb_jR_{kilm}+\sb_kR_{ijlm}+T_{ijs}R_{sklm}+T_{jks}R_{silm}+T_{kis}R_{sjlm}=0.
\end{equation}
Multiplying \eqref{bi22} with $\ph_{ijkp}$ and using the $SU(4)$-instanton conditions \eqref{rr1}, we obtain
\begin{equation}\label{bi23}
-6\sb_iR_{iplm}+3T_{ijs}R_{sklm}\ph_{ijkp}=0.
\end{equation}
An application of  \eqref{ntor1} together with \eqref{rr1} to the second term in \eqref{bi23} yields
\begin{multline*}
T_{ijs}\ph_{ijkp}R_{sklm}=\Big[-T_{ijk}\ph_{ijps}-T_{ijp}\ph_{ijsk}-2T_{skp}-2\theta_r\ph_{rskp}-\tn_r\ps_{rskp}  \Big]R_{sklm}\\
=-T_{ijk}\ph_{ijps}R_{sklm}+2T_{ijp}R_{ijlm}-2T_{skp}R_{sklm}+4\theta_rR_{rplm}=-T_{ijs}\ph_{ijkp}R_{sklm}+4\theta_rR_{rplm}.
\end{multline*}
The latter,  written as
$%\begin{equation}\label{m2}
T_{ijs}\ph_{ijkp}R_{sklm}=2\theta_rR_{rplm} , 
$ %\end{equation}  
 substituted  into \eqref{bi23} proves  the first part of \eqref{bi24}.

To show the second one we consider the $Spin(7)$ structure $\om$ defined by \eqref{suspin}, $\om=\ph+\ps$. Then $\sb\om=0$ and $\sb$ is the unique $Spin(7)$ connection with torsion 3-form $T$. Since $R$ is an $SU(4)$ instanton it follows that $R$ is also $Spin(7)$ instanton for $\om$. Then \cite[(7.97)]{IPU} reads $\sb_iR_{iplm}=\frac76\tsp_rR_{rplm}$ which combined with \eqref{Leefs} gives $\sb_iR_{iplm}=(\theta_r+\tn_r)R_{rplm}$. The already established first equality of \eqref{bi24} completes the proof. 
\end{proof}

In the compact case, we derive
\begin{prop}\label{jriccom1}
Let $(M,g,J)$ be a compact ACYT 8-manifold with $SU(4)$ instanton torsion connection. 

If the Nijenhuis tensor is $\sb$-coclosed, $\delta^{\sb}N=0$, then the Lee form is $\sb$-parallel, $\sb\theta=0$, $\delta T=0$ and the Ricci tensor is  symmetric and $J$-invariant given by $Ric_{ij}=\frac1{12}dT_{ipqr}\ph_{jpqr}=\frac16\sigma^T_{ipqr}\ph_{jpqr}$.
\end{prop}
\begin{proof}
It  follows from Proposition~\ref{jriccom}, Proposition~\ref{jric} and \eqref{instt2} that $\sb\theta=0$. Then \eqref{ins4} yields $\delta T=0=F\lrcorner d^{\sb}T$. Apply \eqref{dtsu} to  $d^{\sb}T$ to get  $d^{\sb}T_{ipqr}\ph_{jpqr}=0$. Now,  \eqref{ricdt} and Proposition~\ref{ricsymj} complete the proof.
\end{proof}
\begin{cor}
Let $(M,g,J)$ be a compact CYT 8-manifold with $SU(4)$ instanton torsion connection.  Then the Lee form is $\sb$-parallel, $\sb\theta=0$, $\delta T=0$ and the Ricci tensor is  symmetric and $J$ invariant given by $Ric(X,Y)=(F\lrcorner\sigma^T)(X,JY)$.
\end{cor}
%We already know that if the torsion is closed, $dT=d^{\sb}T+2\sigma^T=0$ then the Nijenhuis tensor is $\sb$-parallel. Now we show
\subsection{Instanton and $\sb$-closed torsion}
The condition $d^{\sb}T=0$ is equivalent to $dT=2\sigma^T$ because of \eqref{dh}.
\begin{prop}\label{nparins}
If on an ACYT 8-manifold the torsion connection is an $SU(4)$ instanton and $d^{\sb}T=0$ then the Nijenhuis tensor is $\sb$-parallel.
\end{prop}
\begin{proof}
Using the instanton condition, we multiply the curvature with $\ps$  applying  \eqref{rr} ,\eqref{1bi1}, \eqref{1bi}, \eqref{thetN} and \eqref{dh} to get
\begin{equation}\label{ricdtN}
\begin{split}
0=6R_{ipqr}\ps_{jpqr}=2\Big[R_{iabc}+R_{ibca}+R_{icab} \Big]\ps_{jabc}%=\frac1{6}dT_{iabc}\ps_{jabc}+\frac13\sb_iT_{abc}\ps_{jabc}\\
=dT_{ipqr}\ps_{jpqr}-12\sb_i\tn_j;\\
0=6R_{pqri}\ps_{jpqr}=2\Big[ R_{abci}+R_{bcai}+R_{cabi} \Big]\ps_{jabc}=2\Big[dT_{abci}-\sigma^T_{abci}\Big]\ps_{jabc}-12\sb_i\tn_j.%=-2dT_{ipqr}\ps_{jpqr}+2\sigma^T_{ipqr}\ps_{jpqr}-12\sb_i\tn_j.
\end{split}
\end{equation}
The sum of the two equations in \eqref{ricdtN}  gives with the help of \eqref{dh} %and  using \eqref{rics} and Proposition~\ref{delparN} to get
$%\begin{equation}\label{ins4N}
%\begin{split}
%-d^{\sb}\theta_{ij}+\frac12(d\theta_{ij}-d\theta_{JiJj})=-\delta T_{ij}=Ric_{ij}-Ric_{ji}=\frac16\Big[-\frac12dT_{iabc}+\sigma^T_{iabc}\Big]\Phi_{jabc}-2\sb_i\theta_j\\
0=-d^{\sb}T_{pqri}\ps_{pqrj}-24\sb_i\tn_j.%=-3\sb_pT_{qri}\ps_{pqrj}-18\sb_i\tn_j.%=-\frac14\sb_aT_{bci}\Phi_{abcj}-\frac32\sb_i\theta_j.
%\end{split}
$%\end{equation}
%\[R_{pqri}-R_{ripq}=\sb_pT_{qri}-\sb_qT_{pri}-\sb_rT_{pqi}+\sb_iT_{pqr}\] 

If $d^{\sb}T=0$ we obtain $\sb\tn=0$.% from \eqref{ins4N}.
\end{proof}

\begin{thrm}\label{maing2}
Let $M$ be  a compact  ACYT  8-manifold. The next two conditions are equivalent:
\begin{itemize}
\item[a)] The torsion 3-form is parallel with respect to the torsion connection, $\sb T=0$.
\vspace{-2.4mm}
\item[b)] The  curvature of the torsion  connection $\sb$   is an $SU(4)$-instanton and $d^{\sb}T=0$.%\eqref{inst6} holds true.
\end{itemize}
In both cases the Nijenhuis tensor is $\sb$-parallel and  the  Ricci tensor of the torsion connection is symmetric, $\sb$-parallel  and $\sb dT=0$.
\end{thrm}
\begin{proof}
If $\sb T=0$ then clearly $d^{\sb}T=\delta T=\sb\theta=d(Scal)=\sb\tn=0$. Consequently, $\sb N$=0 by Proposition~\ref{npar}.  
Moreover, \eqref{4form} shows that the  curvature $R\in S^2\Lambda^2$ and therefore $R$ is a $SU(4)$ instanton since $\sb\ph=0$ which proves b).

Let b) hods.   Then $\sb N=0$ by Proposition~\ref{nparins}. 

We multiply \eqref{bi24} with $T_{plm}$ using \eqref{1bi1},  the conditions $dT=2\sigma^T$,  $d^{\sb}T=0$ and \eqref{sigmatt}  to calculate 
\begin{multline}\label{bi25}
0=3\Big[\sb_iR_{iplm}-\theta_iR_{iplm}\Big]T_{plm}=\sb_i\Big[-\sigma^T_{plmi}+\sb_iT_{plm}\Big]T_{plm}
-\theta_i\Big[-\sigma^T_{plmi}+\sb_iT_{plm}\Big]T_{plm}\\
=-\sb_i\sigma^T_{plmi}T_{plm}+T_{plm}\sb_i\sb_iT_{plm}-\frac12\sb_{\theta}||T||^2=\sigma^T_{plmi}\sb_iT_{plm}+T_{plm}\sb_i\sb_iT_{plm}-\frac12\sb_{\theta}||T||^2\\=
\frac14\sigma^T_{plmi}d^{\sb}T_{iplm}+T_{plm}\sb_i\sb_iT_{plm}-\frac12\sb_{\theta}||T||^2=T_{plm}\sb_i\sb_iT_{plm}-\frac12\sb_{\theta}||T||^2.
\end{multline}
On the other hand, we calculate  the Laplacian  $\Delta||T||^2$ of the norm of the torsion,
\begin{equation}\label{lapt}
-\frac12\Delta||T||^2=\frac12\LC_i\LC_i||T||^2=\frac12\sb_i\sb_i||T||^2=T_{plm}\sb_i\sb_iT_{plm}+||\sb T||^2.
\end{equation}
A substitution of \eqref{lapt} into \eqref{bi25} yields
\begin{equation}\label{lap}
\Delta||T||^2+\sb_{\theta}||T||^2=-2||\sb T||^2\le 0.
\end{equation}
Since $M$ is compact  we may apply  the strong maximum principle to \eqref{lap} (see e.g. \cite{YB,GFS}) to achieve  $\sb T=0$ which completes the proof of the theorem.
\end{proof}
As a consequence of the proof of Theorem~\ref{maing2}, we obtain from \eqref{lap}
\begin{cor}\label{main3}
Let $M$ be  an ACYT 8-manifold. The next two conditions are equivalent:
\begin{itemize}
\item[a)] The torsion 3-form is parallel with respect to the torsion connection, $\sb T=0$.
\vspace{-2.4mm}
\item[b)] The  curvature   $R$   is an $SU(4)$-instanton, $d^{\sb}T=0$ and the torsion is of constant norm, $d||T||^2=0$. %and \eqref{inst6} holds true.
\end{itemize}
\end{cor}
Note that Theorem~\ref{maing2} and  Corollary~\ref{main3} %and Proposition~\ref{insp00} 
also follow from \cite[Theorem~1.5,  Corollary~7.4]{IPU} applied to the $Spin(7)$ structure $\om=\ph+\ps$. 
\begin{prop}\label{insp00}
Let $M$ be  an ACYT 8-manifold manifold.%  with closed Lee form, $d\theta=0$. 

If the Lee form is closed, the torsion  connection $\sb$  is an $SU(4)$-instanton and $d^{\sb}T=0$,
\[d\theta=0,\qquad R\in \frak{su}(4)\otimes\frak{su}(4), \qquad dT=2\sigma^T,\]
then the torsion 3-form is parallel with respect to the torsion connection, $\sb T=0$.

In this case the  Ricci tensor of the torsion connection is symmetric, $\sb$-parallel  and $\sb dT=0$.
\end{prop}
\begin{proof}

Observe that  the  conditions $d^{\sb} T=d\theta=0$ together with Proposition~\ref{nparins},  Proposition~\ref{delparN} , Proposition~\ref{jric},  \eqref{ins4} and \eqref{ricdt} imply 
\begin{equation}\label{nabdt0}
\sb N=0,\quad \delta T=2\sb\theta=-\theta\lrcorner T\in\frak{su}(4);\quad Ric_{ij}=\frac16\sigma^T_{ipqr}\ph_{jpqr}-\sb_i\theta_j=\frac12T_{ips}T_{sqr}\ph_{jpqr} -\sb_i\theta_j.
%,\quad  \delta T_{ij}\theta_j=\theta_j\sb_i\theta_j=\frac12\sb_i||\theta||^2=\frac12\theta_sT_{sij}\theta_j=0
\end{equation}
The first identity in \eqref{nabdt0}  shows that the Lee vector field is Killing since $2\sb\theta=\delta T$ is skew-symmetric and $d\theta=0$ implies $\LC\theta=0$. In particular $\delta\theta=d||\theta||^2=0$.

We evaluate the second term in \eqref{e1} using \eqref{bi24} and the second equality in \eqref{nabdt0} as follows
\begin{equation}\label{rrrr}
\sb_iRic_{ji}=\sb_i(Ric_{ij}+\delta T_{ij})=\theta_iRic_{ij}+2\sb_i\sb_i\theta_j=\frac12\sb_j||\theta||^2+2\sb_i\sb_i\theta_j,
\end{equation}
where we  use  \eqref{ntor1} and $\theta\lrcorner T\in\frak{su}(4)$ to obtain
\begin{equation}\label{part}
\begin{split}
\theta_pT_{jsl}T_{lmp}\ph_{jsmk}
=-\Big(T_{klm}T_{lmp}+\frac12T_{jsk}\ph_{jslm}T_{lmp}+\theta_q\ph_{qklm}T_{lmp}+\tn_q\ps_{qklm}T_{lmp}\Big)\theta_p=0.
\end{split}
\end{equation}
The last term in \eqref{e1} vanishes because $dT=2\sigma^T$ and \eqref{sigmatt}.% the identity  $T_{pqr}\sigma^T_{pqrj}=0$ observed in \cite{IS}.

\noindent For the fourth term we use   \eqref{iii} to get 
 $\delta T_{pq}T_{pqj}=2\sb_i\delta T_{ij}=4\sb_i\sb_i\theta_j.$ 
Setting $\delta\theta=d||\theta||^2=d||\tn||^2=0$ into \eqref{ricdtsu} yields $d(Scal)=-\frac13||T||^2$ and \eqref{e1} leads to $d||T||^2=0$.
%\begin{equation*}%\label{scalcon}
%0=3d\delta\theta+2d||\theta||^2-\frac13d||T||^2-d||\theta||^2+\frac16d||T||^2=-\frac16d||T||^2.
%\end{equation*}
Corollary~\ref{main3} completes the proof.
\end{proof}

\begin{rmrk}
We remark that the converse in Proposition~\ref{insp00} is not true. We construct in Example~\ref{exdtit} an ACYT 8-manifold having $\sb$-parallel Nijenhuis tensor and $\sb$-parallel closed torsion  with non-closed Lee form which does not belongs to $\frak{su}(4)$, $d\theta\not\in\frak{su}(4)$ which also supports Proposition~\ref{delparN}.
\end{rmrk}
\begin{exam}\label{exdtit}
The group $G=S^1\times SU(2)\times SU(2)\times S^1$ has a Lie algebra $g=\mathbb R\oplus\frak{su}(2)\oplus \frak{su}(2)\oplus\mathbb R$ and structure equations   
\[de^0=0, \quad de_1=e_{23},\quad de_2=e_{31},\quad de_3=e_{12},\quad de_4=e_{56},\quad de_5=e_{64}\quad de_6=e_{45}, \quad de_7=0.\]
Define a left-invariant  $SU(4)$ structure as follows
\[
\begin{split}
F=-e_{12}-e_{34}-e_{05}-e_{67};\\
\ps=e_{0136}-e_{1456}-e_{1357}-e_{0147}-e_{0237}-e_{0246}-e_{2356}+e_{2457}.
\end{split}
\]
This  left-invariant $SU(4)$ structure in not integrable,   generates the bi-invariant metric and the torsion  connection preserving the $SU(4)$ structure is the  left invariant  flat  Cartan  connection with closed  and $\sb$ parallel torsion $T=-[.,.]=e_{123}+e_{456}$. The Lee form  $\theta=e_4$ is not closed, $d\theta=e_{56} \not\in\frak{su}(4)$, the Nijenhuis 1-form $\tn=e_1$ is also not closed and $N=-e_1\lrcorner\ps$. It is easy to verify all the claims of Proposition~\ref{delparN} hold true.
\end{exam}
\subsection{Instanton and closed torsion}
%We recall that an ACYT space with closed torsion  is usually %called strong ACYT space.
%First  observe from the second line of \eqref{ricdtsu} and %Proposition~\ref{npar} the following
%%\begin{cor}\label{closTsu}
%On a strong ACYT 8-space the Nijenhuis tensor is $\sb$-parallel,  ^$dT=0\Rightarrow \sb N=0$. 
%\end{cor}
We begin with 
\begin{thrm}\label{clossT}
Let $(M,g,J,\Psi)$ be a  compact  strong ACYT 8-manifold i.e. the torsion is closed. 

The next conditions are equivalent
\vspace{-2mm}
\begin{itemize}
\item[a)] The torsion is $\sb$-parallel.
\vspace{-2.4mm}
\item[b)] The curvature of $\sb$ is an $SU(4)$ instanton.
\vspace{-2.4mm}
%\item[c)] The torsion is $\LC$-parallel.
\item[c)] The curvature of $\sb$ satisfies the Riemannian first Bianchi identity.
\end{itemize}
\end{thrm}
\begin{proof}
If $dT=0$ then $\sb N=0$ by Proposition~\ref{nparclos}.

If $\sb T=0$ then $R\in S^2\Lambda^2$ by \cite[Lemma~3.4]{I} which shows that $R$ is an $SU(4)$instanton.

For the converse, let $R$ be an $SU(4)$instanton. Then \eqref{bi24} holds true.

We multiply \eqref{bi24} with $T_{plm}$, using \eqref{1bi1}  and $dT=0$, to calculate 
\begin{multline}\label{bi25cl}
0=3\Big[\sb_iR_{iplm}-\theta_iR_{iplm}\Big]T_{plm}=\sb_i\Big[-\frac12dT_{plmi}+\sb_iT_{plm}\Big]T_{plm}
-\theta_i\Big[-\frac12dT_{plmi}+\sb_iT_{plm}\Big]T_{plm}\\
=T_{plm}\sb_i\sb_iT_{plm}
-\theta_i\sb_iT_{plm}T_{plm}=T_{plm}\sb_i\sb_iT_{plm}-\frac12\sb_{\theta}||T||^2.
\end{multline}
A substitution of \eqref{lapt} into \eqref{bi25cl} yields that \eqref{lap} holds true.
%\begin{equation}\label{lap}
%\Delta||T||^2+\sb_{\theta}||T||^2=-2||\sb T||^2\le 0.
%\end{equation}
We  apply  the strong maximum principle to \eqref{lap} on the compact $M$ (\cite{YB,GFS}) to achieve  $\sb T=0$ and get the equivalences between a) and b).

Let $\sb T=0$. Then \eqref{dh} yields $\sigma^T=0$ since $dT=0$. The equivalence of a) and c) follows from Theorem~\ref{tFBI}. 
Theorem~\ref{clossT} is proved.
\end{proof}
We note that the equivalences of a) and b) of Theorem~\ref{clossT} follows also from \cite[Theorem~8.2]{IPU} applied to the $Spin(7)$ structure $\om=\ph+\ps$.

%An immediate consequence of Theorem~\ref{clossT} is the next

\subsection{Proof of Theorem~\ref{main0inst} and Corollary~\ref{cclosTc}}
% \begin{cor}\label{jriccc}
%On a compact ACYT 8-manifold with $SU(4)$ instanton torsion connection if the Ricci tensor vanishes and the Nijenhuis tensor is $\sb$-parallel then $dT=\sb T=\LC T=0$ and the space satisfies the first Bianchi identity.
% \end{cor}
 The proof follows from Theorem~\ref{main0} and Theorem~\ref{clossT}.
 
\subsection{ACYT 8-manifold with a Killing torsion}
The notion of Killing forms was introduced by Yano in \cite{Yan}, where he already noted that a
p-form $\alpha$ is a Killing form if and only if for any geodesic $\gamma$ the (p-1)-form $\dot{\gamma}\lrcorner\alpha$ is parallel
along $\gamma$. In particular, Killing forms define quadratic first integrals of the geodesic equation, i.e.
functions, which are constant along geodesics.
We recall that Killing form on a Riemannian manifold is a differential $p$-form $\alpha$ whose covariant derivative is totally skewsymmetric, $d\alpha=4\LC\alpha$. It was  known that on compact K\"ahler manifolds Killing $p$-forms with $p\ge 2$ are parallel \cite{Yam}. U. Semmelmann   proved in \cite{Uve}  that on a compact manifold with Riemannian holonomy $G_2$ or $Spin(7)$ any Killing form has to be parallel, $\LC\alpha=0$, hence closed. We recall  that  on a Riemannian manifold with  holonomy $G_2$ or $Spin(7)$ the Ricci tensor is zero due to \cite{Bo}.

For a Killing torsion 3-form, these results was extended to compact ACYT 6-manifolds with vanishing Ricci tensor \cite{IS1}, integrable $G_2$ manifolds of constant type with zero Ricci tensor \cite{IS2} and on compact $Spin(7)$ manifolds with vanishing Ricci tensor and closed Lee form \cite{Iap}.

Here we show that a similar result holds for compact ACYT 8-manifolds.

We recall that the torsion 3-form to be  Killing is equivalent to the curvature condition $R\in S^2\Lambda^2$   due to \cite[Lemma~3.4]{I}. It follows  that if $R\in S^2\Lambda^2$ then $R$ is an $SU(4)$ instanton,   the Ricci tensor is  $J$ invariant due to Proposition~\ref{jric} and  symmetric because $Ric(X,Y)=R(e_i,X,Y,e_i)=R(Y,e_i,e_iX)=Ric(Y,X).$  

Proposition~\ref{ricsymj} shows the validity of the next
\begin{prop}\label{mainsu}
Let $(M,g,J,\Psi)$ be an 8-dimensional  ACYT  manifold 
 and  the torsion connection $\sb$   has curvature $R \in S^2\Lambda^2$ i.e. \eqref{r4} holds,  equivalently, $T$ is a Killing 3-form .

Then  the vector field dual to the 1-form $J\theta$ is a  Killing vector field.

If, in addition, the Nijenhuis tensor 
% then the functions $\lambda$ and $\mu$ are constants,
%then the Nijenhuis tensor
is $\sb-$coclosed, $\delta^{\sb} N=0$, then $\sb\theta$  is symmetric and $J$-invariant.
%\begin{equation}\label{symJtheta}
%d\lambda=d\mu=0, \qquad
% (\sb_X\theta)Y=(\sb_Y\theta)X=(\sb_{JX}\theta)JY.
%\end{equation}
%In this case $\theta\lrcorner T=d\theta\in\frak{su}(4)$ and $\tn\lrcorner T=d\tn\in\frak{su}(4)$.
\end{prop}
Theorem~\ref{main0inst} yields
\begin{thrm}\label{KillT}
Let $(M,g,J,\Psi)$ be a compact Ricci flat  ACYT 8-manifold %with a $SU(3)$-structure $(g,J,\Psi=\Psi^++\sqrt{-1}\Psi^-)$
%satisfying the condition \eqref{s2l2}. %{\color{red} If the Nijenhuis tensor is of constant norm, $||N||^2=const.$}
with $\sb$-parallel Nijenhuis tensor  and a Killing torsion. % curvature $R \in S^2\Lambda^2$, 
%\begin{equation}\label{s2l2}
% R(X,Y,Z,V)=R(Z,V,X,Y),\quad Ric(X,Y)=0, \quad \sb N=0.
%\end{equation}
% Then the Lee form $\theta$ is $\sb-$parallel, $\sb\theta=0.$  In particular, the Lee form is co-closed, $\delta\theta=0$.

Then   the torsion is $\LC$-parallel and $\sb$-parallel, $\LC T=\sb T=0$. 
  
In particular, $dT=\delta T=0$  and  the curvature of $\sb$ satisfies the Riemannian first Bianchi identity.
 \end{thrm}

% \end{proof}

\begin{prop}\label{sbtheta1}
On a compact  ACYT 8-manifold with a Killing torsion and $\sb$-parallel Nijenhuis tensor %, i.e. has curvature $R \in S^2\Lambda^2$.
%\begin{itemize}
%\item [a)] If   the Lee form is closed and the Nijenhuis tensor is $\sb$-coclosed, $d\theta=\delta^{\sb}N=0$,  then  the norm of the torsion is a constant, $d||T||^2=0.$
%and  the Nijenhuis tensor is $\sb$-parallel, $\sb N=0$ then  
the norm of the torsion is a constant, $d||T||^2=0.$
%\end{itemize}
\end{prop}
\begin{proof}
We know from $\sb N=0$ that $\sb\tn=0$ and $\delta T=0$ since the Ricci tensor is symmetric.  %Proposition~\ref{npar}  
Proposition~\ref{jriccom} implies  $\sb\theta=0$.  The Ricci identity for $\sb$ yields
\[0=\sb_i\sb_j\theta_k-\sb_j\sb_i\theta_k=-R_{ijks}\theta_s- T_{ijs}\sb_s\theta_k=-R_{ijks}\theta_s, \quad Ric_{is}\theta_s=0=Ric_{si}\theta_s
\]
%Now, \eqref{nnid}  yields $\theta\lrcorner T=0$ since $d\theta=d^{\sb}\theta=0$.
The condition $R\in S^2\Lambda^2$ inply  %is equivalent to $\sb T$ to be a 4-form because of \eqref{4form}, i.e. the torsion 3-form is a Killing form. Moreover, the Ricci tensor of $\sb$ is symmetric and 
 $R$ is an $SU(4)$-instanton. Therefore, \eqref{bi24} holds true and we obtain 
 %$$\sb_iRic_{si}=0.$$
%Substitute \eqref{dh} into \eqref{ricdt} to get using \eqref{liff}  and \eqref{ins4} that
%\begin{equation}\label{su1}
%\begin{split}
%Ric_{ij}=\frac1{12}(4\sb_iT_{pqr}+2\sigma^T_{ipqr})\ph_{jpqr}%=-3\sb_i\theta_j+\frac1{6}\sigma^T_{iabc}\ph_{jabc}
%=\frac12T_{ips}T_{sqr}\ph_{jpqr}.%\\
%{\color{red} 0=Ric_{ij}+\sb_i\theta_j-\frac1{12}(\sb_iT_{abc}-3\sb_aT_{bci}+2\sigma^T_{iabc})\ph_{jabc}=Ric_{ij}+3\sb_i\theta_j-\frac1{6}\sigma^T_{iabc}\ph_{jabc}-\delta T_{ij}}
%\\{\color{blue} 2\sb_iRic_{ij}=-6\sb_i\sb_i\theta_j+\frac13\sb_i\sigma^T_{iabc}\ph_{jabc}=-6\sb_i\sb_i\theta_j-T_{ias}\sb_sT_{ibc}\ph_{jabc}}\\=-6\sb_i\sb_i\theta_j+
%T_{ias}\sb_s\Big(T_{ija}+\frac12 T_{jbc}\ph_{aibc}+\frac12 T_{abc}\ph_{ijbc}+\theta_t\ph_{tija}\Big)
%\end{split}
%\end{equation}
%Note that \eqref{part} holds true. 
%We obtain from \eqref{bi24} and \eqref{su1}  that
\begin{equation}\label{su22}
\sb_i Ric_{ij}=\theta_i Ric_{ij}=0.% \quad since \quad \theta\lrcorner T=0.
\end{equation}
The Killing torsion condition, \eqref{4form}, \eqref{tsym} and \eqref{dh} yield
\begin{equation}\label{su23} T_{pqr}dT_{jpqr}=2\sb_j||T||^2,
\end{equation}
where we used \eqref{sigmatt} %the identity $T_{pqr}\sigma^T_{pqrj}=0$ observed in \cite[Proposition~3.1]{IS} 
and that $\sb T$ is a 4-form.

The second equation in \eqref{ricdtsu}, \eqref{su22}, \eqref{su23}, $\sb\theta=\delta\theta=\sb\tn=0$  and \eqref{e1}  imply
$%\begin{equation}\label{e1su}
\sb_j||T||^2=0.
$%\end{equation}
\end{proof}

 The next example shows the existence of a compact ACYT 8-manifold with $\sb$-parallel  torsion, therefore $\sb$-parallel Nijenhuis tensor due to Proposition~\ref{nparins} and $SU(4)$-instanton curvature,  with non-vanishing Ricci tensor and non-closed torsion.
%\subsection{Example}
%We take the next example from \cite{II} which supports Theorem~\ref{spthm1}.
\begin{exam}\label{acytnoflat}
Let $G$ be the 8-dimensional connected simply connected and nilpotent
 Lie group, determined by the left-invariant 1-forms $\{e_1,\dots,e_8\}$
 such that
\begin{gather}\label{in1}
de_2=de_3=de_6=de_7=de_8=0,\\\nonumber
de_1=e_3\wedge e_6,\quad
de_4= e_2\wedge e_6, \quad
de_5= e_2\wedge e_3.
\end{gather}
Define the metric on $G\cong \mathbb R^8$  by
$g=\sum_{i=1}^8e_i^2$.
Let $(F,\ps)$ be the $SU(4)$-structure on $G$ given by 
\begin{equation*}
\begin{split}
F^1 =-e_{12} - e_{34}-e_{56}-e_{78}, \\
\ps=e_{1357}-e_{1368} - e_{2457}+e_{2468}-e_{1458} - e_{2367}-e_{2358}-e_{1467}.
\end{split}
\end{equation*}
It is easy to verify using  \eqref{cy1} and \eqref{in1} that
\begin{gather}\label{in2}
dF=-3e_{236}, \quad \theta=0, \quad  N=e_{136}+e_{145}+e_{235}-e_{246},\quad \tn=\frac14N\lrcorner\ps=-e_8, \quad N=e_8\lrcorner\ps,\\
d\ps=-e_{34568}-e_{12568}-e_{12348}=\theta\wedge\ps+\tn\wedge\ph.
\end{gather}
Hence, $(G,\Psi,g,J)$ is neither complex nor Nearly K\"ahler manifold but it fulfills
the conditions \eqref{cyconsu}  of Theorem~\ref{cythm1} and
therefore there exists an $SU(4)$ connection with torsion 3-form on $(G,F,g,\ps)$.
The expression
\eqref{cyt8} and \eqref{in1} give
\begin{equation}\label{tor}
T=-2e_{145}+e_{136}+e_{235}-e_{246}, \quad dT=-2(e_{1256}+e_{3456}+e_{1234})=-2\tn\lrcorner d\ps.
\end{equation}
The equalities  \eqref{tor} and  \eqref{cy2} imply  the nonzero essential terms
of the torsion connection are 
\begin{gather}\label{tor1}
\nabla_{e_1}e_6=-e_3, \qquad  \nabla_{e_5}e_2= e_3, \qquad \nabla_{e_4}e_6=-e_2, \\\nonumber
\nabla_{e_4}e_5=-e_1, \qquad \nabla_{e_5}e_1=-e_4, \qquad \nabla_{e_1}e_4=-e_5.
\end{gather}
It is easy to verify  that   the torsion tensor $T$ as well as the Nijenhuis tensor $N$ and $dF$ are parallel with respect to the torsion connection $\nabla$. Consequently, the curvature of the torsion connection satisfies \eqref{r4} and it is an $SU(4)$-instanton. Moreover, according to \eqref{ricdtsu} and \eqref{tor} we have \[Ric_{ij} =\frac1{12}dT_{iabc}\ph_{jabc} -\sb_i\theta_j\not=0, \quad Ric_{11}=Ric_{22}=Ric_{33}=Ric_{44}=-2.\]

The coefficients of the structure equations of the Lie algebra given by \eqref{in1}
are integers. Therefore, the well-known theorem of Malcev \cite{Mal} states that
the group $G$ has a uniform discrete subgroup $\Gamma$ such that $Nil^8=G/\Gamma$ is a
compact 8-dimensional nil-manifold. The $SU(4)$-structure, described above, descends to $Nil^8$
and therefore we obtain a compact example.

\end{exam}

%\newpage

\section{Generalized Ricci solitons}
We recall the following \cite{GFS}
\begin{dfn}
A compact ACYT manifold with closed torsion is called
 steady generalized gradient Ricci  soliton if there exists a smooth function $f$ such that
 \begin{equation}\label{gein1}
Ric^g_{ij}=\frac14T^2_{ij}-\LC_i\LC_j f, \qquad \delta T_{ij}=-df_sT_{sij}, \qquad dT=0.
\end{equation}
\end{dfn}
The function $f$ is known as the potential function of the  steady generalized gradient Ricci  soliton.

 In terms of the torsion connection \eqref{gein1} can be written in the form (see \cite{IS,IS1})
\begin{equation}\label{gein4}
Ric_{ij}=-\sb_i\sb_j f, \qquad \d T_{ij}=-df_sT_{sij}, \qquad dT=0.
\end{equation}
It is known by  \cite[Proposition~8.14]{GFS} that any compact pluriclosed CYT space (N=0) is automatically a steady generalized gradient Ricci soliton.   

This result is extended to a 6-dimensional ACYT space with  closed torsion form in \cite[Theorem~6.5]{IS1} and to a general  ACYT space with closed torsion in \cite{KenStr} . %The proof of the next result follow very similarly to the proof of \cite[Theorem~6.5]{IS1} and we omit it.
We show that similar conclusions hold for a compact closed ACYT  manifold of dimension 8.

Let $dT=0$. Denote by $W$ the vector field dual to one  of the 1-forms $\theta\pm \tn-df$.
\begin{thrm}\label{inf}
 Let $(M,g,J,\Psi)$ be an 8-dimensional  compact ACYT manifold %, i.e. a 6-dimensional $SU(3)$ manifold satisfying \eqref{cycon}
 with closed torsion, $dT=0$. The next two conditions are equivalent.
 \begin{itemize}
\item[a).]  $(M,g,J,\Psi)$   is a steady generalized gradient Ricci soliton.
\vspace{-2.4mm}
\item[b).]  Any one of the vector fields $V=\theta-df,\quad W=\theta\pm\tn-df$  is $\sb$-parallel, $\sb V=\sb W=0.$
\end{itemize}
The $\sb-$parallel vector field $V$  determines $d\theta$,  $W$ determines $d\theta\pm d\tn$ and each of these vector fields  preserves   the metric $g$ and the torsion 3-form
\begin{gather}\label{vsu3}
dV=d\theta=V\lrcorner T, \quad dW=d\theta\pm d\tn=W\lrcorner T\quad {\cal{L}}_Vg={\cal{L}}_V T={\cal{L}}_{W}g={\cal{L}}_{W} T=0, \quad and\\\label{lief1} 
g( {\cal{L}}_VJ)X,Y)=( d\tn\lrcorner \ps)(X,JY)=-\delta N(X,JY)+(J\theta\lrcorner N)(X,Y)\\\label{lief2} %(d\tn\lrcorner\sp)(X,Y).
g( {\cal{L}}_{W}J)X,Y)=((d\tn\pm d\theta)\lrcorner \ps)(X,JY)=-\delta N(X,JY)+(J\theta\lrcorner N)(X,Y)\pm(d\theta\lrcorner\ps(X,JY).
\end{gather}
The vector fields $JV,JW$ are  $\sb-$parallel,  hence Killing, $d(JV)=JV\lrcorner T, d(JW)=JW\lrcorner T$, preserves the torsion, ${\cal{L}}_{JV}T={\cal{L}}_{JW}T=0$ and 
\begin{equation}\label{lijf}
\begin{split}
g(({\cal {L}}_{JV}J)X,Y)=(d\tn\lrcorner\ps)(X,Y)+(V\lrcorner N)(X,Y);\\%=\delta N(X,Y)-(df\lrcorner N)(X,Y);\\
g(({\cal {L}}_{JW}J)X,Y)=((d\tn\pm d\theta)\lrcorner\ps)(X,Y)+(W\lrcorner N)(X,Y).%\\=\delta N(X,Y)-(df\lrcorner N)(X,Y)\pm(d\theta\lrcorner\ps(X,Y).
\end{split}
\end{equation}
The norm of the torsion is constant along the vectors $V,W,JV,JW$, $$\sb_V||T||^2=\sb_{JV}||T||^2=\sb_{W}||T||^2=\sb_{JW}||T||^2=0.$$
If any of the vectors $W$ vanishes at one point then $W=0$ and the space is K\"ahler Calabi-Yau.
\end{thrm}
\begin{proof}
Since $dT=0$ we have $\sb\tn=0$ by Proposition~\ref{nparclos} and Proposition~\ref{npar}.

The proof of the  equivalences of a) and b) follow from \eqref{ricdtsu} and \eqref{gein4} very similarly as in  the proof of \cite[Theorem~6.5]{IS1} and we omit the details.

Since $\sb W=0$ we have  $d\theta\pm d\tn=dW=d^{\sb}W+W\lrcorner T=W\lrcorner T$ and similarly for $d\theta$ and $V$.

Further,  $\sb JV=J\sb V =0$ and  $( \mathbb{L}_Vg)_{ij}=\LC_iV_j+\LC_jV_i=\sb_iV_j+\sb_jV_i=0$ since $\sb V=0$. Hence $V$ and $JV$  are Killing vector fields and $d(JV)=d^{\sb}(JV)+JV\lrcorner T=JV\lrcorner T$. The same  for  $W,JW$.

We calculate from the definition of the Lie derivative  (see e.g. \cite{KN}) using $\sb W=0$ that
\begin{multline}\label{lieF}
( {\cal{L}}_{W}F)(X,Y)=WF(X,Y)-F([W,X],Y)-F(X,[W,Y])=WF(X,Y)\\-F(\sb_WX,Y)-F(X,\sb_WY)+T(W,X,e_a)F(e_a,Y)+T(W,Y,e_a)F(X,e_a)\\=(\sb_W F)(X,Y)+(d\theta\pm d\tn)(X,JY)-(d\theta\pm d\tn)(Y,JX)=-((d\nu\pm d\theta)\lrcorner \ps)(X,JY),
\end{multline}
where we apply the already proved identity $d\theta\pm d\tn=W\lrcorner T$ to get the third equality, use $\sb F=0$ and  Proposition~\ref{delparN} to achieve the last identity.

We are going to take into account the next  general identities 
\[
\begin{split}
g({\cal{L}}_ZJ)X,Y)=- ({\cal{L}}_Zg)(JX,Y)- ({\cal{L}}_ZF)(X,Y);\\
g(({\cal {L}}_{JZ}J)X,Y)+g(({\cal {L}}_{Z}J)X,JY)=N(Z,X,Y).
\end{split}
\]
Use the first  and \eqref{codN} to obtain \eqref{lief1}, \eqref{lief2}.  Apply the second,  \eqref{lief1} and \eqref{lief2} to get  \eqref{lijf}.

The Cartan formula together with \eqref{vsu3} yields
\[{\cal{L}}_VT=d(V\lrcorner T)+V\lrcorner dT=d^2\theta=0, \quad {\cal{L}}_{JV}T=d(JV\lrcorner T)+JV\lrcorner dT=d^2(JV)=0.\]

Similarly as the proof of \eqref{lieF} using $\sb V=\sb JV=0$, we calculate
\[
\begin{split}
0=({\cal{L}}_VT)(X,Y,Z)=(\sb_VT)(X,Y,Z)+\sigma^T(V,X,Y,Z);\\
0=({\cal{L}}_{JV}T)(X,Y,Z)=(\sb_{JV}T)(X,Y,Z)+\sigma^T(JV,X,Y,Z),
\end{split}
\]
Hence, $\sb_V||T||^2=-\sigma^T(V,e_p,e_q,e_r)T(e_p,e_q,e_r)=0$ applying \eqref{sigmatt}. Similary one gets $ \sb_{JV}||T||^2=0$.%-\sigma^T(JV,e_p,e_q,e_r)T(e_p,e_q,e_r)=0$, where we used the identity $\sigma^T_{ijkl}T_{ijk}=0$ from \cite{IS}. 

Since $W\lrcorner T=d(\theta\pm \tn)$ the prove above is applicable also to $W$.

For the last claim, if $W_{|_p}=0$ for a point $p\in M$ then $W=0$ since $\sb W=0$. Then we have $\theta\pm \tn=df$ and the claim follows from Theorem~\ref{stexact}.
\end{proof}

%We used  \eqref{gein4} and the maximum principle to get the equivalences between a) and g).
%If $Scal^g=0$ we have $0=T=d^cF^++\frac14N$ and we have a Calab-Yau space.
\begin{rmrk}\label{su3su4} Compare Theorem~\ref{inf} with \cite[Theorem~6.5]{IS1} to see the difference between strong ACYT 6-manifold and strong ACYT 8-manifold. The main differences are consequences of the fact that in dimension six $d\theta\in\frak{su}(3)$ while in dimension eight it might happened $d\theta\not\in\frak{su}(4)$ and even $(d\theta\pm d\tn)\not\in\frak{su}(4)$ because of Example~\ref{exdtit}.
\end{rmrk}
However, if the any one of the two forms $(d\theta\pm d\tn)$ is an (1,1) form with respect to $J$, for example if the exterior derivative of the Nijenhuis 1-form is an (1,1)-form with respect to $J$, we have
\begin{thrm}\label{infN01}
Let $(M,g,J,\Psi)$ be an 8-dimensional  compact  ACYT manifold  with closed torsion, $dT=0$ and any of the 2-forms  $d\theta\pm d\tn$ is $J$-invariant, $d\theta\pm d\tn=J(d\theta\pm d\tn)$.

\noindent The $\sb-$parallel vector field $W$  determines $d\theta\pm d\tn$,  preserves the   $SU(4)$ structure  and the torsion 3-form,
$$%\begin{equation*}%\label{vsu3}
dW=d\theta\pm d\tn=W\lrcorner T\in\frak{su}(4), \quad {\cal{L}}_Wg={\cal{L}}_WF={\cal{L}}_WJ={\cal{L}}_W\ps={\cal{L}}_W\sp={\cal{L}}_WT=0.
$$%\end{equation*}

The vector field $JW$ is  $\sb-$parallel,  Killing, $d(JW)=JW\lrcorner T$,  preserves the torsion, $${\cal{L}}_{JW}T=0\quad and \quad {\cal{L}}_{JW}J=W\lrcorner N=V\lrcorner N.$$

The distribution $D^{\pm}=span\{W,JW\}$ is integrable, the norm of the torsion is constant along $D^{\pm}$.% $\sb_{W}||T||^2=\sb_{JV^{\pm}}||T||^2=0$.

If any of the vector fields $W$ vanishes at one point then $W=0$ and the space is K\"ahler Calabi-Yau.
\end{thrm}
\begin{proof}
The condition $d\theta\pm d\tn=J(d\theta\pm d\tn)$ yields $(d\theta\pm d\tn)\lrcorner\ps=0$. %Proposition~\ref{nparclos} and  Proposition~\ref{delparN} imply $d\theta\in\frak{su}(4)$. 
In view of Theorem~\ref{inf}, \eqref{lief2} and \eqref{lijf},  it remains  to show that $D^{\pm}$ is integrable and $W$ preserves the $SU(4)$ structure.%, i.e. $  [W,JW]\in D^{\pm},\quad {\cal{L}}_{W}\ps={\cal{L}}_{W}\sp=0$.

Since $\sb W=\sb JW=0$ and $(d\theta\pm \tn)=W\lrcorner T\in\frak{su}(4)$, we  have
\[g([W,JW], X)=-T(W,JW,X)=T(W,W,JX)=0.\]
Further $ \mathbb{L}_{W}\ph=0$ since $ \mathbb{L}_{W}F=0$ and $\ph=\frac12F\wedge F$. 
Similarly to the proof of \eqref{lieF} we obtain
\begin{equation}\label{lieps}
\begin{split}
({\cal{L}}_{W}\ph)_{ijkl}=(d\theta\pm \tn)_{is}\ph_{sjkl}+(d\theta\pm \tn)_{js}\ph_{skli}+(d\theta\pm \tn)_{ks}\ph_{slij}+(d\theta\pm \tn)_{ls}\ph_{sijk}=0;\\
({\cal{L}}_{W}\ps)_{ijkl}=(d\theta\pm \tn)_{is}\ps_{sjkl}+(d\theta\pm \tn)_{js}\ps_{skli}+(d\theta\pm \tn)_{ks}\ps_{slij}+(d\theta\pm \tn)_{ls}\ps_{sijk};
%( {\cal{L}}_V\sp)_{ijkl}=d\theta_{is}\sp_{sjkl}+d\theta_{js}\sp_{skli}+d\theta_{ks}\sp_{slij}+d\theta_{ls}\sp_{sijk};
\end{split}
\end{equation}
%We calculate from \eqref{lieps} using the identities \eqref{idensu} that
%\begin{equation}\label{lieps1}
%\begin{split}
%( \mathbb{L}_V\ps)_{ijkl}\ps_{ijkl}=4d\theta_{is}\ps_{sjkl}\ps_{ijkl}=96d\theta_{is}\delta _{si}=0=( \mathbb{L}_V\sp)_{ijkl}\sp_{ijkl};\\
%( \mathbb{L}_V\ps)_{ijkl}\sp_{ijkl}=4d\theta_{is}\ps_{sjkk}\sp_{ijkl}=96d\theta_{is}F_{si}=0=-( \mathbb{L}_V\sp)_{ijkl}\ps_{ijkl}
%\end{split}
%\end{equation}
%since $d\theta\in\frak{su}(4)$ due to Corollary~\ref{closTsu} which is used to conclude the last equality in \eqref{lieps1}.
Consider the $Spin(7)$ structure $\Omega=\ph+\ps$. Since  $(d\theta\pm \tn)\in\frak{su}(4)$ we conclude $(d\theta\pm \tn)\in\frak{su}(4)\subset\frak{spin}(7)$. Proposition~\ref{spin7} yields
\[
\begin{split}
0=(d\theta\pm \tn)_{is}\ph_{sjkl}+(d\theta\pm \tn)_{js}\ph_{skli}+(d\theta\pm \tn)_{ks}\ph_{slij}+(d\theta\pm \tn)_{ls}\ph_{sijk}+\\
(d\theta\pm \tn)_{is}\ps_{sjkl}+(d\theta\pm \tn)_{js}\ps_{skli}+(d\theta\pm \tn)_{ks}\ps_{slij}+(d\theta\pm \tn)_{ls}\ps_{sijk}\\=(d\theta\pm \tn)_{is}\ps_{sjkl}+(d\theta\pm \tn)_{js}\ps_{skli}+(d\theta\pm \tn)_{ks}\ps_{slij}+(d\theta\pm \tn)_{ls}\ps_{sijk},
%\\
%0=d\theta_{is}\ph_{sjkl}+d\theta_{js}\ph_{skli}+d\theta_{ks}\ph_{slij}+d\theta_{ls}\ph_{sijk}+
%d\theta_{is}\ps_{sjkl}+d\theta_{js}\ps_{skli}+d\theta_{ks}\ps_{slij}+d\theta_{ls}\ps_{sijk}\\=d\theta_{is}\ps_{sjkl}+d\theta_{js}\ps_{skli}+d\theta_{ks}\ps_{slij}+d\theta_{ls}\ps_{sijk},
\end{split}
\]
where we apply the first identity of \eqref{lieps} to achieve the last equality. Hence, ${\cal{L}}_{W}\ps=0={\cal{L}}_{W}\sp$.

The last claim follows from Theorem~\ref{inf}.
\end{proof}
\begin{cor}\label{infN02}
Let $(M,g,J,\Psi)$ be an 8-dimensional  compact  ACYT manifold  with closed torsion, $dT=0$ and the Nijenhuis 1-form satisfies  $d\tn=Jd\tn$.

 The $\sb-$parallel vector field $V$  determines $d\theta$,  preserves the   $SU(4)$ structure and the torsion 3-form,
 
\centerline{$%\begin{equation*}%\label{vsu3}
dV=d\theta=V\lrcorner T\in\frak{su}(4), \quad {\cal{L}}_Vg={\cal{L}}_VF={\cal{L}}_VJ={\cal{L}}_V\ps={\cal{L}}_V\sp={\cal{L}}_VT=0.
$}%\end{equation*}

The vector field $JV$ is also $\sb-$parallel,  Killing, $d(JV)=JV\lrcorner T$,  preserves the torsion, 

\centerline{${\cal{L}}_{JV}T=0\quad  and \quad {\cal{L}}_{JV}J=V\lrcorner N.$}

The distribution $D=span\{V,JV\}$ is integrable and the norm of the torsion is constant along $D$.%, $\sb_V||T||^2=\sb_{JV}||T||^2=0$.
\end{cor}
\begin{proof}
The condition $d\tn=Jd\tn$ yields $d\tn\lrcorner\ps=0$. Proposition~\ref{nparclos} and  Proposition~\ref{delparN} imply $d\theta\in\frak{su}(4)$. The proof follows from the proof of Theorem~\ref{infN01} by replacing $W$ with $V$.
\end{proof}

We note that the $SU(4)$ structures $(F^2,\ps=F^3-F^1)$ and $(F^3,\ps=F^1-F^2)$ of Example~\ref{cstr} 
are non-complex examples supporting Theorem~\ref{infN01}.

In the 8-dimensional complex case $N=\tn=0$, Proposition~\ref{delparN} yields that $d\theta\in\frak{su}(4)$ and we obtain from Theorem~\ref{infN01} and Corollary~\ref{infN02} the following result similar to \cite[Theorem~6.5]{IS1}.
\begin{cor}\label{infN0}
Let $(M,g,J,\Psi)$ be a compact  CYT 8-manifold  with closed torsion ( pluriclosed CYT).

The $\sb-$parallel vector field $V$  determines $d\theta$,  preserves the   $SU(4)$ structure and the torsion 3-form,
$$%\begin{equation*}%\label{vsu3}
dV=d\theta=V\lrcorner T\in\frak{su}(4), \quad {\cal{L}}_Vg={\cal{L}}_VF={\cal{L}}_VJ={\cal{L}}_V\ps={\cal{L}}_V\sp={\cal{L}}_VT=0.
$$%\end{equation*}
\indent The vector field $JV$ is  $\sb-$parallel,   Killing, $d(JV)=JV\lrcorner T$, holomorphic  and preserves the torsion.%, ${\cal{L}}_{JV}J={\cal{L}}_{JV}T=0$.

The distribution $D=span\{V,JV\}$ is integrable and the norm of the torsion is constant along $D$.% $\sb_V||T||^2=\sb_{JV}||T||^2=0$.

If  the vector field $V$ vanishes at one point then $V=0$ and the space is K\"ahler Calabi-Yau.
\end{cor}
We remark that Corollary~\ref{infN0} is proved for pluriclosed CYT spaces in all dimensions in \cite{ABS} (se also \cite{KenStr}) except the statement that $V$ preserves the  $SU(4)$ structure.

%Basic examples of compact Hermitian non-K\"ahler steady generalized gradient Ricci solitons are compact Lie groups endowed with biinvariant metric and the left-invariant Samelson's complex structure. The Strominger-Bismut connection is the left-invariant flat Cartan connection with harmonic torsion 3-form of constant norm and the function $f$ has to be a constant.
%\end{rmrk}

\section{AHKT 8-manifold}
We consider an almost hyperhermitian 8-manifold or a manifold of dimension 8 with an $Sp(2)$ structure admitting linear connection preserving the $Sp(2)$  with torsion 3 form and call it AHKT manifold since if the $Sp(2)$ structure is integrable it is known as  a HKT manifold. 

We  investigate necessary and sufficient conditions an almost hyperhermitian manifold to be AHKT. 

Let us start with an $Sp(2)$ structure $(g,J_1,J_2,J_3)$ with  the  corresponding K\"ahler 2-forms $F^a,  a=1,2,3$ and Nijenhuis tensors $N^1,N^2,N^3$, respectively.
An obvious condition is the coincidence of the unique connections preserving each of the almost complex structure. It follows from \eqref{cy2} that the three Nijenhuis tensors have to be  a 3-form and one needs to have \cite{IvP}
\[dF^1(J_1,J_1,J_1)-N^1=dF^2(J_2,J_2,J_2)-N^2=dF^3(J_3,J_3,J_3)-N^3.
\]
It is well known  that an almost hyperhermitian 4n-dimensional manifold is hyperhermitian if and only if two of the three almost complex structures are integrable. This fact is based on the next formula between the Nijenhuis tensors $N^1,N^2,N^3$ \cite[Lemma~3.2]{AlMar}
\begin{equation}\label{nijk}
\begin{split}
2N^1(X,Y,Z)=N^2(X,Y,Z)+N^2(J_3X,Y,J_3Z)+N^2(J_3X,J_3Y,Z)+N^2(X,J_3Y,J_3Z)\\
+N^3(X,Y,Z)+N^3(J_2X,Y,J_2Z)+N^3(J_2X,J_2Y,Z)+N^3(X,J_2Y,J_2Z).
\end{split}
\end{equation}
Therefore, if two of the Nijenhuis tensors are 3-forms then the third one is also a 3-form.

On an almost hyperhermitian 4n dimensional manifold $(M,g,I_1,J_2,J_3)$ with totally skew-symmetric Nijenhuis tensors we have three $SU(2n)$ structures $(F^a,\psi)$ defined according to \eqref{spsu}.

 Here and everywhere in this section $(a,b,c)$ is a cyclic permutation of $\{1,2,3\}$.
 
 We also have the following 1-forms
\begin{equation}\label{thijk}
\begin{split}
 \theta^a(X)=\frac12dF^a(X,J_ae_s,e_s),%\quad=(F^a\lrcorner dF^a)(X),\quad J_a\theta^a=-dF^a\lrcorner\Phi^a\\ 
% \tn^a=\frac14N^a\lrcorner\psi=\frac14N^a\lrcorner(\Phi^b-\Phi^c),\\
\quad \tn^a(X)=\frac14N^a(J_bX,J_be_s,e_s)=-\frac14N^a(J_cX,J_ce_s,e_s).
\end{split}
\end{equation}
\begin{prop}
On an almost hyperhermitian 4n dimensional manifold $(M,g,J_1,J_2,J_3)$ with totally skew-symmetric Nijenhuis tensors we have the relations
\begin{equation}\label{tauijk}
\tn^a+\tn^b+\tn^c=0
\end{equation}
\end{prop}
\begin{proof}
The formula \eqref{tauijk} follows directly from \eqref{nijk} and \eqref{thijk}.
\end{proof}
Note that in dimension eight the definitions \eqref{thetN} and \eqref{thijk} completely agree. 

Indeed, we have according to \eqref{thetN} that 
$$\tn^a_i=\frac1{24}N^a_{jkl}\ps_{jkli}=\frac1{24}N^a_{jkl}(\ph^b-\ph^c)_{jkli}=\frac14N^a(J_be_i,J_be_s,e_s).$$
\begin{thrm}\label{spthm1}
Let $(M,g,J_1,J_2,J_3)$ be an almost hyperhermitian 8-dimensional manifold. The next conditions are equivalent:
\begin{itemize}
\item[1.)] $(M,g,J_1,J_2,J_3)$ is an AHKT manifold;
\vspace{-2.4mm}
\item[2.)] The Nijenhuis tensors $N^a$ are totally skew-symmetric and the next conditions hold 
\begin{equation}\label{cyconsp}
\delta\Phi^b-\delta\Phi^c=J_a\tn^a\wedge F^a%-N^a
-\theta^a\lrcorner(\Phi^b-\Phi^c)=J_a\tn^a\wedge F^a+J_b\theta^a\wedge F^b-J_c\theta^a\wedge F^c, \quad \tn^a=\theta^b-\theta^c
\end{equation}
  \end{itemize}
On an AHKT 8-manifold the Nijenhuis tensors and the torsion are given by
\begin{equation}\label{torsp}
\begin{split}
N^a=J_b\tn^a\wedge F^b-J_c\tn^a\wedge F^c;\\
T=-\delta\ph^a+J_a\theta^a\wedge F^a+J_b\tn^a\wedge F^b-J_c\tn^a\wedge F^c.%\\
%=-\delta\ph^b+J_b\theta^b\wedge F^b+J_c\tn^b\wedge F^c-J_a\tn^b\wedge F^a\\=-\delta\ph^c+J_c\theta^c\wedge F^c+J_a\tn^c\wedge F^a-J_b\tn^c\wedge F^b.
\end{split}
\end{equation}
\end{thrm}
\begin{proof}
On an almost hyperhermitian 8-manifold  we have three $SU(4)$ structures $(F^a, \psi=\Phi^b-\Phi^c)$.

Suppose 1.) holds. Then, clearly each of the three  $SU(4)$ structures $(F^a, \psi=\Phi^b-\Phi^c)$ is ACYT with the same torsion 3-form $T$. Then \eqref{cyconsp} and  \eqref{torsp} follow from Theorem~\ref{cythm1}  applied to each $SU(4)$ structure taking into account that the three torsion 3-forms for each ACYT structure coincide. We also used that $\theta^a=T\lrcorner\Phi^a$ and $\tn^a=T\lrcorner\psi=T\lrcorner(\Phi^b-\Phi^c)=\theta^b-\theta^c$. Clearly, $\tn^a+\tn^b+\tn^c=0$.
%On AHKT 8-manifold we have
%\begin{equation}\label{ahkt}
%\begin{split}
%\ph^i=\frac12F^i\wedge F^i, \quad \ps^i=\ph^j-\ph^k,\qquad 
%\theta^i_d=\frac16T_{abc}\ph^i_{abcd},\\ \tn^i_d=\frac16T_{abc}\psi_{abcd}=\theta^j_d-\theta^k_d \Longrightarrow \tn^1+\tn^2+\tn^3=0.
%\end{split}
%\end{equation}

For the converse, the first condition in \eqref{cyconsp} yields that we may 
apply Theorem~\ref{cythm1} %and Theorem~\ref{cythm1} we 
to conclude that each of the  $SU(4)$ structure $(F^a, \psi=\Phi^b-\Phi^c)$ is an ACYT structure, i.e.
we obtain 3-connections with torsion 3-forms $T^a$  preserving  $(F^a,\psi)$.  Moreover, Theorem~\ref{cythm1} tells us that the Nijehuis 3-forms are given by $N^a=-\tn^a\lrcorner\psi=-\tn^a\lrcorner(\Phi^b-\Phi^c)$.  We calculate the first identity of \eqref{torsp} hods and 
%obtain using \eqref{ahkt}
\begin{equation}\label{ahkt1}
\begin{split}
%N^a=J_b\tn^a\wedge F^b-J_c\tn^a\wedge F^c;\\
\delta\psi=\delta\ph^b-\delta\ph^c=J_a\tn^a\wedge F^a+J_b\theta^a\wedge F^b-J_c\theta^a\wedge F^c;\\
T^a=-\delta\ph^a+J_a\theta^a\wedge F^a+J_b\tn^a\wedge F^b-J_c\tn^a\wedge F^c.
\end{split}
\end{equation}
We show that the three torsion 3-forms are all equal and therefore we have an AHKT structure.

Indeed, we obtain from \eqref{ahkt1}
\begin{multline}
T^a-T^b=-\delta\ph^a+\delta\ph^b \\+J_a\theta^a\wedge F^a+J_b\tn^a\wedge F^b-J_c\tn^a\wedge F^c-J_b\theta^b\wedge F^b-J_c\tn^b\wedge F^c+J_a\tn^b\wedge F^a\\
=-J_c\tn^c\wedge F^c-J_a\theta^c\wedge F^a+J_b\theta^c\wedge F^b\\+J_a\theta^a\wedge F^a+J_b\tn^a\wedge F^b-J_c\tn^a\wedge F^c-J_b\theta^b\wedge F^b-J_c\tn^b\wedge F^c+J_a\tn^b\wedge F^a\\=J_a\Big(\theta^a-\theta^c+\tn^b \Big)\wedge F^a+J_b\Big(\theta^c-\theta^b+\tn^a \Big)\wedge F^b-J_c\Big(\tn^a+\tn^b+\tn^c\Big)\wedge F^c=0,
\end{multline}
where we apply \eqref{tauijk} and the second conditions in \eqref{cyconsp} to achieve the last equality.

Hence, $T^a=T^b=T^c$ meaning that the three torsion connections coincide.  %The proof is  completed .
\end{proof}
\begin{cor}\label{corhktdt}
An AHKT 8-manifold  is an HKT if and only if the three Lee forms are all equal.
\end{cor}
Note that Corollary~\ref{corhktdt} agrees with \cite[Lemma~4.4]{MCS}.

%We have also  the following 1-forms
%\[\begin{split}
%\theta^i=T^i\lrcorner\Phi^i, \qquad \theta^{i,j}=T^i\lrcorner\Phi^j;\\\tau^i=T^i\lrcorner\Psi^{i+}=T^i\lrcorner\Big(\Phi^j-\Phi^k\Big)=\theta^{i,j}-\theta^{i,k},\\\tau^{i,j}=T^i\lrcorner\Psi^{j+}=T^i\lrcorner\Big(\Phi^k-\Phi^j\Big)=\theta^{i,k}-\theta^{i,j}
%\end{split} \]

%\end{exam}

 \begin{exam}\label{cstr}
 Consider the product of two copies of the primary Hopf surface, the compact Lie goup $G=S^1\times S^3\times S^3\times S^1=U(1)\times SU(2)\times SU(2)\times U(1)$ with the biinvariant metric and define a left-invariant almost hyperhermitian structure  $F^1,F^2,F^3$, which is parallel with respect to the flat left-invariant  Cartan  connection $\sb$ with closed  torsion $T=-[.,.]$.

More precisely,  the group $G$ has Lie algebra $g=\mathbb R\oplus su(2)\oplus su(2)\oplus\mathbb R$ and structure equations   
\begin{equation}\label{struc}de_0=de_7=0, \quad de_1=e_{23},\quad de_2=e_{31},\quad de_3=e_{12},\quad de_4=e_{56},\quad de_5=e_{64},\quad de_6=e_{45}.
\end{equation}
We define the left invariant $Sp(2)$ structure with the following three 2-forms
\[F^1=-e_{01}-e_{23}-e_{45}-e_{67},\quad F^2=-e_{04}+e_{15}-e_{26}+e_{37},\quad F^3=-e_{05}-e_{14}-e_{27}-e_{36}.\]
Clearly, this $Sp(2)$ structure is parallel with respect to the flat left-invariant Cartan connection $\sb$ with closed and $\sb$-parallel  torsion 3-form $T=e_{123}+e_{456}$ and  closed $\sb$-parallel  Lee forms 
\[
\begin{split}
\theta^1=e_7-e_0, \quad  \theta^2=\theta^3=0.
%J_1\theta=6+1,\quad J_2\theta=-3+4\quad J_3\theta=2+5
\end{split}
\]
A simple standard calculations  applying \eqref{struc}  yield $N^1=0$ and 
\[
\begin{split}
N^2=%J_3\tn^2\wedge F^3-J_1\tn^2\wedge F^1=
-J_3\theta^1\wedge F^3+J_1\theta^1\wedge F^1\not=0,%\quad \sb N^2=0;
\quad 
N^3=%J_1\tn^3\wedge F^1-J_2\tn^3\wedge F^2=
J_1\theta^1\wedge F^1-J_2\theta^1\wedge F^2\not=0, \quad \sb N^2= \sb N^3=0.
\end{split}
\] 
Thus, we obtain an example of a compact AHKT 8-manifold with closed and $\sb$-parallel torsion with two balanced almost complex strucutres and the third one is not balanced but complex structure.
\end{exam}
We have
\begin{thrm}\label{balcom}
Any compact AHKT 8-manifold with a balanced complex structure is a compact balanced HKT 8-manifold and therefore is an $SL(2,H)$ and has holomorphically trivial canonical bundle.
\end{thrm}
\begin{proof}
Suppose $\theta^1=N^1=0=\tn^1$. Then $\theta^2=\theta^3=\tn^2=-\tn^3$ according to Theorem~\ref{spthm1}. Apply \eqref{ricdtsu} consequently to the first  and to the second ACYT structures to get 
$$0=Scal+\frac13||T||^2=3\delta\theta^2+2||\theta^2||^2+2||\theta^2||^2.$$
Since $M$ is compact we integrate the last equality to get $\theta^2=\theta^3=0$. Hence, $N^2=N^3=0$ and we obtain a compact balanced HKT 8-manifold. The rest of the Theorem follows from the result of Verbitsky, \cite{Ver1,Ver2} (c.f. also \cite[Theorem~2.2]{IvP}).
\end{proof}

\begin{prop}\label{ahktdelT}
On an AHKT 8-manifold we have the formulas
\begin{equation}\label{ahktdeltaT}
\begin{split}
\d T=d^{\sb}\theta^a+\frac12\Big[J_b(d^{\sb}\theta^b)-J_b(d^{\sb}\theta^c)+J_c(d^{\sb}\theta^c)-J_c(d^{\sb}\theta^b)-\d(J_b\theta^c)F^b-\d(J_c\theta^b)F^c\Big];
%\\
%\delta T(X,Y)=d^{\sb}\theta^a(X,Y)-\frac12\Big[\delta(J_b\theta^c)F^b(X,Y)+\delta(J_c\theta^b)F^c(X,Y)\Big] \\+\frac12\Big[d^{\sb}\theta^b(J_bX,J_bY)-d^{\sb}\theta^b(J_cX,J_cY)-d^{\sb}\theta^c(J_bX,J_bY)+d^{\sb}\theta^c(J_cX,J_cY)  \Big];
\\d^{\sb}\theta^b\lrcorner F^a+d^{\sb}\theta^c\lrcorner F^a=-\delta(J_a\theta^b)-\delta(J_a\theta^c)=0.
%\\d^{\sb}\theta^b_{ij}F^a{ij}+d^{\sb}\theta^c{ij}F^a_{ij}=-\delta(J_a\theta^b)-\delta(J_a\theta^c)=0. 
\end{split}
\end{equation}
\end{prop}
\begin{proof}
We calculate $\delta T$ using  Theorem~\ref{thdeltaT},  the second identity of \eqref{cyconsp} and the first identity of \eqref{ahkt1}.

First we express $\delta^{\sb}N$ with the help of \eqref{ahkt1} applying \eqref{cyconsp} and the identity $\delta(J_a\theta^a)=0$ as follows
\begin{equation}\label{deltaN}
\begin{split}
\d^{\sb}N^a=J_b(d^{\sb}\theta^b)-J_b(d^{\sb}\theta^c)+J_c(d^{\sb}\theta^c)-J_c(d^{\sb}\theta^b)-\d(J_b\theta^c)F^b-\d(J_c\theta^b)F^c\Big.
%\\
%-\delta^{\sb}N^a(X,Y)=(\sb_{e_i}N^a)(e_i,X,Y)=\delta(J_b\theta^c)F^b(X,Y)+\delta(J_c\theta^b)F^c(X,Y)\\-d^{\sb}\theta^b(J_bX,J_bY)+d^{\sb}\theta^b(J_cX,J_cY)+d^{\sb}\theta^c(J_bX,J_bY)-d^{\sb}\theta^c(J_cX,J_cY) .
\end{split}
\end{equation}
Substitute \eqref{deltaN} into the third identity of \eqref{new3} to get the first formula in \eqref{ahktdeltaT}.

For the second one, we use \eqref{new1} and the identities $d^{\sb}\theta^b\lrcorner F^a=-\delta(J_a\theta^b)$ to get
\[
0=-\d T\lrcorner F^b=\delta T(e_i,J_be_i)=2\delta(J_b\theta^a)+4\delta(J_b\theta^c)-2\delta(J_b\theta^c)=2\delta(J_b\theta^a)+2\delta(J_b\theta^c).%\\
%0=\delta T(e_i,J_ce_i)=2\delta(J_c\theta^a)+4\delta(J_c\theta^b)-2\delta(J_c\theta^b)=2\delta(J_c\theta^a)+2\delta(J_c\theta^b)\\
%\delta(J_b\theta^a)=-\delta(J_b\theta^c)\\
%\delta(J_c\theta^b)=-\delta(J_c\theta^a)\\
%\delta(J_a\theta^c)=-\delta(J_a\theta^b).
\]
The proof is completed.
\end{proof}

We have from Theorem~\ref{spthm1} and \eqref{ricdtsu}
\begin{thrm}\label{sp2N}
On an AHKT 8-manifold the next conditions are equivalent:
\begin{itemize}
\item[a)] The Nijenhuis tensors are $\sb$-parallel;
\vspace{-2.4mm}
\item[b)] The Lie forms satisfy the conditions
$\sb\theta^a=\sb\theta^b=\sb\theta^c. $
\vspace{-2.4mm}
\item[c)]  The next equalities  hold  $dT_{ipqr}\ph^a_{jpqr}=dT_{ipqr}\ph^b_{jpqr}=dT_{ipqr}\ph^c_{jpqr}.$
\end{itemize}
In particular, if two of the three Nijenhuis tensors are $\sb$-parallel then the third one is also $\sb$-parallel.
\end{thrm}
Taking into account Proposition~\ref{delparN}, \eqref{cyconsp}, Theorem~\ref{ahktdelT} and Theorem~\ref{sp2N}, we obtain
\begin{cor}\label{sp2Nc}
On an AHKT 8-manifold with $\sb$-parallel Nijenhuis tesors we have 
\[\delta\theta^a=\delta\theta^b, \quad g(\theta^a-\theta^b,\theta^c)=g(\theta^a-\theta^b,J_c\theta^c)=0, \quad d\theta^b\lrcorner F^a=d\theta^c\lrcorner F^a, \quad \delta(J_b\theta^a)=\delta(J_c\theta^a)=0.\]
\begin{equation}\label{ahktdeltaTn}
\begin{split}
\delta T(X,Y)=d^{\sb}\theta^a(X,Y)=d^{\sb}\theta^b(X,Y)=d^{\sb}\theta^c(X,Y).
\end{split}
\end{equation}
\end{cor}
Proposition~\ref{nparclos} implies
\begin{cor}
On an AHKT 8-manifold with closed torsion, $dT=0$, the Nijenhuis tensors $N^a, a=1,2,3$ are $\sb$-parallel, $\sb N^a=0$.
\end{cor}
\subsection{Proof of Theorem~\ref{ahktmain} and Corollary~\ref{closThkt}}
\begin{proof}
We have to show that a) follows from b). Suppose $F^a\lrcorner dT=0, a=1,2,3$. Then \eqref{ricdtsu} and \eqref{dtsu} yield $Ric=-\sb\theta^a=Ric=-\sb\theta^b=-\sb\theta^c$ and Theorem~\ref{sp2N} implies that the Nijenhuis tensors are $\sb$-parallel, $\sb N^a=0$. Apply Theorem~\ref{closT} to complete the proof of Theorem~\ref{ahktmain}.
%\end{proof}
%\begin{thrm}\label{closThkt}
%On a compact AHKT 8-manifold the next conditions are equivalent:
%\begin{itemize}
%%\item[a).] The torsion is closed, $dT=0$.
%\item[b).] There is $a=1,2,3$ such that the Nijenhuis tensor $N^a$ is $\sb$-parallel and $F^a\lrcorner dT=0$.
%\item[c).] There is $a=1,2,3$ such that the Nijenhuis tensor $N^a$ is $\sb$-parallel  and $Ric=-\sb\theta^a$.
%\end{itemize}
%\end{thrm}

In the HKT case the three Lee forms are all equal by Corollary~\ref{corhktdt},  the three two forms $F^a\lrcorner dT$ coincide due to Theorem~\ref{sp2N}  and Corollary~\ref{closThkt} follows from Theorem~\ref{ahktmain}.
\end{proof}
On a strong AHKT 8 manifold we have
\begin{prop}\label{spstrong}
Let $(M,g,J_a, a=1,2,3)$ be a compact strong $(dT=0$) AHKT 8-manifold and let $V^a=\theta^a-df^a, \quad W^a=\theta^a-df^a, a=1,2,3$ be the $\sb$-parallel vector fields from Theorem~\ref{inf} associated to each of the ACYT structure $(J_a,g), a=1,2,3$, respectively.

Then the potential functions $f^a,f^b,f^c$ differ by a constant.
\end{prop}
\begin{proof}
Since these vector fields are $\sb$-parallel and $\tn^a=\theta^b-\theta^c$, we have 
\[0=-g(\sb_{e_i}W^a,e_i)=\delta\theta^a+\d\theta^b-\d\theta^c-\Delta f^a,\quad \delta\theta^a+\d\theta^b-\d\theta^c-\Delta f^a=\delta\theta^b+\d\theta^c-\d\theta^a-\Delta f^a.
\]
We know from Proposition~\ref{nparclos} that $dT=0$ implies  $\sb N=0$. Now Proposition~\ref{sp2Nc} yields $\Delta f^a=\Delta f^b$ which implies $f^a-f^b=const$ since $M$ is compact.
\end{proof}

\subsection{Sp(2)-instanton}
The $Sp(2)$-instanton means that the curvature of $\sb$ is $SU(4)$ instanton simultaneously with respect to the three $SU(4)$ structures $(F^a,\ps^a=\ph^b-\ph^c)$ for the three cyclic permutations of $\{a,b,c\}$. 
The properties of the $Sp(2)$ instanton can be obtained from the properties of the three $SU(4)$ instantons described in the Section~\ref{su4instanton}.

For example, \eqref{bi24} combined with the second identity of \eqref{cyconsp} implies 
\begin{cor}
If the torsion connection of an AHKT 8-manifold is an $Sp(2)$-instanton then
\[\sb_iR_{ijkl}=\theta^a_iR_{ijkl}=\theta^b_iR_{ijkl}=\theta^c_iR_{ijkl}.\]
\end{cor}
Another example comes from Proposition~\ref{jriccom1} taking into account the second identity of \eqref{cyconsp} and Proposition~\ref{npar}, namely we obtain
\begin{cor}\label{jriccomh1}
Let $(M,g,J_a, a=1,2,3)$ be a compact AHKT 8-manifold with $Sp(2)$ instanton torsion connection. 

If the Nijenhuis tensors are $\sb$-coclosed, $\delta^{\sb}N^a=0$ then the Nijenhuis tensors and the Lee forms are $\sb$-parallel, $\sb N^a=\sb\theta^a=0$.

In this case  the Ricci tensor is  symmetric and $J_a,J_b,J_c$-invariant given by $$Ric_{ij}=\frac16\sigma^T_{ipqr}\ph^a_{jpqr}=\frac16\sigma^T_{ipqr}\ph^b_{jpqr}=\frac16\sigma^T_{ipqr}\ph^c_{jpqr}.$$
\end{cor}
The last claim we mention here comes from Theorem~\ref{stexact} combined with the second identity of \eqref{cyconsp}.
\begin{cor}\label{stexacth}
Let $(M,g,J_a, a=1,2,3)$ be a compact AHKT 8-manifold  with closed torsion, $dT=0$.

If the any of the following six two forms $\theta^a+\theta^b-\theta^c, a\not=b\not=c\not=a$ is exact,  then $(M,g,J_a, a=1,2,3)$ is a hyper-K\"ahler 8-manifold, in particular $\sb=\LC$ is an $Sp(2)$-instanton.
%Lee form plus or minus the Nijenhuis 1-form is exact, $dT=0,\quad \theta\pm \tn=df, \quad f-smooth,$  then $(M,g,J_a)$ is K\"ahler Calabi-Yau 8-manifold.
\end{cor}

\subsection{Orthogonal semi-integrable AHKT 8-manifold}
We recall the recent definition by Gentili-Moroianu \cite{GeM} that an almost hyperhermitian structure  is called semi-integrable if it contains one integrable and one semi-K\"ahler (balanced) structure. If these structures are orthogonal we have the notion of orthogonal semi-integrable almost hyperhermitian structure. 

By an orthogonal semi-integrable AHKT 8-manifold we mean an AHKT 8-manifold $(M,g, J_a,J_b,J_c)$ with $J_a$ complex and $J_b$ balanced (semi-K\"ahler). For example, the AHKT manifold described in Example~\ref{cstr} is a compact  orthogonal semi-integrable AHKT 8 manifold.
 
 In this section $(a,b,c)$ is a  fixed cyclic permutation of $\{1,2,3\}$.

Let $(M,g, J_a,J_b,J_c)$  be a semi-integrable AHKT 8-manifold and $J_a$ be integrable, $N^a=0$ and $J_b$ be balanced, $\theta^b=0$. By Theorem~\ref{spthm1} we get 
\begin{equation}\label{semahkt}
0=\tn^a=\theta^b-\theta^c=-\theta^c; \quad \tn^b=\theta^c-\theta^a=-\theta^a, \quad \tn^c=\theta^a-\theta^b=\theta^a.
\end{equation}
\begin{prop}
Let $(M,g, J_a,J_b,J_c)$  be an orthogonal semi-integrable AHKT 8-manifold with $J_a$ integrable and $J_b$ balanced.
Then
\begin{equation}\label{semah}
\begin{split}
\theta^c=\theta^b=0,\qquad \delta\theta^a=0.
\end{split}
\end{equation}
If, in addition, the Nijenhuis tensor $N^b$ is co-closed, $\d N^b=0$, then
the Ricci tensor is $ J_a,J_b$ and $J_c$ invariant,   $d\theta^a\in\frak{sp}(2)$ and $\sb\theta^a$ is symmetric and $J_a$ invariant,
\begin{equation}\label{ahktric}
\begin{split}
Ric(J_bX,J_bY)=Ric(X,Y)=Ric(J_cX,J_cY)=Ric(J_aX,J_aY)=Ric(Y,X) ; \\d\theta^a\in\frak{sp}(2),\quad (\sb_X\theta^a)(Y)=(\sb_{J_aX}\theta^a)(J_aY)=(\sb_Y\theta^a)(X).
\end{split}
\end{equation}
\end{prop}
\begin{proof}
Apply  \eqref{ricdtsu} to the first two structures $(J_a,g)$ and $(J_b,g)$ together with \eqref{semahkt} yields
\[Scal+\frac13||T||^2=3\delta\theta^a+2||\theta^a||^2=2||\tn^b||^2=2||\theta^a||^2 \quad \Longrightarrow \delta\theta^a=0.
\]
%ence $\delta\theta^a=0$.
Applying Theorem~\ref{thdeltaT} to the three structures $(J_a,g), a=1,2,3$ taking into account \eqref{semahkt} to get
\begin{equation}\label{sehkt1}
\begin{split}
\delta T=d^{\sb}\theta^a=\frac12\delta^{\sb}N^b=\delta N^b=\frac12\delta^{\sb}N^c=\delta N^c.\\
d\nu^c-J_cd\nu^c=d\theta^a-J_cd\theta^a=\frac12[d^{\sb}\nu^c-J_cd^{\sb}\tn^c]=\frac12[d^{\sb}\theta^a-J_cd^{\sb}\theta^a]\\
d\theta^a-J_bd\theta^a=\frac12[d^{\sb}\theta^a-J_bd^{\sb}\theta^a],\quad d\theta^a=J_ad\theta^a.
\end{split}
\end{equation}
The first identities in \eqref{sehkt1} together with  the properties of $\d N^b$ and  $\d N^c$ imply
\[d^{\sb}\theta^a=-J_bd^{\sb}\theta^a=-J_cd^{\sb}\theta^a=J_ad^{\sb}\theta^a.\]
Taking into account \eqref{nnid} we obtain from \eqref{sehkt1}
\begin{equation}\label{semin1}
\theta^a\lrcorner T=J_a(\theta^a\lrcorner T),\quad J_b(\theta^a\lrcorner T)=J_c(\theta^a\lrcorner T),\quad d^{\sb}\theta^a=J_b(\theta^a\lrcorner T)-\theta^a\lrcorner T.
\end{equation}

If $\d N^b=0$ then  \eqref{sehkt1} imply $\d T=\d^{\sb}N^b=\d^{\sb}N^c=d^{\sb}\theta^a=0$. Now,  \eqref{semin1} and \eqref{nnid} yield
\begin{equation}\label{semin2}
\theta^a\lrcorner T=J_a(\theta^a\lrcorner T)= J_b(\theta^a\lrcorner T)=J_c(\theta^a\lrcorner T),\quad d\theta^a=\theta^a\lrcorner T\in\frak{sp}(2).
\end{equation}
Taking into account $\theta^b=\theta^c=0$ and Proposition~\ref{ricsymj}, c) to get  the first equality    in \eqref{ahktric}. The second one is a consequence of the first and c) of Proposition~\ref{ricsymj}.
\end{proof}
\begin{thrm}\label{semin}
Let $(M,g, J_a,J_b,J_c)$  be a compact orthogonal semi-integrable AHKT 8-manifold with $J_a$ integrable and $J_b$ balanced with co-closed  Nijenhuis tensor, $\theta^b=\d N^b=0$.

Then the Lee form $\theta^a$ and the Nijenhuis tensors $N^b,N^c$ are $\sb$-parallel,
$\sb N^b=\sb N^c=\sb\theta^a=0.$
\end{thrm}
\begin{proof}
Since $\d\theta^a=0$ and $\sb\theta^a$ is symmetric, the Ricci identity for $\sb$ reads
\begin{multline}\label{riciden}
-\frac12\Delta||\theta^a||^2=\frac12\sb_i\sb_i||\theta^a||^2=\theta^a_j\sb_i\sb_i\theta^a_j+||\sb\theta^a||^2=\theta^a_j\sb_i\sb_j\theta^a_i+||\sb\theta^a||^2\\
=-\theta^a_j\sb_j\d\theta^a-R_{ijis}\theta^a_j\theta^a_s-\theta^a_jT_{ijs}\sb_s\theta^a_i+||\sb\theta||^2=Ric_{js}\theta^a_j\theta^a_s+||\sb\theta||^2.
\end{multline}
Applying \eqref{ricdtsu} and \eqref{dh} we have 
\begin{equation}\label{ricahkt}
\begin{split}
Ric_{ij}\theta^a_i\theta^a_j=\frac1{12}dT_{ipqr}\ph^a_{jpqr}\theta^a_i\theta^a_j-\frac12\sb_{\theta^a}||\theta^a||^2\\=\frac1{12}d^{\sb}T_{ipqr}\ph^a_{jpqr}\theta^a_i\theta^a_j+\frac16\sigma^T_{ipqr}\ph^a_{jpqr}\theta^a_i\theta^a_j-\frac12\sb_{\theta^a}||\theta^a||^2.
\end{split}
\end{equation}
Since $\theta^a\lrcorner T\in \frak{sp}(2)$,  we use  \eqref{nabdt0} together with \eqref{part} and \eqref{ntor1} to get 
$%\begin{equation*}%\label{ricsigma}
\sigma^T_{ipqr}\ph^a_{jpqr}\theta^a_i\theta^a_j=0.
$ %\end{equation*}
We obtain from \eqref{dtsu} and \eqref{kapa1} applied to $F^a$ 
\[dT_{ipqr}\ph^a_{jpqr}\theta^a_i\theta^a_j=3d^{\sb}T(\theta^a,J_a\theta^a,e_i,J_ae_i)=6d^{\sb}(J_a\theta^a)(\theta^a,J_a\theta^a)-6(\sb_{e_i}T)(J_ae_i,\theta^a,J_a\theta^a).%=-4\ka$ppa^b(X,J_aY),
\]
Substitute the last three equalities into  \eqref{ricahkt} to derive 
\begin{equation}\label{ricahkt1}
\begin{split}
Ric_{ij}\theta^a_i\theta^a_j=-\frac12(\sb_{e_i}T)(J_ae_i,\theta^a,J_a\theta^a);
\end{split}
\end{equation}
where we used    $d^{\sb}(J_a\theta^a)(\theta^a,J_a\theta^a)=\sb_{\theta^a}||\theta^a||^2$ since $\sb\theta^a$ is $J_a$ invariant. Now, \eqref{riciden} takes the form
\begin{equation}\label{ricdelahkt}
-\frac12\Delta||\theta^a||^2=-\frac12(\sb_{e_i}T)(J_ae_i,\theta^a,J_a\theta^a)+||\sb\theta^a||^2.
\end{equation}
We apply \eqref{patrecom} and \eqref{kapa1} to get
\begin{multline*}
-\frac12\Delta||\theta^a||^2=\kappa^a(\theta^a,J\theta^a)+\frac12\sb_{\theta^a}||\theta^a||^2+||\sb\theta^a||^2=\frac12(\sb_{e_i}T)(J_ae_i,\theta^a,J_a\theta^a)+\sb_{\theta^a}||\theta^a||^2+||\sb\theta^a||^2.
%\\=\kappa(\theta,J\theta)-(\sb_{J\theta}J\theta)\theta+||\sb\theta||^2=\kappa(\theta,J\theta)+\frac12\sb_{\theta}||\theta||^2+||\sb\theta||^2,
\end{multline*}
Adding the latter to \eqref{ricdelahkt} yields
\[-\Delta||\theta^a||^2-\sb_{\theta^a}||\theta^a||^2=||\sb\theta^a||^2\ge 0.
\]
Apply the strong maximum principle (e.g. \cite{YB,GFS}) to conclude $\sb\theta=0$ and $\sb N^b=\sb N^c=d^{\sb}\theta^a=0$.
\end{proof}
\begin{cor}\label{seminc}
Let $(M,g, J_a,J_b,J_c)$  be a compact orthogonal semi-integrable AHKT 8-manifold with $J_a$ integrable and $J_b$ balanced with co-closed  Nijenhuis tensor, $\theta^b=\d N^b=0$. 

Then the torsion is closed if and only if the torsion connection is Ricci flat. In this case the Nijenhuis tensors are $\sb$-parallel and  the norm of the torsion is a constant.

%In this case the 4-dimensional  distribution $D^4=\{\theta^a,J_a\theta^a,J_b\theta^a,J_c\theta^a\}$ is integrable since the $\sb$-prallel vector fields 
%$\theta,J_a\theta,J_b\theta,J_c\theta$ are Killing and preserve the $sp(2)$-structure, 
\end{cor}
\begin{proof}
Since $\sb\theta^a=\theta^b=\theta^c=\sb N^b=\sb N^c=0$ we obtain from \eqref{ricdtsu} that $Ric=0$ if and only if $F^a\lrcorner dT=F^b\lrcorner dT=F^c\lrcorner dT=0$. The  claim follows from Theorem~\ref{nparclos}, Theorem~\ref{closT} and Theorem~\ref{closTt}.
\end{proof}
Suppose $dT=0$. Then  $(M,g, J_a,J_b,J_c)$ is a compact steady generalized gradient Ricci soliton with constant potential function, $df=0$ because $Ric=0$ and \eqref{gein4} . Set $V=\theta^a$, we get from Theorem~\ref{inf} and Theorem~\ref{infN01}
\begin{prop}
Let $(M,g, J_a,J_b,J_c)$  be a compact orthogonal semi-integrable AHKT 8-manifold with $J_a$ integrable and $J_b$ balanced with co-closed  Nijenhuis tensor and closed torsion, $\theta^b=\d N^b=dT=0$. 

The vector field $V$ preserves the $Sp(2)$ structure and the next identities hold true
\begin{itemize}
\item[a).] $V,J_aV,J_bV,J_cV$ are non-zero Killing and $\sb$-parallel;
\vspace{-2.4mm}
\item[b).] $ {\cal{L}}_VJ_a= {\cal{L}}_VJ_b= {\cal{L}}_VJ_c=0$; 
\vspace{-2.4mm}
\item[c).]$ {\cal{L}}_{J_aV}J_a= {\cal{L}}_{J_bV}J_b= {\cal{L}}_{J_cV}J_c=0$;
\vspace{-2.4mm}
\item[d).] $V\lrcorner dV=J_aV\lrcorner dV=J_bV\lrcorner dV=J_cV\lrcorner dV=0$; 
\vspace{-2.4mm}
\item[e).]$V\lrcorner d(J_aV)=J_aV\lrcorner d(J_aV)=0$ and $d(J_aV)$ is an (1,1) form with respect to $J_a$;
\vspace{-2.4mm}
\item[f).]$V\lrcorner d(J_bV)=J_bV\lrcorner d(J_bV)=0$ and $d(J_bV)$ is an (1,1) form with respect to $J_b$;
\vspace{-2.4mm}
\item[g).]$V\lrcorner d(J_cV)=J_cV\lrcorner d(J_cV)=0$ and $d(J_cV)$ is an (1,1) form with respect to $J_c$;
\end{itemize}
\end{prop}
\begin{proof}
The first one, a) is straightforward.

The point b) follows from \eqref{lief1} since $\d N^a=\d N^b=\d N^c=0$ and $J_b\theta^a\lrcorner N^b=J_c\theta^a\lrcorner N^c=0$ because, for example, $N^b=-\tn^b\lrcorner(\ph^c-\ph^a)=\theta^a\lrcorner(\ph^c-\ph^a)$ and $J_b\theta^a\lrcorner N^b=N^b(J_b\theta^a,\theta^a,.)=N^b(\theta^a,\theta^a,J_b.)=0$;

To show  c), we use\eqref{lijf} and the conditions $\d N^a=\d N^b=\d N^c=df=0$;

The point d) follows from $dV=d\theta^a=V\lrcorner T\in\frak{sp}(2)$. 

For example  $J_bV\lrcorner dV=T(V,J_bV,.)=-T(V,V,J_b.)=0$. 

To show   f)  we use that $V\lrcorner N^b=J_bV\lrcorner N^b=0$,  $dJ_bV=J_bV\lrcorner T$ and $dV\in\frak{sp}(2)$ to get applying  \eqref{12}
\[
\begin{split}
T(V,J_bX,J_bY)=T(V,X,Y)=T^{+b}(V,X,Y)+\frac14N^b(V,X,Y)=T^{+b}(V,X,Y); \\T(J_bV,X,Y)=T^{+b}(J_bV,X,Y)+\frac14N^b(J_bV,X,Y)=T^{+b}(J_bV,X,Y);\\
T(V,X,Y)=T(V,J_bX,J_bY)+T(J_bV,J_bX,Y)+T(J_bV,X,J_bY) \Rightarrow T(J_bV,J_bX,Y)=-T(J_bV,X,J_bY).
\end{split}
\]
%where we used \eqref{12}. 
Hence, $d(J_bV)(X,Y)=d(J_bV)(J_bX,J_bY)$. Similarly we obtain the validity of e) and f).
\end{proof}
%If $D$ is not integrable the $d\theta=0$

%$J_1VJ_2V,X=J_1VJ_3V,J_1X=J_2VJ_3V,J_2X$
%By d)  the vector fields $\{V, J_aV, J_bV, J_cV \}$ span a rank-4 distribution $F\subset Ker(d\theta^a)$.We consider the orthogonal  decomposition \[TM=F\oplus F^{\perp}\]

%%%%%%%%%%%%%%%%%%%%

\end{document}